\documentclass{amsart}
\usepackage[svgnames]{xcolor}
\usepackage[unicode,colorlinks=true,linktocpage=true,citecolor=ForestGreen,linkcolor=MediumOrchid]{hyperref}
\usepackage{tikz-cd}
\usepackage{mathrsfs}
\usepackage{amssymb}
\usepackage{mathtools}
\usepackage{pinlabel}
\usepackage{enumitem}
\usepackage{booktabs}
\usepackage{microtype}
\usepackage[capitalise,noabbrev]{cleveref}

\usepackage[only,mapsfromchar]{stmaryrd}
\newcommand{\ract}{\rightarrow\mapsfromchar}

\newtheorem{theorem}{Theorem}[section]
\newtheorem{proposition}[theorem]{Proposition}
\newtheorem{corollary}[theorem]{Corollary}
\newtheorem{lemma}[theorem]{Lemma}
\theoremstyle{definition}
\newtheorem{definition}[theorem]{Definition}
\newtheorem{example}[theorem]{Example}
\newtheorem{remark}[theorem]{Remark}
\newtheorem{warning}[theorem]{Warning}
\newtheorem{notation}[theorem]{Notation}
\newtheorem{construction}[theorem]{Construction}

\newcommand{\actrm}{\mathrm{act}}
\newcommand{\intrm}{\mathrm{int}}
\newcommand{\elrm}{\mathrm{el}}
\newcommand{\oprm}{\mathrm{op}}

\newcommand{\finset}{\mathbf{FinSet}}
\newcommand{\wproperad}{\mathbf{WPpd}}

\newcommand{\modular}{\mathbf{ModOp}}
\newcommand{\set}{\mathbf{Set}}
\newcommand{\spaces}{\mathbf{Space}}
\newcommand{\cat}{\mathbf{Cat}}
\newcommand{\gcat}{\mathbf{U}}
\newcommand{\gcatemb}{\gcat_{\mathrm{int}}}
\newcommand{\gcatembo}{\gcat^{\mathrm{int}}_{/\opshf}}
\newcommand{\gcatact}{\gcat_{\actrm}}
\newcommand{\gcatcyc}{\gcat_{\mathrm{cyc}}}
\newcommand{\gcatnought}{\gcat_{0}}
\newcommand{\egcat}{\widetilde{\mathbf{U}}}
\newcommand{\opshf}{\mathfrak{o}}
\newcommand{\egcato}{\egcat_{/\opshf}}
\newcommand{\gcato}{\gcat_{/\opshf}}
\newcommand{\pgcat}{\mathbf{G}}
\newcommand{\pgcatsc}{\mathbf{G}_{\mathrm{sc}}}
\DeclareMathOperator{\id}{id}

\DeclareMathOperator{\dom}{dom}
\DeclareMathOperator{\nbhd}{nb}

\newcommand{\emb}{\operatorname{Emb}}
\newcommand{\exedge}{{\updownarrow}}
\newcommand{\oldnew}{\mathfrak{N}}
\newcommand{\newold}{\mathfrak{O}}
\newcommand{\inp}{\operatorname{in}}
\newcommand{\out}{\operatorname{out}}
\newcommand{\ssub}{\operatorname{sSb}}
\newcommand{\rat}{\rightarrowtail}

\newcommand{\uamalg}[1]{\underset{#1}{{\amalg}}} 
\newcommand{\mydef}[1]{\textbf{#1}}

\newcommand{\CC}{\mathbf{C}}
\newcommand{\DD}{\mathbf{D}}

\newcommand{\fstarop}{\mathbf{F}_\ast^\oprm}
\newcommand{\finsetstarop}{\finset_\ast^\oprm}
\newcommand{\lr}[1]{\langle #1 \rangle}

\newcommand{\PP}{\mathbf{P}}

\tikzset{
  act /.tip = >|
}

\DeclareFontFamily{U} {MnSymbolC}{}
\DeclareFontShape{U}{MnSymbolC}{m}{n}{
  <-6> MnSymbolC5
  <6-7> MnSymbolC6
  <7-8> MnSymbolC7
  <8-9> MnSymbolC8
  <9-10> MnSymbolC9
  <10-12> MnSymbolC10
  <12-> MnSymbolC12}{}
\DeclareFontShape{U}{MnSymbolC}{b}{n}{
  <-6> MnSymbolC-Bold5
  <6-7> MnSymbolC-Bold6
  <7-8> MnSymbolC-Bold7
  <8-9> MnSymbolC-Bold8
  <9-10> MnSymbolC-Bold9
  <10-12> MnSymbolC-Bold10
  <12-> MnSymbolC-Bold12}{}
\DeclareSymbolFont{MnSyC} {U} {MnSymbolC}{m}{n}
\DeclareMathSymbol{\medstar}{\mathbin}{MnSyC}{130}

\title{Categories of graphs for operadic structures}

\author{Philip Hackney}
\address{Department of Mathematics, University of Louisiana at Lafayette, USA}
\email{philip@phck.net} 
\urladdr{http://phck.net}
\thanks{This work was supported by a grant from the Simons Foundation (\#850849)}



\begin{document}
\begin{abstract}
We recall several categories of graphs which are useful for describing homotopy-coherent versions of generalized operads (e.g.\ cyclic operads, modular operads, properads, and so on), and give new, uniform definitions for their morphisms.
This allows for straightforward comparisons, and we use this to show that certain free-forgetful adjunctions between categories of generalized operads can be realized at the level of presheaves. 
This includes adjunctions between operads and cyclic operads, between dioperads and augmented cyclic operads, and between wheeled properads and modular operads.
\end{abstract}

\maketitle

\setcounter{tocdepth}{1}
\tableofcontents

\section{Introduction}

\begin{figure}
\includegraphics[width=0.8\textwidth]{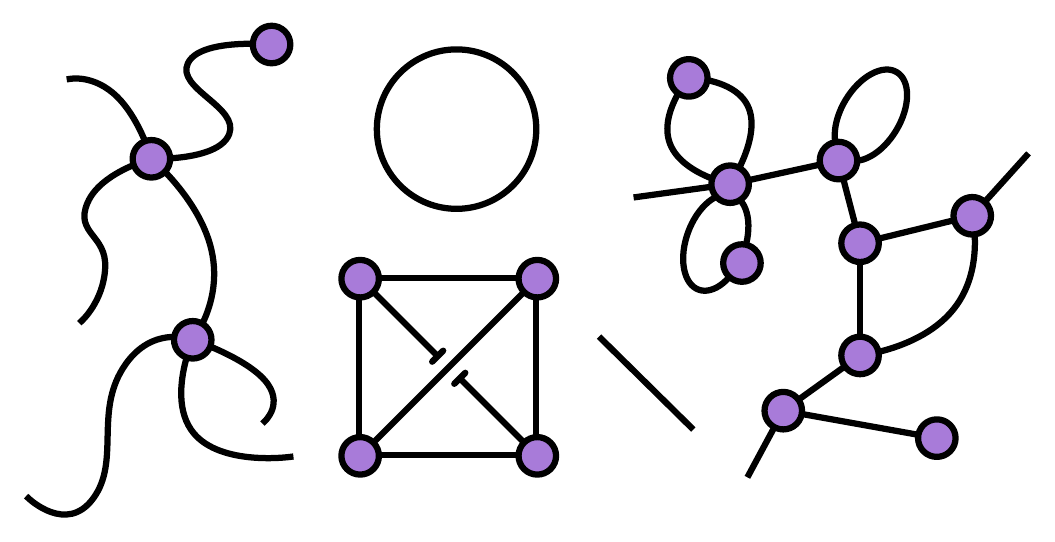}
\caption{A graph with loose ends}\label{fig loose ends}
\end{figure}

It is often convenient to encode weak categorical or algebraic structures as certain kinds of presheaves on a shape category that models the situation in question, just as we model the homotopy theory of spaces using simplicial sets.
For example, the simplicial indexing category $\mathbf{\Delta}$ is used in $(\infty,1)$-category theory \cite{Rezk:MHTHT,JoyalTierney:QCSS,Lurie:HTT} and Joyal's cell category $\mathbf{\Theta}_n$ is used for $(\infty,n)$-categories \cite{Rezk:CPWnC,Ara:HQCHRS}. 
For $(\infty,1)$-operads, the Moerdijk--Weiss dendroidal category $\mathbf{\Omega}$ \cite{MoerdijkWeiss:DS}, which is a category of rooted trees and an enlargement of $\mathbf{\Delta}$, underlies several of the many equivalent models \cite{CisinskiMoerdijk:DSMHO,CisinskiMoerdijk:DSSIO,CisinskiMoerdijk:DSSO,HeutsHinichMoerdijk,Barwick:FOCHO,ChuHaugsengHeuts:TMHTIO}.
The present paper is concerned with shape categories that have been used to model higher versions of certain important types of generalized operads.

The shapes we study will be certain graphs with loose ends (\cref{fig loose ends}), which might come with additional data (e.g.\ orientations of edges, planar structures) or restrictions (e.g.\ connected, acyclic).
Such graphs are well suited to describing the kinds of operations appearing in generalized operads via the mechanism of `graph substitution' or `graph insertion.'
Given two graphs $G$ and $H$ and a bijection between the germs of edges at a vertex of $G$ and the loose ends of $H$, we can form a new graph $G\{H\}$ where a small neighborhood of $v$ has been replaced by $H$.
The categories we focus on have graph substitution built in at a fundamental level, as is the case for several other useful categories involving such graphs (see, for instance, \cite{BataninMarkl:OCDDC,KaufmannWard:FC}).
However, in this paper we shift attention in our definitions away from graph substitution and towards \emph{embeddings} of graphs.

The two main graph categories we will discuss are related to modular operads \cite{GetzlerKapranov:MO} (also known as compact symmetric multicategories \cite{JoyalKock:FGNTCSM}), and to wheeled properads \cite{MarklMerkulovShadrin}.
The objects of our shape categories are connected graphs with loose ends, which are undirected in the first case, and directed in the second.

The graphical category $\gcat$ from \cite{HRY-mod1}, whose objects are undirected graphs, is intimately connected with modular operads \cite{GetzlerKapranov:MO,JoyalKock:FGNTCSM}, in that there is an associated nerve theorem (proved in \cite{HRY-mod2}, but see also \cite{Raynor:DLCSM}).
This theorem says that there is a fully-faithful inclusion of $\modular$ into the category of presheaves $\widehat{\gcat} \coloneqq \mathrm{Fun}(\gcat^\oprm, \set)$, whose essential image consists of those presheaves satisfying a Segal condition (\cref{def segal}).
This is the modular operad analogue of the dendroidal nerve theorem from \cite{MoerdijkWeiss:OIKCDS,Weber:F2FPRA}. 
One can also utilize the category $\gcat$ to give notions of $(\infty,1)$-modular operads, as in \cite{HRY-mod1,Strumila:GDHC,ChuHaugseng:HCASC}.
Our first main theorem is a new description of the graphical category $\gcat$. 
{
\renewcommand{\thetheorem}{1}
\begin{theorem}[\cref{old new equivalence} and \cref{thm extended graph cat}]\label{main theorem undirected}
Let $G$ and $G'$ be undirected connected graphs.
A graphical map $\varphi \colon G \to G'$ is the same thing as a pair consisting of an involutive function $\varphi_0\colon A_G \to A_{G'}$ and a function $\hat \varphi \colon \emb(G) \to \emb(G')$, so that this pair is compatible with respect to boundaries and the function $\hat \varphi$ preserves unions, vertex disjoint pairs of embeddings, and edges.
Composition of graphical maps is given by composition of pairs of functions.
\end{theorem}
}
This theorem holds uniformly for maps in both the graphical category $\gcat$ and the extended graphical category $\egcat$ from \cite{HRY-mod1}.
This has aesthetic appeal, as the previous definitions of these categories relied on separate ad-hoc conditions to prevent `collapse.'
A key insight of this paper is the correct notion of unions of embeddings that make this theorem hold, and this appears in \cref{section unions}.
We will state the precise conditions for the maps appearing in \cref{main theorem undirected} in \cref{def new graph map}.

One upshot of \cref{main theorem undirected} is that composition becomes easy, and another is that it makes more transparent the active-inert factorization system on $\gcat$ (see \cref{rmk factorization}). 

There is a parallel story for directed graphs and wheeled properads.
Wheeled properads are a variant of dioperad or polycategory which are capable of modeling both parallel processes and feedback of processes.
In \cite{HRYbook}, the wheeled properadic graphical category was introduced as a category whose morphisms are certain maps between the free wheeled properads generated by directed graphs.
As for $\gcat$, there is a nerve theorem for wheeled properads, allowing us to regard $\wproperad$ as a full subcategory of presheaves on the wheeled properadic graphical category.
The following is an analogue of \cref{main theorem undirected} for directed graphs.
{
\renewcommand{\thetheorem}{1'}
\begin{theorem}[\cref{thm oriented wheeled equiv} and \cref{thm oriented wheeled equiv extended}]\label{main theorem directed}
Let $G$ and $G'$ be directed connected graphs.
A map $\varphi \colon G \to G'$ in the wheeled properadic graphical category is the same thing as a pair consisting of a function $\bar\varphi_0\colon E_G \to E_{G'}$ and a function $\hat \varphi \colon \emb(G) \to \emb(G')$, so that this pair is compatible with respect to inputs/outputs and the function $\hat \varphi$ preserves unions, vertex disjoint pairs of embeddings, and edges.
Composition is given by composition of pairs of functions.
\end{theorem}
}
This gives a new, purely categorical/combinatorial definition of the wheeled properadic graphical category from \cite{HRYbook,HRYfactorizations} which does not rely on the category of wheeled properads.
We identify the wheeled properadic graphical category as a comma category $\gcato$ for a certain orientation $\gcat$-presheaf $\opshf$.

With this identification, the forgetful functor $\gcato \to \gcat$ becomes relatively easy to understand and analyze. 
We use this to establish that the adjunction $\wproperad \rightleftarrows \modular$ between the categories of wheeled properads and modular operads can be well understood at the presheaf level.

{
\renewcommand{\thetheorem}{2}
\begin{theorem}[\cref{prop restriction segal} and \cref{thm lke segal}]\label{main theorem segal}
Let $f\colon \gcato \to \gcat$ be the forgetful functor from the category of directed graphs to the category of undirected graphs.
At the level of presheaves, both left Kan extension
\[f_! \colon \widehat{\gcato} \to \widehat{\gcat}\] and restriction \[f^* \colon \widehat{\gcat} \to \widehat{\gcato}\] preserve Segal objects.
\end{theorem}
}
The first part of this theorem is somewhat unusual, and arises from a new, particularly simple formula for the left Kan extension when passing from directed to undirected graphs.
We only discovered this formula because of the relationship in the descriptions from \cref{main theorem undirected} and \cref{main theorem directed}.

We are also interested in several graph categories suitable for modeling other operadic structures: dioperads/symmetric polycategories \cite{Gan:KDD,Garner:PPDL}, properads/compact symmetric polycategories \cite{Vallette:KDP,Duncan:TQC}, and (augmented) cyclic operads/$\ast$-polycategories \cite{GetzlerKapranov:COCH,HinichVaintrob:COACD,Shulman:2Chu}. 
In order to streamline the paper, discussion of these categories is mostly deferred to \cref{sec tree cat} and \cref{sec properadic gcat}.
We also have a third appendix, \cref{sec extended}, which addresses the extended (oriented) graphical category, where nodeless loops (that is, vertex-free graphs whose geometric realization is a circle) are present.

In \cref{sec tree cat} we show how to use \cref{main theorem undirected} to give a description of categories of simply-connected undirected graphs in a very similar manner to the `complete morphisms' version of the unrooted tree category $\mathbf{\Xi}$ from \cite{HRY-cyclic}.
In \cref{sec properadic gcat}, we show how to identify the properadic graphical category as a subcategory of the wheeled properadic graphical category, and we show that the dioperadic graphical category is a full subcategory of the wheeled properadic graphical category.
Finally, in \cref{sec extended}, we show that the results of \cref{sec new graph maps} and \cref{sec directed} also hold when nodeless loops are considered.
The results in \cref{sec extended} are important, and arguably the `correct' versions of the theorems, but due to the inadequacies of the underlying definitions of graphs themselves they are best treated separately. 

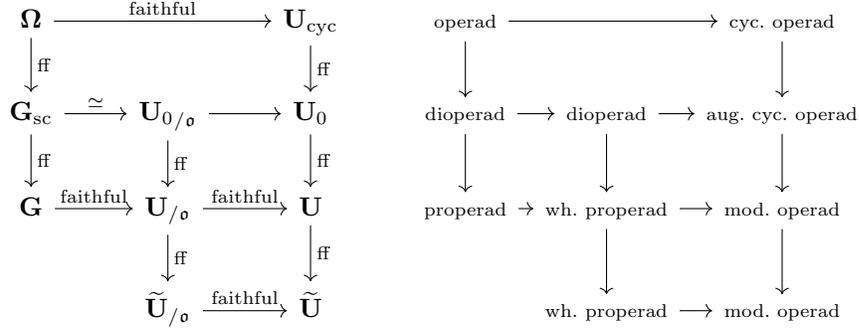
\begin{figure}[htb]
\[ \begin{tikzcd}
& \mathbf{\Omega} \ar[rr,"\text{faithful}"] \dar{\text{ff}} & &  \gcatcyc \dar{\text{ff}} 
& \text{\scriptsize operad} \ar[rr] \dar &[-1.8em] &[-1.8em]  \text{\scriptsize cyc.\ operad} \dar \\
& \pgcatsc \rar["\simeq"] \dar{\text{ff}} & {\gcatnought}_{/\opshf} \rar \dar{\text{ff}}& \gcatnought  \dar{\text{ff}} 
& \text{\scriptsize dioperad} \rar \dar & \text{\scriptsize dioperad} \rar \dar & \text{\scriptsize aug.\ cyc.\ operad}  \dar \\
& \pgcat  \rar{\text{faithful}}  & \gcato \rar{\text{faithful}} \dar{\text{ff}}  & \gcat \dar{\text{ff}} 
& \text{\scriptsize properad}  \rar  & \text{\scriptsize wh.\ properad} \rar \dar  & \text{\scriptsize mod.\ operad} \dar \\
& & \egcato \rar{\text{faithful}} & \egcat
& & \text{\scriptsize wh.\ properad} \rar & \text{\scriptsize mod.\ operad}
\end{tikzcd}
\]
\caption{Functors between graph categories}\label{fig graph cat functors}
\end{figure}

A diagram relating the graph categories from this paper appears in \cref{fig graph cat functors}. 
The categories in the top right are from \cref{sec tree cat}, the categories on the middle left are in \cref{sec properadic gcat}, and the two categories on the bottom are in \cref{sec extended}.

\subsection*{Further directions}
One motivation for this work is the hope that an alternative definition of the graph categories could help facilitate comparisons to other, related graph categories in the literature.
Of special interest are the operadic categories of graphs appearing in \cite{BataninMarkl:OCDDC,BataninMarkl:OCNEKD,BataninMarklObradovic:MMGROA}, which should be closely related to subcategories of active maps by \cite[Proposition 3.2]{Berger:MCO}.
We are also interested in the apparent differences with the graph categories arising from a twisted arrow construction in \cite{Burkin:TACOSC}.

Another motivation for this paper is to provide a blueprint for work on operadic structures based on disconnected graphs. 
Such operadic structures include props, wheeled props, and non-connected modular operads.
We expect that the definitions in this paper generalize to yield plausible graph categories for these three structures, once the appropriate notion of embedding is established. 
Moreover, the active-inert factorization systems on these categories, which are used to formulate the Segal condition, should arise in a natural and transparent way (see \cref{rmk factorization}).
This would all play a role in developments related to \cite{Raynor:BDMOGNT}, which concerns the circuit algebras of Bar-Natan and Dansco \cite{MR3623232}.
These circuit algebras are a generalization of the planar algebras of Jones \cite{JonesPlanarAlgebras}, and turn out to be the same thing as wheeled props \cite{MR4265709}.
In addition, it would be interesting to compare the kinds of infinity-props arising from this Segal condition with those appearing in \cite{HaugsengKock}.
	
Lastly, the reader will notice that we have almost entirely avoided any homotopical issues, but these should be addressed later.
We would like to know how the formulas for left Kan extensions in \cref{subsec lke segal} behave with respect to Quillen model structures for higher operadic structures. 
Can the formula from \cref{rem omega cyc} be leveraged to resolve \cite[Remark 6.8]{DrummondColeHackney:DKHTCO}, which seems to be needed for a conditional result of Walde \cite[Remark 5.0.20]{Walde:2SSIIO}?
It would also be interesting to know if \cref{main theorem segal} can be promoted to the $\infty$-categories of Segal objects in spaces from \cite{ChuHaugseng:HCASC}, and also to the enriched context following \cite{ChuHaugseng:EIO,GepnerHaugseng}.

\subsection*{Structure of the paper}
\Cref{sec prelim} mostly contains preliminary material which has previously appeared elsewhere.
\Cref{section unions} contains a definition of unions of embeddings, but only \cref{def a union,example order} are needed for the main narrative of the paper; the rest of the section is examining questions of existence of unions and may be skipped.
\Cref{sec new graph maps} introduces the new description of graphical maps at the beginning, and the remainder of the section is devoted to the proof of \cref{main theorem undirected}.

Directed graphs appear in \cref{sec directed}, and the new definition of the wheeled properadic graphical category appears as \cref{def oriented graphical cat,prop oriented graph cat}. 
A reader who is not already familiar with wheeled properads and the wheeled properadic graphical category may wish to stop reading this section after \cref{rem broad poset} (to skip the proof of \cref{main theorem directed}).

The final section of the main narrative of the paper is \cref{sec hypermoment}, which turns its attention to Segal presheaves which (via nerve theorems of the author, Robertson, and Yau) are the same thing as modular operads and wheeled properads, respectively.
\Cref{main theorem segal} is proved in \cref{subsec lke segal}.
\Cref{sec hypermoment} also deals with several other graph categories from the appendices, though relatively little knowledge is needed about them to follow this.
It is not necessary to read the appendices to understand most of this section.

The paper contains three appendices, which are largely independent and may be read or skipped based on the interest of the reader.
\Cref{sec tree cat} is about a characterization of maps in categories of undirected trees, and can be read at any point after \cref{sec new graph maps}.
\Cref{sec properadic gcat} exhibits the properadic graphical category as a particular non-full subcategory of $\gcato$ whose objects are acyclic directed graphs, and can be tackled after \cref{prop oriented graph cat}.
Finally, \cref{sec extended} is about the `extended' versions of graphical categories that include the nodeless loop as an object.
The first part can be read after \cref{sec new graph maps}, whereas \cref{subsec extended oriented} relies on \cref{subsec wheeled properads and oriented}.

\subsection*{Acknowledgments}
This work owes a tremendous amount to numerous conversations about related topics with friends, colleagues, and collaborators over the years.
Marcy Robertson, in particular, deserves special thanks for encouraging me to write this up and to pursue some of the avenues of inquiry within. 

\section{Preliminaries}\label{sec prelim}
The purpose of this section is to recall the most essential information about the graphical category $\gcat$ from \cite{HRY-mod1}.
The following definition of a combinatorial model for `graphs with loose ends' is due to Joyal and Kock, and appears in \cite[\S3]{JoyalKock:FGNTCSM} under the name \emph{Feynman graphs}.

\begin{definition}[Graphs]\label{def jk graphs}
Write $\mathscr{I}$ for the category freely generated by the graph \[\begin{tikzcd}[column sep=small] \mathtt{a} \arrow[loop left] & \lar \mathtt{d} \rar & \mathtt{v} \end{tikzcd}\]
subject to the relation that the nontrivial endomorphism of $\mathtt{a}$ squares to the identity.
\begin{itemize}
\item 
A \mydef{graph} is a functor $\mathscr{I} \to \finset$ so that the self-map of $\mathtt{a}$ goes to a fixpoint-free involution and the generating map $\mathtt{a} \leftarrow \mathtt{d}$ goes to a monomorphism.
\end{itemize}
A graph $G$ will be written as
\[\begin{tikzcd}[column sep=small] A_G \arrow[loop left,"\dagger"] & \lar[hook'] D_G \rar{t} & V_G \end{tikzcd}\]
and we usually behave as though the leftward arrow is a subset inclusion $D_G \subseteq A_G$.
\begin{itemize}
\item
If $v\in V_G$ is a vertex, then we write $\nbhd(v) \coloneqq t^{-1}(v) \subseteq D_G$ for the \mydef{neighborhood} of $v$.
	\item 
The \mydef{boundary} of a graph $G$ is the set $\eth(G) = A_G \setminus D_G$.
\item 
The set of \mydef{edges} 
\[
	E_G = \{ [a,a^\dagger] \mid a \in A_G\} \cong A_G / {\sim}
\]
is the set of $\dagger$-orbits, and has half the cardinality of $A_G$.
\item
An edge $[a,a^\dagger]$ is an \mydef{internal edge} if neither of $a,a^\dagger$ are elements of $\eth(G)$, otherwise it is a \mydef{boundary edge}.
\end{itemize}
\end{definition}

In \cref{sec extended} we will use a more general (but also more fiddly) version of graph (Definition 4.1 of \cite{HRY-mod1}), where the boundary is additional specified data.

\begin{example}\label{ex edge star}
We give two basic examples of graphs that we will use repeatedly (see Example 1.4 \cite{HRY-mod1}).
\begin{itemize}
\item We write $G = {\exedge}$ for the \mydef{edge}, which has exactly two arcs $A_G = \{ \sharp, \flat \}$, and $D_G = \varnothing = V_G$. Note that $\eth(G) = A_G$.
\item If $n\geq 0$, we write $G = \medstar_n$ for the \mydef{$n$-star}. This graph has $V_G$ a one-point set, $D_G = \{1, \dots, n \}$, and $A_G = \{ 1, 1^\dagger, \dots, n, n^\dagger \}$. Note that $\eth(G) = \{ 1^\dagger, \dots, n^\dagger \}$. The $5$-star is depicted on the left of \cref{fig contr corolla}.
\end{itemize}
\end{example}

From a graph $G$, one can obtain a topological space as a quotient space of \[ V_G \amalg \left( \coprod_{D_G} [0,{\textstyle\frac{2}{3}}) \right) \amalg \left( \coprod_{\eth(G)} (0,{\textstyle\frac{2}{3}}) \right) \]
by identifying, $x \sim 1-x$ where $x\in (\frac{1}{3},\frac{2}{3})$ is in the $a$-component and $1-x\in (\frac{1}{3},\frac{2}{3})$ is in the $a^\dagger$-component, and letting $t$ identify $0$ in the $d$-component with $t(d) \in V_G$.
As vertices may have arity two, in order to not lose any information one should consider this as a pair of spaces $\mathscr{X} = (X,V)$ with $V$ finite.
Notice that we can recover all of the data of $G$ from the homeomorphism class of $\mathscr{X}$ (see \S1 of \cite{HRY-mod1}).

The following definition of \'etale map appeared in \cite{JoyalKock:FGNTCSM}; these are the maps of graphs which preserve arity of vertices.

\begin{definition}\label{def etale embed undirected}
An \mydef{\'etale map} between graphs is a natural transformation of functors 
\[ \begin{tikzcd}
A_G\dar  \arrow[loop left,"\dagger"] & \lar[hook'] D_G\dar \ar[dr, phantom, "\lrcorner" very near start]  \rar & V_G\dar  \\
A_{G'} \arrow[loop left,"\dagger'"] & \lar[hook'] D_{G'} \rar & V_{G'}
\end{tikzcd} \]
so that the right hand square is a pullback.
An \mydef{embedding} between connected graphs is an \'etale map which is injective on vertices.
\end{definition}

`Connected' in the preceding definition means in the usual sense for an object in $\finset^{\mathscr{I}}$, namely $G$ is connected if and only if given any coproduct splitting $G \cong G_1 \amalg G_2$ precisely one of $G_1$ or $G_2$ is empty.
This is equivalent to $G$ being non-empty and every two elements have a path between them (see \cref{def paths and trees}), and is also equivalent to the associated topological space being connected.

\'Etale maps from the edge $\exedge$ classify arcs, as they are determined by where $\sharp$ is sent. 
\'Etale maps from stars weakly classify vertices: if $\nbhd(v)$ has cardinality $n$, then there are precisely $n!$ \'etale maps $\medstar_n \to G$ which send the unique vertex to $v$.
See also \cref{example embeddings}.

\begin{remark}\label{rem etale local homeo}
Each \'etale map $G\to G'$ induces a local homeomorphism $\mathscr{X} \to \mathscr{X}'$ of the corresponding topological graphs, and each embedding between connected graphs induces an injective local homeomorphism.
If we insist that we can recover the original \'etale map from the local homeomorphism, then this local homeomorphism is essentially unique up to deformation through a family of local homeomorphisms.
In contrast, there are two distinct (orientation-preserving) embeddings $(0,1) \amalg (0,1) \to (0,1)$ up to deformation through families of embeddings, while there is only one when considered up to deformation through families of local homeomorphisms.
This hints at why \'etale maps may be ill-suited for describing embeddings with disconnected domain.
\end{remark}

\begin{example}[Examples of embeddings]\label{example embeddings}
If $G$ is a graph, then every edge and every vertex determines an embedding.
\begin{itemize}
\item Suppose that $e = [a,a^\dagger]$ is an edge of $G$. 
We can define a graph, denoted $\exedge_e$, with $V = D = \varnothing$, and $A = \eth(\exedge_e) = \{a,a^\dagger\}$ together with the unique fixpoint-free involution.
Inclusion of arc sets yields a canonical embedding ${\exedge}_e \rightarrowtail G$.
\item If $f \colon H \rightarrowtail G$ is any embedding with $H$ isomorphic to $\exedge$ (which happens if and only if $V_H = \varnothing$), we say that $f$ \mydef{is an edge}.
Such an $f$ will be isomorphic to one of the inclusions ${\exedge}_e \rightarrowtail G$.
\item Suppose $v$ is a vertex of $G$. Define $\medstar_v$ to be the graph with $V_{\medstar_v} = \{ v \}$, $D_{\medstar_v} = \nbhd(v) = \{d_1, \dots, d_n \} \subseteq D_G$, $\eth(\medstar_v) = \{b_1, \dots, b_n\}$ (which are new, formally defined elements), and $d_i^\dagger \coloneqq b_i$.
Then $\medstar_v$ is isomorphic to $\medstar_n$, and there is an embedding $\iota_v \colon \medstar_v \rightarrowtail G$ which sends $d_i$ to $d_i \in D_G$ and $b_i$ to $d_i^\dagger \in A_G$.
\item It is not necessary for an embedding to be injective on arcs. As an example, take the quotient $G$ of $\medstar_5$ (see \cref{ex edge star}) by identifying the arcs $2 \sim 5^\dagger$ and $2^\dagger \sim 5$ but keeping the rest of the structure the same. 
This is depicted in \cref{fig contr corolla}.
The canonical map $\medstar_5 \rat G$ is an embedding but not a monomorphism.
Likewise, if $v$ is an arbitrary vertex of a graph $G$, then $\iota_v \colon \medstar_v \rat G$ is injective on arcs if and only if there are no loop edges at $v$.
\end{itemize}
This last point essentially indicates the only way an embedding can fail to be a monomorphism on arcs and edges --- by identifying pairs of boundary edges (\cref{mod1 lem 1.22}). 
Another example appears in \cref{fig no joins} below.
\end{example}

\begin{figure}[tb]
\labellist
\small\hair 2pt
 \pinlabel {$1$} [B] at 74 45 
 \pinlabel {$1^\dagger$} [B] at 96 45 
 \pinlabel {$2$} [B] at 74 65
 \pinlabel {$2^\dagger$} [B] at 90 84
 \pinlabel {$3$} [B] at 51 73
 \pinlabel {$3^\dagger$} [B] at 40 91
 \pinlabel {$4$} [B] at 36 62
 \pinlabel {$4^\dagger$} [B] at 14 62
 \pinlabel {$5$} [Bl] at 40 36
 \pinlabel {$5^\dagger$} [Bl] at 40 17
 \pinlabel {$1$} [B] at 187 45
 \pinlabel {$1^\dagger$} [B] at 209 45
 \pinlabel {$3$} [B] at 164 73
 \pinlabel {$3^\dagger$} [B] at 153 91
 \pinlabel {$4$} [B] at 149 62
 \pinlabel {$4^\dagger$} [B] at 127 62
 \pinlabel {$\bar 2$} [B] at 187 65 
 \pinlabel {$\bar 5$} [Bl] at 153 36
\endlabellist
\centering
\includegraphics[scale=0.8]{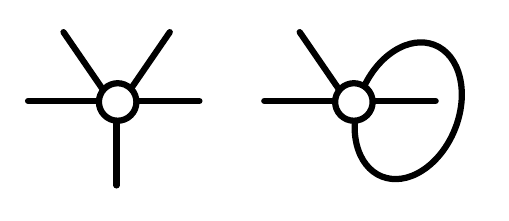}
\caption{An embedding $\medstar_5 \rat G$}
\label{fig contr corolla}
\end{figure}

\begin{lemma}\label{lem embed iso}
If $f \colon H \rightarrowtail G$ is an embedding and $H\cong G$, then $f$ is an isomorphism.
\end{lemma}
\begin{proof}
The map $f$ provides injections of finite sets $D_H \hookrightarrow D_G$ and $V_H \hookrightarrow V_G$ which must be isomorphisms.
To show that $A_H \to A_G$ is an injection (and hence an isomorphism), it is enough to show that the injection $\eth(H) \to A_G$ lands in $\eth(G)$.
This is clear if $G$ is an edge, since then $A_G = \eth(G)$. 
Assume that $G$ is not an edge.
Given $a\in \eth(G)$, we know $a^\dagger \in D_G$, so $a^\dagger = f(d)$ for some unique $d\in D_H$.
Notice that $d^\dagger \in \eth(H)$, for otherwise $a = f(d^\dagger)$ would be an element of $D_G$.
This shows that $f(\eth(H)) \subseteq \eth(G)$ as desired, so $f \colon A_H \to A_G$ is an isomorphism.
\end{proof}

In Definition 1.28 of \cite{HRY-mod1}, a set $\emb(G)$ is defined by considering all embeddings with codomain $G$, and then identifying isomorphic embeddings. That is, $f\sim h$ just when there exists a (unique) isomorphism $z$ with $f=hz$.
We next give an alternative definition of this set which reveals additional structure.

Given a (small) category $\mathbf{C}$, the posetal reflection of $\mathbf{C}$ is obtained by identifying objects $c_0$ and $c_1$ if and only if there are arrows $c_0 \to c_1$ and $c_1 \to c_0$ in $\mathbf{C}$, and declaring that on equivalence classes $[c] \leq [c']$ just when there is an arrow $c \to c'$ in $\mathbf{C}$.

\begin{proposition}
Let $\mathbf{\acute{E}}$ be the category of graphs and \'etale maps between them, let $G$ be a connected graph, and let $\mathbf{C} \subseteq \mathbf{\acute{E}}_{/G}$ be the full subcategory of the slice on the embeddings with codomain $G$.
Then $\emb(G)$ is the underlying set of the posetal reflection of (a skeleton of) $\mathbf{C}$.
\end{proposition}
\begin{proof}
First, notice that for a morphism 
\[ \begin{tikzcd}[column sep=small]
H \ar[rr,"f"] \ar[dr,tail,"h"'] & & K \ar[dl,tail,"k"] \\
& G
\end{tikzcd} \]
from $h$ to $k$, we have that the \'etale map $f$ is an embedding.
Suppose that we have two morphisms $f$ and $g$ of $\mathbf{C}$
\[ \begin{tikzcd}[column sep=small]
H \ar[tail,r,"f"] \ar[dr,tail,"h"'] & K \dar[tail,"k"] \ar[tail,r,"g"] & H  \ar[dl,"h"] \\
& G
\end{tikzcd} \]
then $gf \colon H \to H$ is an isomorphism by \cref{lem embed iso}.
Likewise, $fg \colon K \to K$ is an isomorphism.
Since $fg$ and $gf$ are isomorphisms, we conclude that $f$ is an isomorphism and hence $[h]=[k]$ in $\emb(G)$.

If $h = kz$ for some isomorphism $z$, then $h$ and $k$ are identified in the posetal reflection.
\end{proof}

In order to give the definition of the graphical category $\gcat$, we first need some notation.

\begin{definition}[Boundary and vertex sum of embeddings]\label{def vertex sum}
If $S$ is a set, write $\wp(S)$ for the power set, that is, for the set whose elements are the subsets of $S$.
Let $\mathbb{N}S$ denote the free commutative monoid on $S$, whose elements are finite unordered lists of elements of $S$. 
When $S$ is finite, regard $\wp(S)$ as the subset of $\mathbb{N}S$ containing those lists which do not have repeated elements.
Now suppose $G$ is a graph.
\begin{itemize}
\item Write $\varsigma \colon \emb(G) \to \mathbb{N}V_G$ for the map of sets which sends the class of an embedding $f\colon H \rat G$ to $\varsigma[f] = \sum_{v\in V_H} f(v)  \in \mathbb{N}V_G$.
As embeddings are injective on vertices, this comes from a function $\varsigma \colon \emb(G) \to \wp(V_G)$ sending $[f]$ to $f(V_H)$.
\item Write $\eth \colon \emb(G) \to \wp(A_G) \subseteq \mathbb{N}A_G$ for the map which sends the class of an embedding $f \colon H \rat G$ to $\eth([f]) = f(\eth(H)) \subseteq A_G$.
\end{itemize}
\end{definition}
By \cite[Lemma 1.20]{HRY-mod1}, $\eth(H) \to A_G$ is injective, so $\eth([f])$ is isomorphic to $\eth(H)$ and may be written as $\sum_{b\in \eth(H)} f(b) \in \mathbb{N}A_G$.
See \cref{subsec embeddings} below for more on embeddings.
The following appears as Definition 1.31 of \cite{HRY-mod1}.

\begin{definition}[Graphical map]\label{def graphical map}
Let $G$ and $G'$ be connected (undirected) graphs.
A \mydef{graphical map} $\varphi \colon G \to G'$ is a pair $(\varphi_0, \varphi_1)$ consisting of an involutive function $\varphi_0 \colon A_G \to A_{G'}$ and a function $\varphi_1 \colon V_G \to \emb(G')$, so that:
\begin{enumerate}[label=(\roman*),ref=\roman*]
\item The inequality
\[
	\sum_{v\in V_G} \varsigma(\varphi_1(v)) \leq \sum_{w\in V_{G'}} w
\]
holds in $\mathbb{N}V_{G'}$ (i.e.,\ $\sum \varsigma(\varphi_1(v)) \in \wp(V_{G'})$). \label{old graph def vertices}
\item For each $v \in V_G$, there is a bijection making the diagram 
\[ \begin{tikzcd} 
\nbhd(v) \rar[hook, "\dagger"] \dar[dashed, "\cong"] & A_G \dar{\varphi_0} \\
\eth(\varphi_1(v)) \rar[hook] & A_{G'}
\end{tikzcd} \]	
commute.\label{old graph def boundary}
\item If $\eth(G)$ is empty, then there exists $v \in V_G$ so that $\varphi_1(v)$ is not an edge.\label{old graph def collapse}
\end{enumerate}
\end{definition}

Connected graphs and graphical maps form a category denoted by $\gcat$.
The composition, which we will rarely need to deal with explicitly, appears in Definition 1.44 of \cite{HRY-mod1}.
See the second paragraph of \cref{remark composition} below for a description that will be used in the proof of \cref{lem functoriality of hatting}.
The purpose of condition \eqref{old graph def collapse} is to avoid the situation where graph substitution would result in a nodeless loop (\cref{ex undirected nodeless loop}).

\begin{definition}[Inert and active maps]\label{def emb act}
Let $H,G,G'$ be connected graphs.
\begin{itemize}
\item If $f\colon H \rat G$ is an embedding, then there is an associated graphical map $H \to G$ whose action on arcs agrees with $f$ and whose action on vertices is given by the composite
\[
	V_H \xrightarrow{f} V_G \hookrightarrow \emb(G).
\]
We often also call such a map \mydef{inert} and write $\gcatemb \subset \gcat$ for the wide subcategory consisting of the inert maps. (In \cite{HRY-mod1,HRY-mod2} this category was denoted by $\gcat_{\mathrm{emb}}$.)
\item A graphical map $\varphi \colon G \to G'$ is called \mydef{active} if $\varphi_0$ induces a bijection 
\[ \begin{tikzcd}
\eth(G) \rar[hook] \dar[dashed, "\cong"] & A_G \dar{\varphi_0} \\
\eth(G') \rar[hook] & A_{G'}
\end{tikzcd} \]
between boundary sets.
We write $\gcatact \subset \gcat$ for the wide subcategory consisting of all of the active maps, and we use the notation $G\ract G'$.
\end{itemize}
\end{definition}

\begin{example}\label{ex graphical star}
If $G$ is a graph and $\eth(G)$ has cardinality $n$, then there are $n!$ different active maps $\medstar_n \ract G$, and these are determined solely by a choice of bijection $\eth(\medstar_n) \cong \eth(G)$.
Given any two such active maps $\alpha, \alpha' \colon \medstar_n \ract G$, there is a unique automorphism $z$ of $\medstar_n$ with $\alpha' z = \alpha$.
\end{example}

The following is Theorem 2.15 of \cite{HRY-mod1}.
\begin{proposition}\label{prop U ofs}
The pair $(\gcatact,\gcatemb)$ is an orthogonal factorization system on $\gcat$.
That is, both subcategories contain all isomorphisms, every graphical map $\varphi \colon G \to G'$ factors as an active map followed by an embedding
\[ \begin{tikzcd}[column sep=small]
G \ar[rr,"\varphi"] \ar[dr, -act] & & G' \\
& K \ar[ur,tail] 
\end{tikzcd} \]
and this factorization is unique up to unique isomorphism.
\qed
\end{proposition}

\begin{remark}\label{remark composition}
If $\varphi \colon G \to G'$ is a graphical map and $\varphi_v \colon H_v \rat G'$ represents $\varphi_1(v)$, condition \eqref{old graph def boundary} of \cref{def graphical map} provides a bijection between $\nbhd(v)$ and $\eth(H_v)\cong \eth(\varphi_1(v))$.
This is the data necessary to construct the \emph{graph substitution} $G\{H_v\}$ (see, for instance, \cite[Chapter 5]{YauJohnson:FPAM}, \cite[\S1.2]{HRY-mod1}, and \cite[\S13]{BataninBerger:HTAPM}).
Condition \eqref{old graph def collapse} is included to guarantee that $G\{H_v\}$ is not a nodeless loop, while condition \eqref{old graph def vertices} ensures that there is an injection $V_{G\{H_v\}} \cong \amalg V_{H_v} \hookrightarrow V_{G'}$.
This graph substitution comes equipped with an embedding $G\{H_v\} \rat G'$ canonically factoring all of the embeddings $\varphi_v \colon H_v \rat G'$, and we can take $K= G\{H_v\}$ in the factorization from \cref{prop U ofs}.

Composition of graphical maps was formulated in \cite[Definition 1.44]{HRY-mod1} by first defining precomposition with inert maps, and then utilizing the above graph substitutions.
This means that if $\psi \colon G' \to G''$ is another graphical map and $v\in V_G$, then $(\psi \circ \varphi)_1(v) \in \emb(G'')$ is represented by the right vertical embedding in the diagram below.
\[ \begin{tikzcd}
& H_v \rar[-act] \dar["\varphi_v"', tail] & \bullet \dar[tail] \\
G \rar{\varphi} & G' \rar{\psi} & G''
\end{tikzcd} \]
\end{remark}

\subsection{Lemmas concerning embeddings} \label{subsec embeddings}
We now give three lemmas governing the behavior of embeddings which will be useful several times in this paper.
The first appears as Lemma~1.22 of \cite{HRY-mod1}, and the reader should note that it is not possible to have $i=j$.
\begin{lemma}\label{mod1 lem 1.22}
Suppose $f \colon G \rat G'$ is an embedding and $a_1\neq a_2$ are distinct arcs of $G$ with $f(a_1) = f(a_2)$.
Then there are indices $i,j \in \{1,2\}$ so that $a_i, a_j^\dagger \in \eth(G)$ and $a_i^\dagger, a_j \in D_G$. \qed
\end{lemma}

The second, which appears as Proposition 1.25 of \cite{HRY-mod1}, says that embeddings are nearly determined by their boundary, up to some ambiguity about whether or not they contain a vertex.
\begin{lemma}\label{mod1 prop 1.25}
Let $G$ be a connected graph, and $E_G \hookrightarrow \emb(G)$ be the inclusion of edges into embeddings from \cref{example embeddings}.
Then the composites
\begin{align*}
E_G \hookrightarrow \emb(G) & \xrightarrow{\eth} \mathbb{N}A_G \\
\emb(G) \setminus E_G \hookrightarrow \emb(G) & \xrightarrow{\eth} \mathbb{N}A_G
\end{align*}
are injective. \qed
\end{lemma}

The third is a generalization of Lemma 1.26 of \cite{HRY-mod1}, which says that an embedding whose domain has empty boundary must be an isomorphism.

\begin{lemma}\label{lem boundary inclusion}
If $f \colon H \rightarrowtail G$ is an embedding and $\eth(f) \subseteq \eth(G)$, then $f$ is an isomorphism.
In particular, $f$ must be an isomorphism if $\eth(H) = \varnothing$.
\end{lemma}
\begin{proof}
If $H$ is an edge, then $\eth(G)$ contains a pair $\{a,a^\dagger\} = \eth(f)$. 
Since $G$ is connected, it must be an edge as well.

Now suppose $V_H$ is inhabited.
The argument from the proof of \cite[Lemma 1.26]{HRY-mod1} shows that $V_H \to V_G$ is surjective, hence a bijection.
Since $f$ is \'etale, $D_H \to D_G$ is an isomorphism as well.
If there were an element $a\in \eth(G) \setminus \eth(f)$, then since $G$ is not an edge we would have $a^\dagger \in D_G$, hence $a^\dagger = f(d)$ for some (unique) $d\in D_H$.
But then $f(d^\dagger) = a \in \eth(G)$, and only elements of $\eth(H)$ can map to elements of $\eth(G)$, so $d^\dagger \in \eth(H)$. This contradicts the assumption on $a$.
It follows that $\eth(H) \hookrightarrow \eth(G)$ is also surjective, hence $A_H \to A_G$ is an isomorphism.
\end{proof}

\section{Unions of embeddings}\label{section unions}
In this section we consider a notion of unions of embeddings into a graph $G$.
Despite the fact that $\emb(G)$ is a poset, it is not particularly well-behaved, and our notion has little to do with any joins that happen to exist.
As an example, the embeddings $h$ and $k$ in \cref{fig no joins} do not have a least upper bound, since the domain of such would necessarily be disconnected.
Even when a least upper bound exists, it may be too big to be a union (see the paragraph following Proposition 2.2.8 of \cite{ChuHackney} for one example).
See also \cref{ex lub not preserved}.

\begin{definition}\label{def a union}
Suppose that $h\colon H \rightarrowtail G$ and $k\colon K \rightarrowtail G$ are two embeddings.
An embedding $\ell \colon L \rightarrowtail G$ is called \mydef{a union of $h$ and $k$} if 
\begin{enumerate}
\item $[\ell]$ is an upper bound for both $[h]$ and $[k]$ in the poset $\emb(G)$ (that is, there is a factorization $H \rightarrowtail L \rightarrowtail G$ of $h$ and likewise for $k$), and \label{def a union ub}
\item $\ell (V_L) = h(V_H) \cup k(V_K)$. \label{def a union cup}
\end{enumerate}
\end{definition}

Unions in this sense are frequently not unique, as we can see in \cref{fig no joins}. 
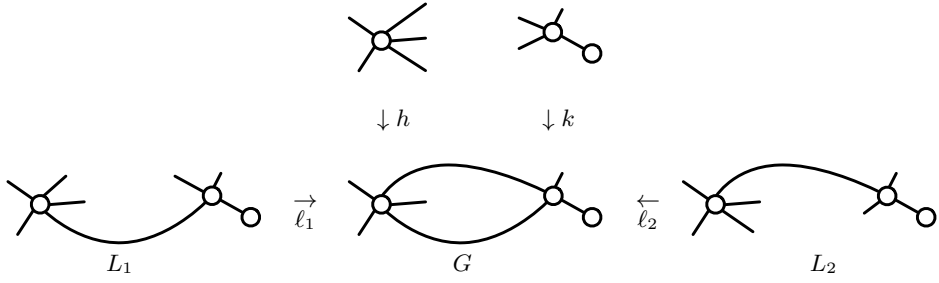
\begin{figure}[tb]
\labellist
\small\hair 2pt
 \pinlabel {$\downarrow h$} at 406 139
 \pinlabel {$\downarrow k$} at 574 139
 \pinlabel {$\rightarrow$} at 316 55
 \pinlabel {$\leftarrow$} at 665 55
 \pinlabel {$\ell_1$} [B] at 316 30
 \pinlabel {$\ell_2$} [B] at 665 30
 \pinlabel {$L_1$} [B] at 127 -15
 \pinlabel {$G$} [B] at 476 -15
 \pinlabel {$L_2$} [B] at 844 -15
\endlabellist
\centering
\includegraphics[width=\textwidth]{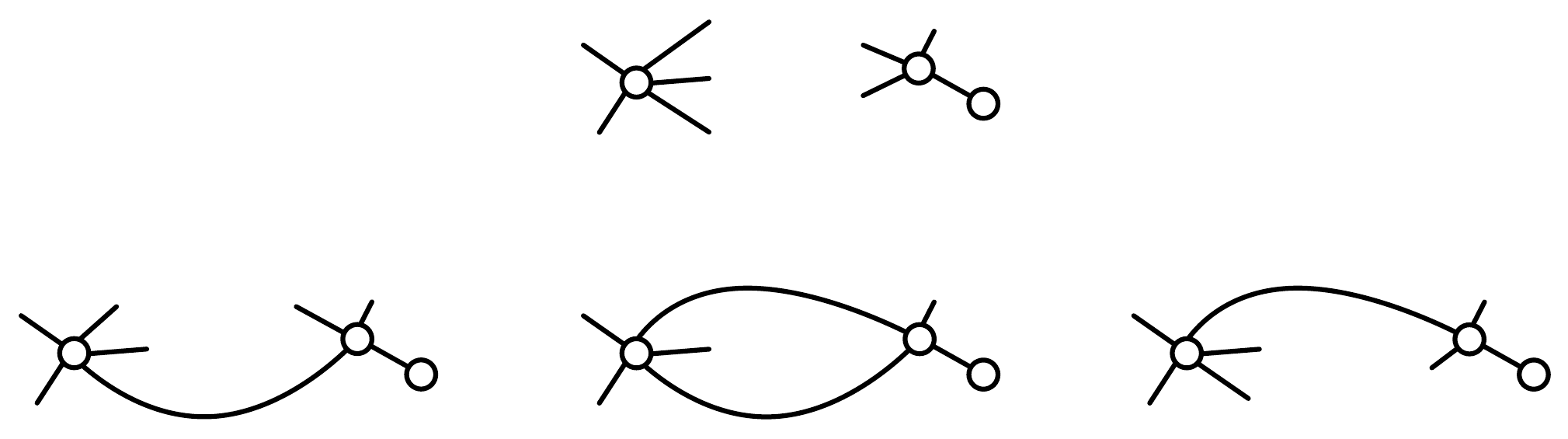}
\caption{Each of $\ell_1$, $\ell_2$, and $\id_G$  is a union for $h$ and $k$.}
\label{fig no joins}
\end{figure}

\begin{example}\label{example order}
If $[h] \leq [k]$ in the poset $\emb(G)$, then $k$ is a union of $h$ and $k$.
\end{example}

We now show that if two embeddings have any intersection at all, then they have at least one union.
\begin{proposition}\label{prop unions exist}
Suppose $h\colon H \rat G$ and $k\colon K \rat G$ are embeddings.
If either $h(A_H) \cap k(A_K)$ or $h(V_H) \cap k(V_K)$ is inhabited, then there exists at least one union $\ell \colon L \rat G$ of $h$ and $k$.
\end{proposition}
This proposition follows from \cref{lem: L is graph} and \cref{lem: ell is embedding} below.
It is useful if one wishes to directly construct the active-inert factorizations from \cref{prop U ofs} using only the description of graphical maps from \cref{def new graph map}.
Existence of unions isn't strictly necessary for the rest of this paper, and we recommend skipping ahead to \cref{sec new graph maps} on first reading.

Suppose we have two embeddings $h\colon H \rightarrowtail G$ and $k\colon K \rightarrowtail G$.
We can form the pullback 
\[ \begin{tikzcd}
J \rar \dar \ar[dr, phantom, "\lrcorner" very near start] & H \dar{h} \\
K \rar{k} & G
\end{tikzcd} \]
in the category $\finset^{\mathscr{I}}$.
The object $J$ will still be a graph, though in general it does not need to be connected.
In fact, since embeddings are not always monomorphisms, we may have that $J$ is not connected even when $h=k$.
See \cref{ex loop vertex} below.

\begin{lemma}\label{lem: L is graph}
Suppose $h$ and $k$ are embeddings.
Form an object $L\in \finset^{\mathscr{I}}$ by taking a pullback followed by a pushout.
\[ 
\begin{tikzcd}
J \rar \dar \ar[dr, phantom, "\lrcorner" very near start] & H \dar{h} \\
K \rar{k} & G
\end{tikzcd}
\qquad \rightsquigarrow \qquad
\begin{tikzcd}
J \rar \dar \ar[dr, phantom, "\ulcorner" very near end] & H \dar \\
K \rar & L 
\end{tikzcd} \]
The object $L$ is a graph.
\end{lemma}
\begin{proof}
Consider the map $\ell$ induced by the pushout property:
\[ \begin{tikzcd}
J \rar \dar \ar[dr, phantom, "\ulcorner" very near end] & H \dar \ar[ddr,"h", bend left] \\
K \rar \ar[drr,"k", bend right] & L \ar[dr,dotted,"\ell"] \\ 
& & G
\end{tikzcd} \]
First note that the involution on $A_L$ does not have a fixed point: suppose, to the contrary, that $a\in A_L$ satisfies $a = a^\dagger$.
Then $\ell(a) = \ell(a^\dagger) = \ell(a)^\dagger$, implying that $\ell(a)$ is a fixed point of $A_G$, a contradiction.

It remains to show that $D_L \to A_L$ is injective.
But this map fits into the diagram
\[ \begin{tikzcd}
D_L  \dar{\cong} \rar & A_L \dar{\cong}  \\
D_H \uamalg{D_J} D_K \rar \dar & A_H \uamalg{A_J} A_K \dar \\
D_G \rar & A_G.
\end{tikzcd} \]
Since $D_H \to D_G$ and $D_K\to D_G$ are injective and $D_J$ is $D_H \times_{D_G} D_K$, the left vertical map is injective. 
The bottom map is injective because $G$ is a graph, hence the top map is injective as well.
\end{proof}

\begin{lemma}\label{lem: ell is embedding}
Consider the diagram
\[ \begin{tikzcd}
J \rar \dar \ar[dr, phantom, "\ulcorner" very near end] & H \dar{f} \ar[ddr,"h", bend left] \\
K \rar[swap]{g} \ar[drr,"k", bend right] & L \ar[dr,dotted,"\ell"] \\ 
& & G
\end{tikzcd} \]
from the previous lemma. 
If $J$ is inhabited, then $L$ is connected and $f,g,$ and $\ell$ are embeddings.
\end{lemma}
\begin{proof}
Suppose $J$ contains an arc or a vertex, classified by an \'etale map $E \to J$ from $E = \exedge$ or $E = \medstar_n$.
If $L$ splits as a coproduct $L = L_1 \amalg L_2$, then since $H$ and $K$ are connected, $f$ factors through exactly one $L_i$ and $g$ factors through exactly one $L_j$. 
Since both of these factor the composition $E \to J \to L$ and $E$ is connected, we have $i=j$ and hence the other term is empty. Thus $L$ is connected.

We next show that $\ell$ is an embedding.
We have $V_J \cong V_H \times_{V_G} V_K \cong h(V_H) \cap k(V_K)$, so $V_L = V_H \amalg_{V_J} V_K$ maps monomorphically to $h(V_H) \cup k(V_K) \subseteq V_G$.
Similarly, we have $D_L \cong h(D_H) \cup k(D_K)$.
As $h$ and $k$ are embeddings, the bottom square in the diagram
\[ \begin{tikzcd}
h(D_H)  \rar\dar  & h(V_H)  \dar \\
h(D_H) \cup k(D_K) \rar\dar  \ar[dr, phantom, "\lrcorner" very near start] & h(V_H) \cup k(V_K) \dar \\
D_G \rar & V_G 
\end{tikzcd} \]
is a pullback, hence $\ell$ is an embedding.
Using the pasting law for pullbacks and the fact that $\ell$ and $h$ are embeddings, the top square is also a pullback, hence $f$ is an embedding.
A similar proof shows $g$ is an embedding.
\end{proof}

The construction from \cref{lem: L is graph} produces a union from embeddings $h$ and $k$ which intersect, but it may not be the one you'd expect.
For example, even when $h=k$ it may be that $[\ell]$ is strictly larger than $[h]$.

\begin{example}\label{ex loop vertex}
Suppose that $h=k \colon H=K=\medstar_2 \rat G$ picks out the vertex in the loop with one vertex (see \cref{fig loop one vertex}).
One can compute that the pullback $J$ has
\[
	A_J = \{ (1,1), (1^\dagger,1^\dagger), (2,2), (2^\dagger, 2^\dagger), (1,2^\dagger), (1^\dagger,2), (2^\dagger,1), (2,1^\dagger) \}
\]
while $D_J = \{1,2\}$ and $V_J = \{ \bullet \}$.
It follows that $\ell \colon L \rightarrowtail G$ is an isomorphism, hence $[h] < [\id_G] = [\ell]$.
\end{example}

\begin{figure}[tb]
\includegraphics[scale=0.6]{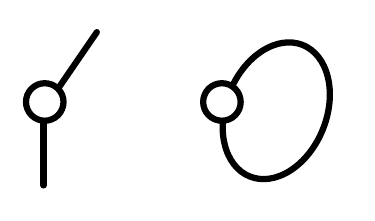}
\caption{The embedding $h=k \colon \medstar_2 \rat G$ from \cref{ex loop vertex}}
\label{fig loop one vertex}
\end{figure}

\section{A new description of graphical maps}\label{sec new graph maps}

In this section we give an alternative definition of graphical map as a pair of functions with certain properties, which makes composition transparent.
These new maps encode the same data as those from \cref{def graphical map}, but encode it in a different way.
After presenting the definition, we spend the remainder of the section showing how to go back and forth between the two descriptions, and finally give the equivalence in \cref{old new equivalence}.

\begin{definition}\label{def vertex disjoint}
Two embeddings $h \colon H \rightarrowtail G$ and $k \colon K \rightarrowtail G$ are said to be \mydef{vertex disjoint} if $h(V_H) \cap k(V_K) \subseteq V_G$ is empty.
In this case we also say the pair $[h],[k]\in \emb(G)$ is vertex disjoint, and we note that this property does not depend on a choice of representatives.
\end{definition}

\begin{definition}\label{def new graph map}
Let $G$ and $G'$ be connected graphs.
A \mydef{new graph map} $\varphi \colon G \to G'$ is a pair $(\varphi_0, \hat \varphi)$ consisting of an involutive function $\varphi_0 \colon A_G \to A_{G'}$ and a function $\hat \varphi \colon \emb(G) \to \emb(G')$, so that:
\begin{enumerate}[label=(\roman*),ref=\roman*]
\item The function $\hat \varphi$ sends edges to edges. \label{new graph def edges}
\item The function $\hat \varphi$ preserves unions: 
if $[\ell]$ is a union of $[h]$ and $[k]$ in $\emb(G)$, then $\hat \varphi[\ell]$ is a union of $\hat \varphi[h]$ and $\hat \varphi[k]$ in $\emb(G')$. \label{new graph def union}
\item The function $\hat \varphi$ takes vertex disjoint pairs to vertex disjoint pairs. \label{new graph def intersect}
\item The diagram
\[ \begin{tikzcd}
\emb(G) \rar{\eth} \dar{\hat \varphi}& \mathbb{N}A_G \dar{\mathbb{N} \varphi_0 }\\
\emb(G') \rar{\eth} & \mathbb{N}A_{G'}
\end{tikzcd} \]
commutes. \label{new graph def boundary}
\end{enumerate}
Composition of new graph maps is given by composition of the two constituent functions.
\end{definition}

\begin{remark}\label{rem graph def boundary}
Condition \eqref{new graph def boundary} is equivalent to insisting that, for each $[h] \in \emb(G)$, there is a (necessarily unique) bijection as displayed to the left:
\[ \begin{tikzcd}
\eth([h]) \rar[hook] \dar[dashed, "\cong"] & A_G \dar{\varphi_0} \\
\eth(\hat\varphi [h]) \rar[hook] & A_{G'}
\end{tikzcd} \]	
\end{remark}

By \cref{example order}, condition \eqref{new graph def union} in particular implies that $\hat \varphi \colon \emb(G) \to \emb(G')$ preserves the partial order.
Notice that condition \eqref{new graph def edges} is nearly the same as asking that $\hat \varphi$ preserves minimal elements: indeed, for almost all connected graphs $G$, the set of edges coincides with the set of minimal elements of $\emb(G)$.
The lone exception is the case when $G = \medstar_0$ is an isolated vertex, which has no edges -- $\emb(\medstar_0)$ has a unique element, and \cref{lem boundary inclusion} implies that this element maps to the \emph{maximal} element of $\emb(G')$.

Our aim in this section is to show that new graph maps are equivalent to graphical maps from \cref{def graphical map}.
This work culminates in \cref{old new equivalence}, which shows that the following two constructions are inverses.

\begin{construction}\label{cns new to old}
Given a new graph map $\varphi \colon G \to G'$, we obtain the composite
\[
	\varphi_1 \colon V_G  \hookrightarrow \emb(G) \xrightarrow{\hat \varphi} \emb(G'). 
\]
We write $\newold$ for the assignment that takes a new graph map $\varphi = (\varphi_0, \hat \varphi)$ and produces the pair $\newold \varphi \coloneqq (\varphi_0, \varphi_1)$.
\end{construction}

\begin{construction}\label{cns old to new}
Suppose that $\varphi = (\varphi_0, \varphi_1) \colon G \to G'$ is a graphical map.
Define a function
\[
	\hat \varphi \colon \emb(G) \to \emb(G')
\]
by declaring that $\hat \varphi [f] \coloneqq [g]$, where $g$ sits inside the active-inert factorization of $\varphi \circ f$ from \cref{prop U ofs}:
\[ \begin{tikzcd}
\bullet \dar[tail,"f"] \rar[-act] & \bullet \dar[tail,"g"] \\
G \rar{\varphi} & G'.
\end{tikzcd} \]
As we are working with an orthogonal factorization system, the element $[g]$ does not depend on the choice of representative for $[f]$ or the chosen factorization.
We write $\oldnew$ for the assignment that takes a graphical map $\varphi = (\varphi_0, \varphi_1)$ and produces the pair $\oldnew \varphi \coloneqq (\varphi_0, \hat \varphi)$.
\end{construction}

With this construction in hand, we give an example to show that graphical maps do not preserve least upper bounds of embeddings. 

\begin{example}\label{ex lub not preserved}
Consider the graphs $G$ (on the left) and $G'$ (on the right) of \cref{fig: first counterexample}, and let $\varphi \colon G \to G'$ be the embedding.
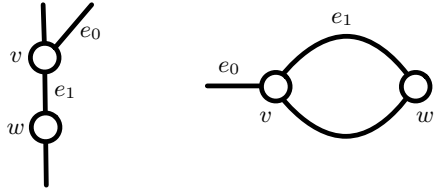
\begin{figure}[htb]
\labellist
\small\hair 2pt
 \pinlabel {$v$} at 3 94
 \pinlabel {$w$} at 3 49
 \pinlabel {$v$} at 160 56
 \pinlabel {$w$} at 260 56
 \pinlabel {$e_0$} at 50 109
 \pinlabel {$e_0$} at 137 87
 \pinlabel {$e_1$} at 34 72
 \pinlabel {$e_1$} at 208 118
\endlabellist
\centering
\includegraphics[scale=0.6]{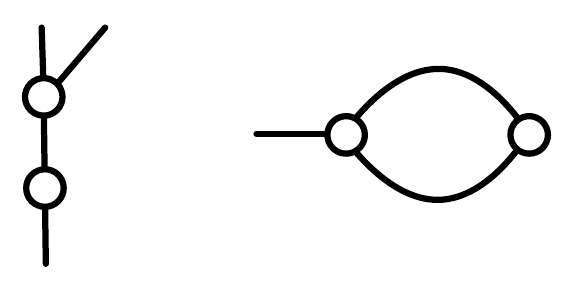}
\caption{Graphs $G$ and $G'$}
\label{fig: first counterexample}
\end{figure}
Now $\iota_v \colon \medstar_v \rightarrowtail G$ and $\iota_w \colon \medstar_w \rightarrowtail G$ have a least upper bound, namely $\id_G$.
But 
\begin{align*}
	\hat\varphi [\iota_v] &= [\iota_v \colon \medstar_v \to G'] \\
	\hat\varphi [\iota_w] &= [\iota_w \colon \medstar_w \to G'] \\
	\hat\varphi [\id_G] &= [\varphi \colon G \to G'].
\end{align*}
The element $\hat\varphi [\id_G] = [\varphi]$ is not the least upper bound for the two star inclusions in $G'$.
Indeed, by snipping the edge $e_1$ in $G'$ we obtain another embedding which factors the two star inclusions, but which does not factor $\varphi$. 
(As in \cref{fig no joins}, the least upper bound does not exist.)
\end{example}

\subsection{From new graph maps to graphical maps}\label{subsec new to old}
Our aim in this section is to show that $\newold$ takes new graph maps to graphical maps.

\begin{theorem}\label{thm new graph to graphical}
If $\varphi = (\varphi_0, \hat \varphi)$ is a new graph map, then $\newold\varphi = (\varphi_0,\varphi_1)$ is a graphical map in the sense of \cref{def graphical map}.
\end{theorem}

The following statement involves the vertex sum function $\varsigma \colon \emb(G) \to \mathbb{N}V_G$ from \cref{def vertex sum}.

\begin{lemma}\label{lem vertex sums}
Given a diagram of embeddings, 
\[ \begin{tikzcd}
H \rar[tail,"f"]  \ar[dr,"h"', tail] & L \dar["\ell", tail] &
K \ar[dl,"k", tail] \lar[tail,"g"'] \\
& G
\end{tikzcd} \]
we have
\begin{equation}
\varsigma[\ell] \geq \varsigma[h] + \varsigma[k] - \sum_{v\in f(V_H) \cap g(V_K)} \ell(v)
\end{equation}
in $\mathbb{N}V_G$. 
Equality holds if and only if $\ell$ is a union of $h$ and $k$.
If $h$ and $k$ are vertex disjoint and $\ell$ is a union of $h$ and $k$, then $\varsigma[\ell] = \varsigma[h] + \varsigma[k]$.
\end{lemma}
\begin{proof}
We compute, using that embeddings are injective on vertices:
\begin{align*}
\varsigma[h] + \varsigma[k] &= \sum_{u\in V_H} h(u) + \sum_{w\in V_K} k(w) = \sum_{u'\in f(V_H)} \ell(u') + \sum_{w'\in g(V_K)} \ell(w') \\
&= \sum_{v\in f(V_H) \cup g(V_K)} \ell(v) + \sum_{v\in f(V_H) \cap g(V_K)} \ell(v)  \leq \varsigma[\ell] + \sum_{v\in f(V_H) \cap g(V_K)} \ell(v)
\end{align*}
The last inequality is an equality if and only if $f(V_H) \cup g(V_K) = V_L$, that is, if and only if $\ell$ is a union of $h$ and $k$.
If $h$ and $k$ are vertex disjoint then $f(V_H) \cap g(V_K)$ is empty, so the final statement holds.
\end{proof}

\begin{lemma}\label{lem find deletable vertex}
Given an embedding $\ell \colon L \rightarrowtail G$ where $L$ has an internal edge, there exists a vertex $v \in L$ and an embedding $g \colon K \rightarrowtail L$ with $g(V_K) = V_L \setminus \{v\}$.
In particular, $\ell$ is a union of $h = \ell \iota_v \colon \medstar_v \rightarrowtail G$ and $k = \ell g \colon K \rightarrowtail G$.
\end{lemma}
\begin{proof}
If $L$ has only a single vertex $v$, then since $L$ contains an edge, we can form
\[ \begin{tikzcd}
\medstar_v \rar[tail,"\iota_v"]  \ar[dr,"h"', tail] & L \dar["\ell", tail] &
K = \exedge \ar[dl,"k", tail] \lar[tail,"g"'] \\
& G
\end{tikzcd} \]
where $g$ classifies some arc of $L$. The conclusion follows.

Now suppose that $L$ has two or more vertices.
Consider the classical graph $\mathrm{core}(L)$ obtained by deleting the $\dagger$-orbit of $\eth(L)$ (Definition 1.9 of \cite{HRY-mod1}).
As $\mathrm{core}(L)$ is connected and has more than one vertex, there exists a vertex $v$ so that deleting $v$ and all edges incident to it from $\mathrm{core}(L)$ will still give a connected graph. 
Setting
\[
S = \nbhd(v)^\dagger \cap (\eth(L)\cup \nbhd(v)),
\] we have $S^\dagger \subsetneq \nbhd(v)$; also notice that $S \cap S^\dagger = \nbhd(v) \cap \nbhd(v)^\dagger$.
Define a graph $K$ by
\begin{align*}
V_K &= V_L \setminus \{ v\} \\
D_K &= D_L \setminus \nbhd(v) \\
A_K &= A_L \setminus (S \cup S^\dagger).
\end{align*}
We compute the boundary of $K$ as
\[
	\eth(K) = (\eth(L) \setminus S) \cup (\nbhd(v) \setminus S^\dagger),
\]
and use this to show that $\eth(K)^\dagger \subseteq D_K$.
If $a\in \eth(L)\setminus S$, then $a^\dagger \notin \nbhd(v)$, hence $a^\dagger \in D_K$.
If $a \in \nbhd(v) \setminus S^\dagger$, then we can't have $a^\dagger \in \eth(L)$ for then $a^\dagger \in \nbhd(v)^\dagger \cap \eth(L) \subseteq S$, and we also can't have $a^\dagger \in \nbhd(v)$ since then $a \in \nbhd(v)\cap\nbhd(v)^\dagger \subseteq S^\dagger$; it follows that $a^\dagger\in D_K$.
We have thus shown that $K$ does not contain any edge components. As $\mathrm{core}(K) = \mathrm{core}(L) - v$ is connected, we conclude that $K$ is connected.
The inclusions of subsets give our desired embedding $g\colon K \rightarrowtail L$.
\end{proof}

\begin{proposition}\label{prop sum splitting new graph}
Given an embedding $\ell \colon L \rat G$ and a new graph map $\varphi \colon G \to G'$, we have
\[
	\varsigma \hat \varphi [\ell] = \sum_{v\in V_L} \varsigma \hat \varphi [\ell \iota_v].
\]
\end{proposition}
\begin{proof}
If $L$ is an edge, so is $\hat \varphi[\ell]$ by  \cref{def new graph map}\eqref{new graph def edges}, hence $\varsigma \hat \varphi[\ell] = 0$ as desired.
If $L$ is a star, then the formula is immediate. 
We have thus established the formula whenever $L$ does not contain an internal edge.

For the general result, we induct on the number of vertices of $L$, which we assume to have an internal edge (in other words, we are investigating what happens when $L$ is not an edge and not a star).
Applying \cref{lem find deletable vertex}, we obtain a diagram
\[ \begin{tikzcd}
H = \medstar_v \rar[tail,"\iota_v"]  \ar[dr,"h"', tail] & L \dar["\ell", tail] &
K \ar[dl,"k", tail] \lar[tail,"g"'] \\
& G
\end{tikzcd} \]
with $\ell$ a union of $h$ and $k$, and $h, k$ vertex disjoint.
Applying $\hat \varphi$, we obtain embeddings
\[ \begin{tikzcd}
H' \rar[tail]  \ar[dr,"\hat h"', tail] & L' \dar["\hat \ell", tail] &
K' \ar[dl,"\hat k", tail] \lar[tail] \\
& G
\end{tikzcd} \]
with $\hat \ell$ a union of $\hat h$ and $\hat k$ (by \cref{def new graph map}\eqref{new graph def union}) and $\hat h, \hat k$ vertex disjoint (by \cref{def new graph map}\eqref{new graph def intersect}), hence 
\[ \varsigma \hat \varphi [\ell] = \varsigma [\hat \ell] = \varsigma [\hat k] + \varsigma [\hat h] = \varsigma \hat \varphi [k] + \varsigma \hat \varphi [\ell \iota_v] \]
by \cref{lem vertex sums}.
As $K$ has one less vertex than $L$, by the induction hypothesis we have the first equality below:
\[
	\varsigma \hat \varphi [k] = \sum_{w\in V_K} \varsigma \hat \varphi [k \iota_w] = \sum_{w\in V_K} \varsigma \hat \varphi [\ell g \iota_w] = \sum_{u\in V_L \setminus \{ v\} } \varsigma \hat \varphi [\ell \iota_u].
\]
The result follows by combining the previous two displays.
\end{proof}

\begin{proof}[Proof of \cref{thm new graph to graphical}]
Suppose $(\varphi_0, \hat \varphi) \colon G \to G'$ is a new graph map.
We verify each of the three conditions from \cref{def graphical map}.
\begin{enumerate}[label=(\roman*),ref=\roman*]
	\item Applying \cref{prop sum splitting new graph} to the identity embedding $\id_G \colon G \rat G$, we find
	\[
		\sum_{v\in V_G} \varsigma(\varphi_1(v)) = \sum_{v\in V_G} \varsigma \hat \varphi[\iota_v] = \varsigma \hat \varphi[\id_G] \leq \varsigma [\id_{G'}] = \sum_{w\in V_{G'}} w.
	\]
	\item This is the special case of the interpretation of \cref{def new graph map}\eqref{new graph def boundary}
	from \cref{rem graph def boundary} where we take $h = \iota_v$:

	\[ \begin{tikzcd}
\nbhd(v) \rar["\dagger", "\cong" swap] & \eth([\iota_v]) \rar[hook] \dar[dashed, "\cong"] & A_G \dar{\varphi_0} \\
\eth(\varphi_1(v)) \rar{=} & \eth(\hat\varphi [\iota_v]) \rar[hook] & A_{G'}
\end{tikzcd} \]	
\item Suppose the boundary of $G$ is empty, and consider an embedding $h \colon H \rat G'$ which represents $\hat \varphi [\id_G]$.
By \cref{def new graph map}\eqref{new graph def boundary} we have that $\eth(H)$ is empty, hence $h$ is an isomorphism by \cref{lem boundary inclusion}.
By \cref{prop sum splitting new graph} we have
\[
	\sum_{v\in V_G} \varsigma \varphi_1(v) = \varsigma \hat \varphi [\id_G] = \varsigma [h] = \varsigma [\id_{G'}] = \sum_{w\in V_{G'}} w.
\]
Since $\eth(G')$ is empty, $G'$ contains at least one vertex, implying that there is at least one $v\in V_G$ with $\varsigma \varphi_1(v) \neq 0$. \qedhere
\end{enumerate}
\end{proof}

\subsection{From graphical maps to new graph maps}\label{subsec old to new}
Recall \cref{cns old to new} which takes a graphical map $\varphi = (\varphi_0, \varphi_1) \colon G \to G'$ to the pair $\oldnew \varphi = (\varphi_0, \hat \varphi)$.
This function $\hat \varphi$ is defined so that $\hat \varphi [f] = [g]$ where $g$ sits inside an active-inert factorization of $\varphi f$:
\[ \begin{tikzcd}
\bullet \dar[tail,"f"] \rar[-act] & \bullet \dar[tail,"g"] \\
G \rar{\varphi} & G'.
\end{tikzcd} \]
We will show that $\oldnew \varphi$ is a new graph map in \cref{thm oldnew graph map}, but first we first establish functoriality of the construction.
After establishing that $\oldnew \varphi$ is a new graph map, we will show in \cref{old new equivalence} that $\oldnew$ and $\newold$ are inverses.

\begin{lemma}\label{lem functoriality of hatting}
Given a composable pair of graphical maps
\[
	G \xrightarrow{\varphi} H \xrightarrow{\psi} K
\]
we have that $\hat \psi \circ \hat \varphi = \widehat{\psi \circ \varphi}$.
Further, $\widehat {\id_G} = \id_{\emb(G)}$.
\end{lemma}
\begin{proof}
The second statement is clear since we have an active-inert orthogonal factorization system on $\gcat$.
The first statement also follows from the factorization system: suppose $\ell \colon L \rightarrowtail G$ is an embedding.
\[ \begin{tikzcd}
& \bullet \rar[-act,"q"] \ar[dr,tail,"f"] & \bullet \ar[dr,tail,"h"] \\
L \rar[tail,"\ell"] \ar[ur,-act,"p"] & G \rar{\varphi} & H \rar{\psi} & K
\end{tikzcd} \]
First factor $\varphi \ell$ as $f\circ p$, and then factor $\psi f$ as $h\circ q$, and notice that $h \circ (qp)$ is an active-inert factorization of $\psi \varphi \ell$.
It follows that $\hat \varphi [\ell] = [f]$, $\hat \psi [f] = [h]$, and $\widehat {\psi \circ \varphi} [\ell] = [h]$.
\end{proof}

Our next goal is to show that $\oldnew$ from \cref{cns old to new} takes graphical maps to new graph maps.
For this purpose, consider a graphical map $\varphi \colon G \to G'$ and embeddings $h,k,\ell$ into $G$ with $[h] \leq [\ell]$ and $[k]\leq [\ell]$.
We form an active-inert factorization $\hat \ell \circ \alpha$ of $\varphi\ell$, and then active-inert factorizations of $\alpha f$ and $\alpha g$ as displayed below-right.
\begin{equation}\label{eq upper bounds names}
\begin{tikzcd}
H \dar[tail,"f"']  \ar[dr,"h", tail] &  \\
L \rar["\ell", tail] & G \rar["\varphi"] & G' \\
K \ar[ur,"k", tail] \uar[tail,"g"] &
\end{tikzcd} 
\qquad \rightsquigarrow \qquad
\begin{tikzcd}
H \rar[-act,"\beta"] \dar[tail,"f"'] & H'\dar[tail,"f'"']  \ar[dr,"\hat h", tail] \\
L \rar[-act,"\alpha"] & L' \rar[tail,"\hat \ell"] & G' \\
K \rar[-act,"\gamma"] \uar[tail,"g"] & K' \uar[tail,"g'"] \ar[ur,"\hat k", tail]
\end{tikzcd} \end{equation}
In particular, we have that $\hat \varphi [\ell] = [\hat \ell]$ is an upper bound for both $\hat \varphi [h]$ and $\hat \varphi [k]$.

\begin{lemma}
If $\ell$ is a union of $h$ and $k$, then $\hat \varphi [\ell]$ is a union of $\hat \varphi [h]$ and $\hat \varphi [k]$.
\end{lemma}
\begin{proof}
Suppose $v\in V_{L'}$. 
Our goal is to show that $\hat \ell (v)$ is an element of $\hat h(V_H') \cup \hat k(V_{K'})$.
Since $\alpha \colon L\to L'$ is active, by \cite[Proposition 2.3]{HRY-mod1} there is a unique vertex $w\in V_L$ so that $\alpha_w \colon M_w \rightarrowtail L'$ (representing $\alpha_1(w)$) has $v$ in its image.
As $V_H \amalg V_K \to V_L$ is surjective, there is a vertex $w'$ of $H$ or $K$ mapping to $w$; without loss of generality we assume $w'\in V_H$.
Considering the square
\[ \begin{tikzcd}
H \rar[-act,"\beta"] \dar[tail,"f"] & H'\dar[tail,"f'"] \\
L \rar[-act,"\alpha"] & L' \\
\end{tikzcd} \]
of \eqref{eq upper bounds names}, 
we get that the triangle
\[ \begin{tikzcd}
M_w \rar[tail,"\beta_{w'}"] \ar[dr,"\alpha_w"',tail] & H' \dar[tail,"f'"] \\
& L'
\end{tikzcd} \]
commutes by the descriptions of composing with embeddings from \cite[Definition 1.34]{HRY-mod1}.
Hence $v$ is in the image of $H' \rightarrowtail L'$, so $\hat \ell (v)$ is an element of $\hat h(V_{H'})$.
\end{proof}

A similar line of argument gives the following.

\begin{lemma}
If $h$ and $k$ are vertex disjoint, then so are $\hat h$ and $\hat k$.
\end{lemma}
\begin{proof}
We prove the contrapositive.
Suppose $\hat h(V_{H'}) \cap \hat k(V_{K'})$ is inhabited and let $\hat h(v') = \hat k(w')$ be a vertex in this intersection.
Since $\beta \colon H \to H'$ is active, by \cite[Proposition 2.3]{HRY-mod1}, there exists a unique vertex $v \in V_H$ so that $\beta_{v} \colon M_{v} \rightarrowtail H'$ (with $[\beta_v] = \beta_1(v)$) has $v'$ in its image.
Likewise, there exists a unique vertex $w\in V_K$ so that $\gamma_w \colon N_w \rightarrowtail K'$ has $w'$ in its image.
We have $\alpha_1(f(v)) = [f' \circ \beta_v]$ and $\alpha_1(g(w)) = [g' \circ \gamma_w]$ by Remark 1.45 of \cite{HRY-mod1}.
\[ \begin{tikzcd}[column sep=tiny]
M_v \ar[dr, tail, "\beta_v"'] \ar[rr,"\alpha_{f(v)}", tail] &  & L' & 
N_w \ar[dr, tail, "\gamma_w"'] \ar[rr,"\alpha_{g(w)}", tail] &  & L' & 
\\
& H' \ar[ur, tail, "f'"'] & & 
& K' \ar[ur, tail, "g'"'] & & 
\end{tikzcd} \]
Now $f'(v')$ and $g'(w')$ are in the images of both $\alpha_1(f(v))$ and $\alpha_1(g(w))$; since $\alpha$ is a graphical map, \cref{def graphical map}\eqref{old graph def vertices} gives $f(v) = g(w)$.
Hence $h(v) = \ell f(v) = \ell g(w) = k(w)$, so $h$ and $k$ are not vertex disjoint.
\end{proof}

\begin{theorem}\label{thm oldnew graph map}
If $\varphi = (\varphi_0, \varphi_1) \colon G \to G'$ is a graphical map, then $\oldnew\varphi = (\varphi_0, \hat\varphi)$ is a new graph map, where $\hat \varphi$ is as defined in \cref{cns old to new}.
\end{theorem}
\begin{proof}
We verify the four conditions from \cref{def new graph map}.
Condition \eqref{new graph def edges} follows from Lemma 2.2 of \cite{HRY-mod1}.
The preceding two lemmas give \eqref{new graph def union} and \eqref{new graph def intersect}, respectively.
The interpretation of \eqref{new graph def boundary} given in \cref{rem graph def boundary} is Lemma 1.47 of \cite{HRY-mod1}.
\end{proof}

\begin{theorem}\label{old new equivalence}
Suppose $G$ and $G'$ are connected graphs.
The assignments $\oldnew$ and $\newold$ are inverse bijections between new graph maps from $G$ to $G'$ and graphical maps from $G$ to $G'$.
Moreover, these assignments identify the category of connected graphs and new graph maps with the category $\gcat$.
\end{theorem}
\begin{proof}
The conclusion will follow once we establish the bijection, as \cref{lem functoriality of hatting} implies that $\oldnew$ is an identity-on-objects functor from $\gcat$ to the category of new graph maps.

It is nearly immediate that $\newold\oldnew$ is an identity on $\gcat(G,G')$, by the same reasoning as in the first bullet point of Remark 1.45 of \cite{HRY-mod1}.
Namely, if $\varphi_v \colon H_v \rat G'$ represents $\varphi_1(v)$, then we have a factorization in $\gcat$
\[ \begin{tikzcd}
\medstar_v \rar[-act] \dar[tail,"\iota_v"] &  H_v \dar[tail,"\varphi_v"] \\
G \rar{\varphi}  & G'
\end{tikzcd} \]
so $\newold\oldnew(\varphi)_1(v) = \hat \varphi([\iota_v]) = [\varphi_v] = \varphi_1(v)$.

Now suppose that $\varphi = (\varphi_0, \hat \varphi) \colon G \to G'$ is a new graph map, $\newold(\varphi) = (\varphi_0, \varphi_1)$ is the associated graphical map, and $\oldnew\newold(\varphi) = (\varphi_0, \tilde \varphi)$, where $\tilde \varphi$ is a function from $\emb(G)$ to $\emb(G')$.
Let $f \colon H \rightarrowtail G$ be an arbitrary embedding; we wish to show that $\tilde \varphi [f] = \hat \varphi [f]$.
Notice that we already have a special case: 
if the domain of $f$ is a star, then $[f] = [\iota_v]$ for some $v \in V_G$, and we know $\tilde \varphi [\iota_v] \coloneqq \newold(\varphi)_1(v) \coloneqq \hat \varphi [\iota_v]$.
Let us turn the general case.

Since $\varphi$ and $\oldnew\newold(\varphi)$ are new graph maps, by \cref{rem graph def boundary}, we have the commutative diagram
\[ \begin{tikzcd}
\eth(\hat\varphi [f]) \dar[hook] & \eth([f]) \dar[hook] \lar[dashed, "\cong" swap] \rar[dashed, "\cong"] &  \eth(\tilde\varphi [f]) \dar[hook] \\
A_{G'} & A_G \lar[swap]{\varphi_0} \rar{\varphi_0} & A_{G'}
\end{tikzcd} \]
and we conclude that $\eth(\hat\varphi [f]) = \eth(\tilde\varphi [f])$.
This is nearly enough to conclude that $\hat\varphi [f]$ and $\tilde\varphi [f]$ are equal --- the only thing that could go wrong is if one of them is an edge while the other is not.
Applying \cref{prop sum splitting new graph} twice we have
\[
	\varsigma \hat \varphi [f] = \sum_{v\in V_H} \varsigma \hat \varphi [f\iota_v] = \sum_{v\in V_H} \varsigma \tilde \varphi[f\iota_v] = \varsigma \tilde \varphi [f]
\]
where the middle equality comes from the previously established fact that $\hat \varphi [f\iota_v] = \tilde \varphi [f\iota_v]$.
Since an embedding $h$ is an edge if and only if $\varsigma[h] = 0 $, we conclude that $\hat \varphi [f]$ is an edge if and only if $\tilde \varphi [f]$ is an edge.
By \cref{mod1 prop 1.25}, we conclude that $\hat \varphi[f] = \tilde \varphi[f]$.
Thus $\varphi = \oldnew\newold(\varphi)$.
\end{proof}

\begin{remark}\label{rmk factorization}
We can describe the active-inert factorization system on $\gcat$ from \cref{prop U ofs} in terms of new graph maps.
The active maps are precisely those maps $\varphi \colon G \ract G'$ so that $\hat \varphi([\id_G]) = [\id_{G'}] \in \emb(G')$, that is, maps where $\hat \varphi$ preserves the top element.
Likewise, the inert maps are those where $\hat \varphi$ admits a dashed arrow making the following diagram commute:
\[ \begin{tikzcd}
V_G \rar[hook] \dar[dashed] & \emb(G) \dar{\hat \varphi} \\
V_{G'} \rar[hook] & \emb(G').
\end{tikzcd} \]
The factorization of a map $\varphi \colon G \to G'$ can be obtained by taking $h\colon H \rat G'$ to be an embedding with $\hat \varphi [\id_G] = [h]$, and explicitly defining an active map $\psi \colon G \ract H$ so that $\varphi = h \circ \psi$.
This requires a bit of delicacy, since the embedding $h$ might not be injective on arcs.
\end{remark}

\section{Directed graphs and the wheeled properadic graphical category}
\label{sec directed}

We now turn our attention to the wheeled properadic graphical category.
This category controls wheeled properads \cite{MarklMerkulovShadrin} (which extend the notion of properad \cite{Vallette:KDP} to encode traces)
in the sense that there is a relevant nerve theorem.
Wheeled properads can be regarded as modular operads with additional directional structure, which was observed in \cite[Example 1.29]{Raynor:DLCSM}.

In this section we show how the wheeled properadic graphical category can be described in a similar manner to that of $\gcat$ above.
One drawback of the existing definition is that it relies on the notion of wheeled properad and maps between wheeled properads.
Our new definition does not have this requirement, though of course we will need to consider wheeled properads to prove the equivalence.
As one caveat: in this paper we consider only the version of the wheeled properadic graphical category from \cite{HRYfactorizations}, rather than the one from \cite{HRYbook}.
The latter can be obtained from the former by inverting a single map.
For brevity, we do not include the original definition of the wheeled properadic graphical category here, except indirectly in the proof of \cref{thm oriented wheeled equiv}.

\subsection{Directed graphs}
The following encoding of a directed graph with loose ends was introduced as Definition 1.1.1 of \cite{Kock:GHP}.

\begin{definition}[Directed graphs]\label{def dir graph}
Let $\mathscr{G}$ denote the category
\[
\begin{tikzcd}
\mathtt{e} & \mathtt{i} \dar \lar \\
\mathtt{o} \uar \rar & \mathtt{v}.
\end{tikzcd}
\]
A \mydef{directed graph} $G$ is a functor $\mathscr{G} \to \finset$ which sends the two maps with codomain $\mathtt{e}$ to monomorphisms, that is, a directed graph is a diagram of finite sets of the form
\[
\begin{tikzcd}[column sep=small]
	E_G & I_G \rar\lar[hook'] & V_G & O_G \rar[hook] \lar & E_G.	
\end{tikzcd}
\]
Each vertex $v\in V_G$ has a set of inputs $\inp(v) \subseteq I_G$ which is the preimage of $v$ under $I_G \to V_G$, and likewise has a set of outputs $\out(v) \subseteq O_G$.
The set of inputs of $G$ is defined to be $\inp(G) \coloneqq E_G \setminus O_G$ and the set of outputs of $G$ is $\out(G) \coloneqq E_G \setminus I_G$.
\end{definition}

Directed graphs in this sense turn out to be equivalent to others in the literature (this is mostly due to Batanin--Berger; see \cref{prop directed graph iso classes} and \cite[Remark 2.0.2]{ChuHackney}), with the caveat that it is not possible to encode graphs with components that are nodeless loops.
See \cref{def ext dir graph} below for a `hack' that fixes this problem, which is similar to that appearing in Definition 4.1 of \cite{HRY-mod1}.

A key upside to Kock's definition compared to other formalisms is that it is admits a straightforward notion of \'etale maps \cite[1.1.7]{Kock:GHP} as certain natural transformations, similar to that of \cref{def etale embed undirected}.
\begin{definition}\label{def dir embedding}
An \mydef{\'etale map} $f\colon G \to H$ between directed graphs is a commutative diagram of sets
\[ \begin{tikzcd}[sep=tiny]
E_G \ar[dd] & & I_G \ar[ll,hook'] \ar[dr] \ar[dd] \\
& O_G \ar[ul,hook'] \ar[rr,crossing over] & & V_G \ar[dd] \\
E_H & & I_H \ar[ll,hook'] \ar[dr] \\
& O_H \ar[ul,hook'] \ar[rr] \ar[from=uu, crossing over] & & V_H 
\end{tikzcd} \qquad \begin{tikzcd}[column sep=small]
E_G  \dar[swap]{f|_{E}} & I_G \rar\lar[hook']  \dar \ar[dr, phantom, "\lrcorner" very near start] & V_G  \dar& O_G \rar[hook] \lar  \dar \ar[dl, phantom, "\llcorner" very near start]& E_G \dar{f|_{E}} \\
E_H & I_H \rar\lar[hook'] & V_H & O_H \rar[hook] \lar & E_H	
\end{tikzcd} \]
so that the middle two squares are pullbacks.
An \mydef{embedding} is an \'etale map between \emph{connected} directed graphs which is injective on vertices.
\end{definition}

These embeddings of directed graphs correspond to those for undirected graphs in \cref{def etale embed undirected}. 
They are different from the `open subgraph inclusions' of \cite{Kock:GHP}, since the edge map $f|_E$ is not required to be injective and the domain and codomain must be connected.

We now relate the definitions of directed and undirected graphs.

\begin{definition}[Orientation presheaf]\label{def or presheaf}
The \mydef{orientation presheaf}, denoted $\opshf$, is the $\gcat$-presheaf so that
\[
\opshf_G \subseteq \prod_{a\in A_G}
	\{ +1, -1 \} 
\]
consists of those elements $x = (x_a)$ satisfying $x_a = -x_{a^\dagger}$ for all $a\in A_G$.
If $\varphi \colon H \to G$ is a graphical map and $x\in \opshf_G$, then $\varphi^*(x) \in \opshf_H$ is the element with 
\[
	\varphi^*(x)_a = x_{\varphi_0(a)}
\]
for all $a\in A_H$.
\end{definition}

Pairs $(G,x)$ consisting of a connected undirected graph $G \in \gcat$ and an element $x\in \opshf_G$ (equivalently a map of presheaves $x \colon G\to \opshf$) coincide with the connected directed graphs of \cref{def dir graph}.
This is a special case of \cite[1.1.13]{Kock:GHP}, and in order to fix conventions the following construction exhibits one direction of the correspondence; the reverse direction sends a directed graph to an undirected graph with $D = I \amalg O$ and $A = E\amalg E$. 
See also \cref{prop dir graph equivs}.

\begin{construction}\label{cnst gcato to kock}
If $G$ is a graph and $x\in \opshf_G$, then we have a splitting of $A_G$ as 
\begin{align*}
A_G^+ &= \{ a \in A_G \mid x_a = +1\} \\
A_G^- &= \{ a \in A_G \mid x_a = -1 \}
\end{align*}
and likewise for $D_G = D_G^+ \amalg D_G^-$. 
We thus obtain a diagram of sets
\[ \begin{tikzcd}[column sep=small]
A_G^- & \lar[hook'] D_G^- \rar & V_G & D_G^+ \lar \rar[hook] & A_G^+.
\end{tikzcd} \]
There are isomorphisms $A_G^- \xrightarrow{\cong} E_G \xleftarrow{\cong} A_G^+$, each sending an arc $a$ to the edge $[a,a^\dagger]$ it spans, so we conclude that $(G,x)$ determines a diagram
\[ \begin{tikzcd}[column sep=small]
E_G & \lar[hook'] D_G^- \rar & V_G & D_G^+ \lar \rar[hook] & E_G, 
\end{tikzcd} \]
as in \cref{def dir graph}.
\end{construction}

\begin{remark}\label{rmk inp out plus minus}
When working with this presentation, we have, for a vertex $v \in V_G$
\begin{align*}
	\inp(v) = \nbhd(v) \cap D_G^- &\subseteq D_G^-\\
	\out(v) = \nbhd(v) \cap D_G^+ &\subseteq D_G^+.
\end{align*}
We prefer to think about $\inp(G)$ and $\out(G)$ as subsets of $A_G^+$ and $A_G^-$, rather than as subsets of $E_G$:
\begin{align*}
	\inp(G) = \eth(G) \cap A_G^+ &\subseteq A_G^+\\
	\out(G) = \eth(G) \cap A_G^- &\subseteq A_G^-.
\end{align*}
Notice the discrepancy in signs.
This is because arcs in $\nbhd(v)$ point \emph{towards} the vertex while arcs in $\eth(G)$ point \emph{away from} the graph.
We think about input edges as coming in from above and output edges as coming in from below, so arcs of $\inp(v)$ point \emph{down} while arcs of $\inp(G)$ point \emph{up}.
\end{remark}

Every \'etale map of directed graphs induces an \'etale map of the corresponding undirected graphs (this was also mentioned in \cite[1.1.13]{Kock:GHP}). This restricts to embeddings, yielding the following equivalence.
\begin{proposition}\label{prop dir graph equivs}
There is an equivalence of categories between the category of connected directed graphs and embeddings (in the sense of \cref{def dir graph,def dir embedding}) and the category $\gcatemb \times_{\widehat{\gcat}} \widehat{\gcat}_{/\opshf} \eqqcolon \gcatembo$. \qed
\end{proposition}

This proposition can't be promoted as is to the category of graphs and \'etale maps, since \'etale maps between connected graphs which are not embeddings do not yield maps in $\gcat$. 

\begin{definition}[Inputs and outputs of embeddings]\label{def inout emb}
Suppose $G$ is a directed graph.
Then there are functions $\inp \colon \emb(G) \to \wp(A_G^+) \cong \wp(E_G)$ and $\out \colon \emb(G) \to \wp(A_G^-) \cong \wp(E_G)$ defined by \begin{align*} \inp([f]) &= \eth([f]) \cap A_G^+ \\
\out([f]) &= \eth([f]) \cap A_G^-. \end{align*}
Alternatively, we can consider $f\colon (H,f^*x) \rightarrowtail (G,x)$ as an embedding between directed graphs and we have
\begin{align*}
\inp(f) &= f(\inp(H)) \subseteq A_G^+ \\
\out(f) &= f(\out(H)) \subseteq A_G^-.
\end{align*}
\end{definition}

\subsection{Wheeled properads and the oriented graphical category}\label{subsec wheeled properads and oriented}
We now introduce the oriented graphical category $\gcato$, which we will show is equivalent to the wheeled properadic graphical category.
Thus, by the nerve theorem (Theorem 10.33 and Proposition 10.35 of \cite{HRYbook}), this category governs wheeled properads.
More precisely, the category of (set-based) wheeled properads is equivalent to the category of Segal presheaves on $\gcato$ (in the sense of \cref{def segal}).

\begin{definition}[Oriented graphical category]\label{def oriented graphical cat}
The \mydef{oriented graphical category} is the category $\gcato \coloneqq \gcat \times_{\widehat{\gcat}} \widehat{\gcat}_{/\opshf}$ whose objects are pairs consisting of an undirected graph $G$ and map of presheaves $x \colon G \to \opshf$ (equivalently an element $x \in \opshf_G$).
\end{definition}

\begin{remark}\label{remark pre-discrete fibration}
By \cref{def or presheaf}, if $G$ is a graph, then an element $x\in \opshf_G$ is the same thing as an involutive map from $A_G$ to $\{ +1, -1 \}$.
A map $(G,x) \to (G',x')$ in $\gcato$ is a just a map $G\to G'$ in $\gcat$ so that the triangle
\[ \begin{tikzcd}[column sep=tiny]
A_G \ar[rr,"\varphi_0"] \ar[dr] & & A_{G'} \ar[dl] \\ 
& \{ + 1 , -1 \}
\end{tikzcd} \] 
commutes.
\end{remark}

The following uses the two functions appearing in \cref{def inout emb} and the correspondence from \cref{cnst gcato to kock} to reinterpret \cref{def new graph map} in the oriented setting.
\begin{proposition}\label{prop oriented graph cat}
A map $\varphi \colon (G,x) \to (G',x')$ in the oriented graphical category $\gcato$ may be described as a pair $(\bar\varphi_0, \hat\varphi)$ consisting of functions $\bar\varphi_0 \colon E_G \to E_G'$ and $\hat \varphi \colon \emb(G) \to \emb(G')$ so that \eqref{new graph def edges}, \eqref{new graph def union}, and \eqref{new graph def intersect} from \cref{def new graph map} hold and, additionally,  
{\rm
\begin{enumerate}[label=(\roman*'),ref=\roman*',start=4]
	\item the diagram
\[ \begin{tikzcd}
\mathbb{N}E_G  \dar{\mathbb{N} \bar \varphi_0 } & \emb(G) \rar{\out}\lar[swap]{\inp} \dar{\hat \varphi}& \mathbb{N}E_G \dar{\mathbb{N} \bar \varphi_0 }\\
 \mathbb{N}E_{G'} & \emb(G') \rar{\out}\lar[swap]{\inp} & \mathbb{N}E_{G'}
\end{tikzcd} \]
commutes. \label{oriented graph map boundary} 
\hspace{\fill} \qedsymbol
\end{enumerate} }
\end{proposition}

\begin{definition}\label{def dendroidal category}
The Moerdijk--Weiss \mydef{dendroidal category}, denoted $\mathbf{\Omega}$, is the full subcategory of $\gcato$ consisting of those graphs $G$ so that $\out(v)$ is a singleton set for every $v\in V_G$.
Alternatively, it consists of those graphs $G$ so that $\out \colon \emb(G) \to \mathbb{N}E_G$ factors through $E_G \hookrightarrow \mathbb{N}E_G$.
\end{definition}

\begin{remark}\label{rem broad poset}
The equivalence of the usual definition of $\mathbf{\Omega}$ from \cite[\S3]{MoerdijkWeiss:DS} with \cref{def dendroidal category} follows, for instance, from \cref{cor sc equiv} and \cite[Remark 6.55]{HRYbook}.
But this may also be seen more directly by comparing \cref{prop oriented graph cat} with the presentation of the dendroidal category as the \emph{dendroidally ordered broad posets} from \cite[\S2.1.4]{Weiss:FODS}.
Indeed, it turns out that for each $G\in \mathbf{\Omega}$, we have $\out \times \inp \colon \emb(G) \hookrightarrow E_G \times \mathbb{N}E_G$ is an inclusion, giving a heterogeneous relation between $E_G$ and $\mathbb{N}E_G$. 
This relation is the relevant broad poset structure on $E_G$, and condition \eqref{oriented graph map boundary} is merely asserting that $\bar\varphi_0 \colon E_G\to E_G'$ is a map of broad posets.
\end{remark}

Our goal in this subsection is to prove the following theorem.

\begin{theorem}\label{thm oriented wheeled equiv}
The oriented graphical category $\gcato$ is equivalent to the wheeled properadic graphical category which excludes the nodeless loop.
\end{theorem}

This version of the wheeled properadic graphical category was called $\mathcal{B}$ in \cite[\S2]{HRYfactorizations}.
The nodeless loop in question shows up as \cref{ex dir nodeless loop}, but does not appear among the graphs defined in \cref{def dir graph}.
An extension of the preceding theorem will be given in \cref{thm oriented wheeled equiv extended}.

\begin{remark}[A change to objects]\label{remark no listings}
The objects of the wheeled properadic category from \cite{HRYbook,HRYfactorizations} are (isomorphism classes of) graphs equipped with a \emph{listing}, that is, a total ordering on each of the subsets $\inp(v)$, $\out(v)$ (both as $v$ varies), $\inp(G)$, and $\out(G)$ of $E_G$.
Every graph admits such a listing, and if we consider two choices of listing for the same graph, then there is a canonical isomorphism between them in the wheeled properadic graphical category.
In what follows, we consider only the equivalent category whose objects coincide with the objects of $\gcato$ (resp., with the objects of $\egcato$ in \cref{subsec extended oriented}).
That is, we are considering all directed graphs rather than just isomorphism classes of such, and we are not imposing a listing.
The equivalence can be realized by choosing, for each directed graph $G \in \gcato$, some fixed listing and then passing to the associated isomorphism class.
The term `wheeled properadic graphical category' will refer to this equivalent category henceforth.
\end{remark}

We will need a few preliminary results before tackling the proof of \cref{thm oriented wheeled equiv}.
The main one is the following, which is the directed analogue of the construction from \S2.2 of \cite{HRY-mod2}.
As this is just a minor variation of the undirected case, we merely provide a sketch and refer the reader to \cite{HRY-mod2} for details.

\begin{proposition}\label{prop functor to wproperad}
There is a functor $F \colon \gcato \to \wproperad$ from the oriented graphical category to the category of (colored) wheeled properads \cite[Definition 9.1]{HRYbook} which takes a directed graph to the wheeled properad it freely generates.
\end{proposition}
\begin{proof}[Proof sketch]
On objects the value of $F$ is given in the same was as in \cite[\S9.2.1]{HRYbook}. 
Briefly, if $G$ is a directed graph then $F(G)$ is the free $E_G$-colored wheeled properad which has one generator for each vertex $v\in V_G$.
For each vertex $v$, choose orderings $\inp(v) = e_1, \dots, e_m$ and $\out(v) = e_1', \dots, e_n'$ for the inputs and output edges of $v$, and regard $v$ as a generator in the set
\begin{equation}\label{eq v location}
	v\in F(G)(e_1, \dots, e_m; e_1', \dots, e_n').
\end{equation}
If $H$ is a graph equipped with a total ordering of the sets $\inp(H) = \{i_1, \dots, i_k\}$ and $\out(H) = \{o_1, \dots, o_j\}$, then an \'etale map $\mu \colon H \to G$ determines an element in 
\[
	F(G)(\mu i_1, \dots, \mu i_k; \mu o_1, \dots, \mu o_j).
\]

Given a graphical map $\varphi \colon G \to G'$, using that $F(G)$ is a free wheeled properad we can specify a map of wheeled properads (see \cite[Lemma 9.23]{HRYbook})
\[
	F\varphi = f \colon F(G) \to F(G')
\]
by defining the color map $E_G \to E_{G'}$ to be $\bar \varphi_0$, and value on the generator \eqref{eq v location},
\[
	f_1(v) \in F(G)(fe_1, \dots, fe_m; fe_1', \dots, fe_n'),
\]
is given by the embedding $\varphi_1(v)$.
Composition is defined using graph substitution, and associativity of composition is proved as in \cite[Proposition 2.25]{HRY-mod2}.
\end{proof}

In \cref{lem emb subgraph} we will compare embeddings with the notion of `subgraph' from Definition 9.50 of \cite{HRYbook}, and to do so it will be convenient to have the following notion of `graph complement' at hand.
We show that such graph complements always exist in \cref{lem graph complement} (for completeness, we also prove they are unique in \cref{lem complements unique}).
See \cref{fig graph complement} for a pair of examples.

\begin{definition}\label{def graph compl}
Let $f \colon G \rightarrowtail H$ be an embedding of (undirected or directed) connected graphs (in $\gcatemb$ or $\gcatembo$).
A \mydef{graph complement} of $f$ consists of a triple $(K,v^G,\alpha)$ where $K$ is another connected graph, $v^G \in V_K$ is a distinguished vertex, and $\alpha \colon K \ract H$ is an active map.
This data must satisfy $\alpha_1(v^G) = [f]$ and, for all other vertices $v \in V_K\setminus \{v^G\}$, there is a vertex $w\in V_H$ with $\alpha_1(v) = [\iota_w]$.
\end{definition}

It is implicit that if $f$ is an embedding of directed graphs, then $K$ should be directed and $\alpha$ should be an active map in $\gcato$.
In either case, since $\alpha$ is active, every vertex $w$ of $H$ is either in $f(V_G)$ or comes from a unique $v\in V_K \setminus \{v^G\}$ where $\alpha_1(v) = [\iota_w]$.
Thus $\alpha$ induces a bijection $V_K \setminus \{v^G\} \cong V_H \setminus f(V_G)$.

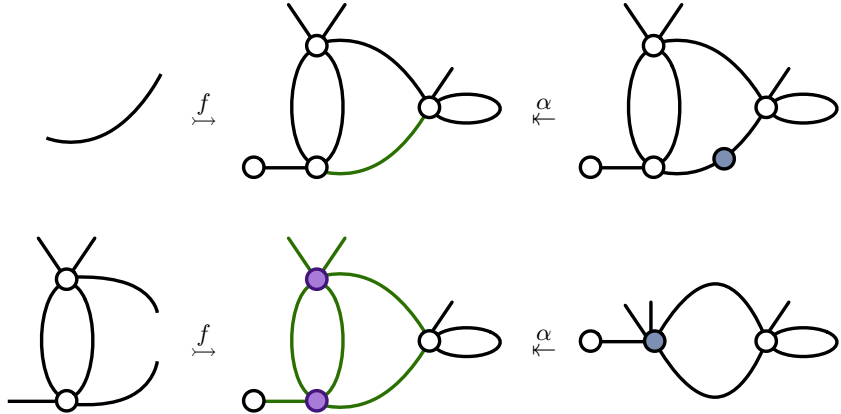
\begin{figure}[tbh]
\labellist
\small\hair 2pt
 \pinlabel {$f$} [ ] at 185 285
 \pinlabel {$\rat$} [ ] at 185 270
 \pinlabel {$\alpha$} [ ] at 486 285
 \pinlabel {\reflectbox{$\ract$}} [ ] at 486 270
 \pinlabel {$f$} [ ] at 185 80
 \pinlabel {$\rat$} [ ] at 185 65
 \pinlabel {$\alpha$} [ ] at 486 80
 \pinlabel {\reflectbox{$\ract$}} [ ] at 486 65
\endlabellist
\centering
\includegraphics[width=0.9\textwidth]{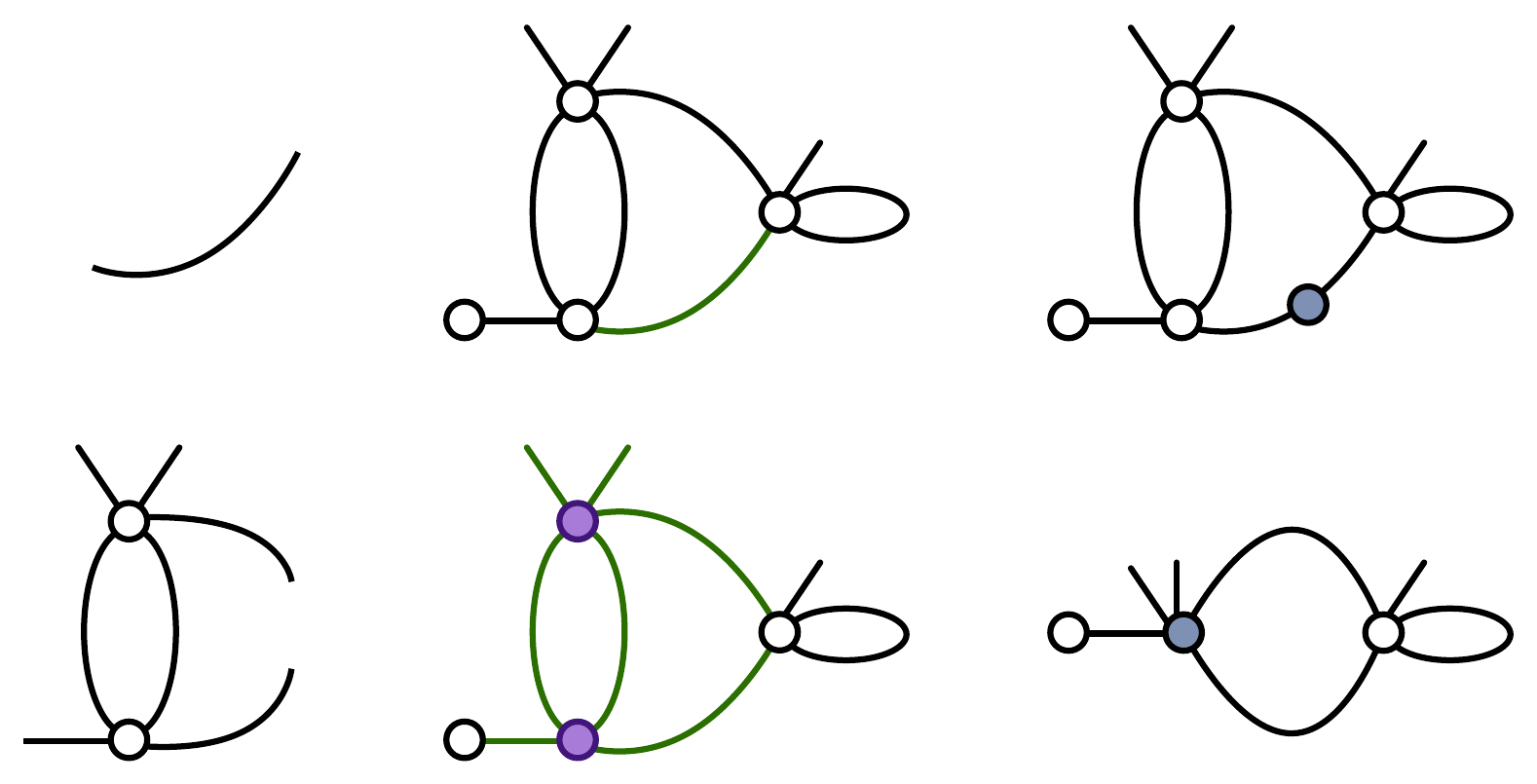}
\caption{Two examples of graph complement}
\label{fig graph complement}
\end{figure}

\begin{remark}[Relation to \'etale complement]
By deleting the vertex $v^G$ from $K$ (retaining all of the arcs, but removing $\nbhd(v^G)$ from $D_K$), we obtain another graph $K'$ which comes equipped with an \'etale map $K' \to H$.
The graph $K'$ will frequently be disconnected.
If the embedding $f \colon G \rat H$ is injective on arcs, then we could ask what the relationship is between $K'$ and the \emph{\'etale complement} of $f$ from \cite[1.6.2--1.6.4]{Kock:GHP}. 
These turn out to coincide if either $G$ contains at least one vertex or if $f$ classifies an internal edge of $H$.
If $f$ classifies an incoming or outgoing edge of $H$, then $K' \cong {\exedge} \amalg H$, whereas the \'etale complement is $H$.
\end{remark}

The following lemma is not strictly necessary for the developments below, but is included for completeness.

\begin{lemma}\label{lem complements unique}
Graph complements are unique up to isomorphism.
\end{lemma}
\begin{proof}
Suppose $(K, v^G, \alpha)$ and $(J, u^G, \beta)$ are two graph complements of an embedding $f \colon G \rightarrowtail H$.
The isomorphisms
\[
	V_K \setminus \{v^G \} \xrightarrow{\cong} V_H \setminus f(V_G) \xleftarrow{\cong} V_J \setminus \{u^G \}
\]
extend to a bijection $z_V \colon V_K \to V_J$ sending $v^G$ to $u^G$.

By \cref{def graphical map}\eqref{old graph def vertices} we have
\[ \begin{tikzcd}
\nbhd(v) \rar[dashed,"\cong"] \dar["\dagger",hook] & \eth(\alpha_1(v)) = \eth(\beta_1(zv)) \dar[hook] & \nbhd(zv) \lar[dashed,"\cong"'] \dar["\dagger",hook] \\
A_K \rar{\alpha_0} & A_H & A_J \lar[swap]{\beta_0}
\end{tikzcd} \]
and these isomorphisms $\nbhd(v) \to \nbhd(z_V(v))$ assemble into a bijection
\[
	z_D \colon D_K = \coprod_{v\in V_K} \nbhd(v)  \to \coprod_{u \in V_J} \nbhd(u) = D_J
\]
over $z_V \colon V_K \cong V_J$.
Finally, since $\alpha$ and $\beta$ are active we have isomorphisms
\[ \begin{tikzcd}
\eth(K) \rar[dashed,"\cong"] \dar[hook] & \eth(H) \dar[hook] & \eth(J) \lar[dashed,"\cong"'] \dar[hook] \\
A_K \rar{\alpha_0} & A_H & A_J \lar[swap]{\beta_0}
\end{tikzcd} \]
which gives 
\[
	z_A \colon A_K = D_K \amalg \eth(K) \xrightarrow{\cong} D_J \amalg \eth(J) = A_J.
\]
If $a \in \nbhd(v)$ (resp.\ $a \in \eth(K)$) then $z_A(a) = z_D(a) \in \nbhd(z_V(v))$ (resp.\ $z_A(a) \in \eth(J)$) is the unique element satisfying $\alpha_0(a) = \beta_0(z_A(a))$.

The only thing to check is that $z_A\colon A_K \to A_J$ preserves the involution.

If $G$ is not an edge, then $\alpha_0$ and $\beta_0$ are injective by Lemma 2.9 of \cite{HRY-mod1}, so $z_A(a)$ is defined by the equation $\alpha_0(a) = \beta_0(z_A(a))$ with no stipulation about where $a$ or $z_A(a)$ live.
Using this equation for both $a$ and $a^\dagger$, we have
 \[ \beta_0(z_A(a^\dagger)) = \alpha_0(a^\dagger) = \alpha_0(a)^\dagger = \beta_0(z_A(a))^\dagger = \beta_0(z_A(a)^\dagger),\]
hence $z_A(a^\dagger) = z_A(a)^\dagger$.

If $G$ \emph{is} an edge, then again using Lemma 2.9 of \cite{HRY-mod1} we have $\alpha_0$ is injective when restricted to $A_K \setminus \nbhd(v^G)$.
Writing $\nbhd(v^G) = \{d_0,d_1\}$, the only identifications that $\alpha_0$ makes are $\alpha_0(d_0) = \alpha_0(d_1^\dagger)$ and $\alpha_0(d_1) = \alpha_0(d_0^\dagger)$.
A similar situation occurs with $\beta_0$ and $\nbhd(u^G)$.
Thus, as in the previous paragraph, $z_A(a^\dagger) = z_A(a)^\dagger$ so long as $a\notin \{d_0,d_1,d_0^\dagger,d_1^\dagger\}$.
But now $\nbhd(u^G) = \{ z_A(d_0), z_A(d_1) \}$, and the only way for our bijection to not be involutive would be if $z_A(d_0^\dagger) = z_A(d_1)^\dagger$ and $z_A(d_1^\dagger) = z_A(d_0)^\dagger$.
This cannot happen, as
\[
	\beta_0 (z_A(d_0^\dagger)) = \alpha_0(d_0^\dagger) = \alpha_0(d_1)
\neq \alpha_0(d_1)^\dagger= \beta_0(z_A(d_1))^\dagger. \qedhere 
\]
\end{proof}

\begin{lemma}\label{lem graph complement}
If $f \colon G \rightarrowtail H$ is an embedding between connected (directed or undirected) graphs, then a graph complement of $f$ exists.
\end{lemma}
This lemma is closely related to the notion of a `\emph{shrinkable pasting scheme},' which includes the pasting scheme for connected directed graphs by Proposition 3.3(i) of \cite{HRY-shrinkability}.
One can iteratively shrink away (in the graph $H$) all edges coming from internal edges of $G$ and identify their end vertices, and this will yield the graph complement (so long as $G$ is not an edge).
See the discussion following Remark 3.8 of \cite{HRY-shrinkability} for more on this viewpoint; in the following proof we simply give formulas for the resulting graph $K$.
\begin{proof}
We can ignore directionality when constructing the graph complement: if the embedding $f$ is a map in $\gcato$, then the produced undirected graph $K$ becomes directed via the composite map of presheaves $K \to H \to \opshf$.

We first address the special case when $G$ does not have a vertex.
If $G$ is an edge hitting arcs $a_0, a_1 \in A_H$ (with $a_0^\dagger = a_1$), we set $V_K = V_H \amalg \{ v^G \}$, $\nbhd(v^G)$ consists of two new elements $d_0$ and $d_1$, $D_K = D_H \amalg \nbhd(v^G)$, $A_K = A_H \amalg \{d_0,d_1\}$, and all structure the same except that $a_i^\dagger$ is redefined to be $d_i$ for $i=0,1$.
The definition of $\alpha \colon K \to H$ is forced by the requirements in \cref{def graph compl}.

Now suppose that $G$ has at least one vertex.
Define the following three sets
\begin{align*}
V_K &= (V_H \setminus f(V_G)) \amalg \{v^G\} \cong V_H / {\sim} \\
D_K &= D_H \setminus \{ f(d) \mid d, d^\dagger \in D_G \} \\
A_K &= A_H \setminus \{ f(d) \mid d, d^\dagger \in D_G \}.
\end{align*}
The quotient presentation for $V_K$ identifies all of the vertices of $f(V_G) \subseteq V_H$ to a single vertex $v^G$. We thus have a map $D_K \hookrightarrow D_H \to V_H \to V_H/{\sim} \cong V_K$. 
The set $A_K$ inherits an involutive structure from $A_H$.
This data defines a graph $K$.
Notice that $\eth(K) = A_K \setminus D_K = A_H \setminus D_H = \eth(H)$.

There is a clear candidate for a graphical map $\alpha \colon K \to H$ which is an inclusion on arcs and satisfies $\varphi_1(v) = [\iota_v]$ for $v\in V_H \setminus f(V_G)$ and $\varphi_1(v) = [f]$ for $v=v^G$. 
This satisfies the conditions from \cref{def graphical map}; the only thing to note is that $\nbhd(v^G) = f(\eth(G)^\dagger)$ which is used in showing \eqref{old graph def vertices}.
The triple $(K,v^G,\alpha)$ satisfies the conditions of \cref{def graph compl}.
\end{proof}

\begin{lemma}\label{lem emb subgraph}
The functor $F \colon \gcato \to \wproperad$ gives a bijection between embeddings $G \rat G'$ in $\gcatembo$ and subgraphs $F(G) \to F(G')$ in the sense of \cite[Definition 9.50]{HRYbook}.
\end{lemma}
\begin{proof}
Suppose $f \colon G \rightarrowtail G'$ is an embedding,
which determines a map of wheeled properads $F(f) \colon F(G)\to F(G')$ by \cref{prop functor to wproperad}.
Let $\alpha \colon K \ract G'$ be the graph complement of $f$ from \cref{lem graph complement}.
Since $\alpha$ is active, it induces an isomorphism $K\{G\} \cong G'$ by Definition 1.39 and Proposition 2.3 of \cite{HRY-mod1}; here $K\{G\}$ is the graph substitution of $G$ into the vertex $v^G$ of $K$ (see Construction 1.18 of \cite{HRY-mod1}).
By Theorem 9.52 of \cite{HRYbook}, $F(G) \to F(G')$ is a subgraph.

Now suppose that $f \colon F(G) \to F(G')$ is a map of wheeled properads which is a subgraph.
By Theorem 9.52 of \cite{HRYbook} there is a graph substitution decomposition $K\{G\} \cong G'$ so that $f$ sends the edges and vertices in $G$ to their corresponding images in $K\{G\}$ through this isomorphism.
As we always have an embedding $G \rat K\{G\}$, we see that the composition $G \rat K\{G\} \cong G'$ is an embedding in $\gcato$. The two indicated constructions are inverses to one another.
\end{proof}

\begin{proof}[Proof of \cref{thm oriented wheeled equiv}]
Recall from \cref{remark no listings} that we are taking the objects of the wheeled properadic graphical category, denoted in this proof by $\mathbf{C}$, to be the same as the objects of $\gcato$.
This category comes with a faithful (but non-full) functor $\mathbf{C} \to \wproperad$ which has the same action on objects as the functor $F \colon \gcato \to \wproperad$ from \cref{prop functor to wproperad}. 
Our strategy is to exhibit a lift
\begin{equation}\label{eq whprpd triangle}
\begin{tikzcd}[column sep=tiny]
\gcato \ar[dr,"F"'] \ar[rr,dashed] & & \mathbf{C} \ar[dl] \\
 & \wproperad
\end{tikzcd} \end{equation}
which is the identity-on-objects, makes the triangle commute, and is fully-faithful.

Suppose $G$ and $G'$ are objects in $\gcato$, and let $f \colon F(G) \to F(G')$ be a map of wheeled properads.
The \emph{image} (Definition 9.55 of \cite{HRYbook}) of $f$ is the $G'$-decorated graph $(f_0G)\{f_1(u)\}$.
Namely, $f_1(u)$ is given by a connected graph $K_u$ together with an \'etale map $K_u \to G'$, and $(f_0G)\{f_1(u)\}$ is the (connected) graph substitution $G\{K_u\}_{u\in V_G}$ together with the induced \'etale map. 
An embedding $h\colon H \rat G$ determines a subgraph $F(H) \to F(G)$ by the equivalence of \cref{lem emb subgraph}, and there is an associated image of this subgraph given as the \'etale map 
\begin{equation}\label{image of H} H\{K_{hv}\}_{v\in V_H} \to G\{K_{u}\}_{u\in V_G} \to G'. \end{equation}

Leaning on the equivalence from \cref{lem emb subgraph}, by \cite[Theorem 9.62]{HRYbook} the map $f$ constitutes a morphism in $\mathbf{C}(G,G')$ if and only if for every embedding $h$ the \'etale map \eqref{image of H} is an embedding.
When $f= F\varphi$, by \cref{cns old to new} we have that \eqref{image of H} is an embedding representing $\hat \varphi [h]$, hence $F\varphi$ is in $\mathbf{C}(G,G')$.
On the other hand, if $f\in \mathbf{C}(G,G')$ is an map in the wheeled properadic graphical category, we can define a graphical map $\varphi \colon G \to G'$ by setting $\bar \varphi_0 = f_0 \colon E_G \to E_G$ and letting $\hat \varphi [h]$ be the embedding \eqref{image of H} (or using the special case when $h = \iota_v$ and letting $\varphi_1(v)$ be represented by the embedding $K_v \rat G'$ associated to $f_1(v)$).
Using properties of graph substitution, one checks that $(\bar \varphi_0, \hat \varphi)$ satisfies the conditions from \cref{prop oriented graph cat}, hence constitutes a morphism in $\gcato(G,G')$.

By looking at their actions on edges and vertices, we see that the two assignments $\varphi \mapsto F\varphi$ and $f \mapsto \varphi$ are inverses. 
Thus the desired identity-on-objects lift \eqref{eq whprpd triangle} of $F$ exists and is fully-faithful.
\end{proof}

\begin{remark}
An observant reader may have noticed that we glossed over a subtlety in the second paragraph of the preceding proof, namely the fact that our current formalism for graphs does not include \emph{nodeless loops} (\cref{ex dir nodeless loop}).
It is possible that either $K_u$ or one of the graph substitutions $G\{K_u\}$, $H\{K_{hv}\}$ is a nodeless loop, in which case one should be careful with the phrase \emph{\'etale map} (here it should mean a natural transformation satisfying \cref{def extended etale embedding}\eqref{item p-map} from \cref{sec extended}). 
This is not so important, but what is important is that we never regard the map \eqref{image of H} as an embedding when $H\{K_{hv}\}$ is a nodeless loop.
This follows from \cite[Theorem 9.52]{HRYbook} since $G'$ is not a nodeless loop, and also matches our later convention from \cref{def extended etale embedding}. 
\end{remark}

\section{Hypermoment categories, algebraic patterns, Segal presheaves}\label{sec hypermoment}
We now turn to the structure of graph categories that make them suitable for describing various kinds of operadic structures (operads, properads, cyclic operads, and so on).
This is all akin to how the inclusion of the simplicial indexing category $\mathbf{\Delta} \hookrightarrow \cat$ induces a fully-faithful functor $\cat \to \widehat{\mathbf{\Delta}}$, and we can identify its essential image as those simplicial sets $X$ satisfying a \emph{Segal condition} 
\[
	X_n \cong X_1 \times_{X_0} X_1 \times_{X_0} \cdots \times_{X_0} X_1
\]
(see \cref{def segal}).
We will situate our discussion within the (closely related) general frameworks of Berger's hypermoment categories \cite{Berger:MCO} and the algebraic patterns of Chu--Haugseng \cite{ChuHaugseng:HCASC}.

In this section, we will discuss all of the graph categories that appear in \cref{fig graph cat functors} of the introduction, though several such only make their appearances in appendices:
\begin{itemize}
\item The full subcategories $\gcatcyc \subset \gcatnought \subset \gcat$ consist of those undirected graphs which are  simply-connected, and graphs in $\gcatcyc$ have non-empty boundary. These categories are considered in \cref{sec tree cat}.
\item The properadic graphical category $\pgcat$ from \cref{def properadic gcat} (called $\mathbf{\Gamma}$ in \cite{HRYbook}), which is a non-full subcategory of $\gcato$ (see \cref{thm properadic gcat}) whose objects are acyclic directed graphs.
\item The dioperadic graphical category $\pgcatsc$ (called $\mathbf{\Theta}$ in \cite{HRYbook}), which is a full subcategory of both $\pgcat$ and $\gcato$ on the simply-connected graphs. 
\item The extended graphical category $\egcat$ and the extended oriented graphical category $\egcato$, which appear in \cref{sec extended}.
These include a single new object (up to isomorphism), the nodeless loop, which does not fit into our earlier graphs definition.
\end{itemize}
The appendices that the categories from the first three bullet points appear in are mainly devoted to comparing alternative descriptions of these categories, and these descriptions are not necessary for understanding this section.
In particular, the first and third bullet points give everything that's needed to think about $\gcatcyc$, $\gcatnought$, and $\pgcatsc$, while one can take \cref{def convex inclusion} and \cref{thm properadic gcat} as providing a definition of $\pgcat$.
A reader wanting a shortcut for the extended graphical categories can rely on \cref{facts about nodeless loop} and the statement of \cref{thm extended graph cat} for $\egcat$, and \cref{def extended oriented} for $\egcato$. 
Alternatively, the reader may elect to ignore all of these additional categories and focus on the particular cases of $\gcat$ and $\gcato$.

\begin{remark}[Factorization system on pointed finite sets]
Recall that the category of finite pointed sets, $\finset_*$, has an inert-active factorization system (see Remark 2.1.2.2 of \cite{Lurie:HA}).
A basepoint-preserving function $f \colon X \to Y$ is \mydef{inert} if $f^{-1}(y)$ has cardinality $1$ whenever $y$ is not the basepoint, and is \mydef{active} if the preimage of the basepoint of $Y$ is just the basepoint of $X$.

If $A$ is a finite set and $A_+$ is $A$ together with an added basepoint, then maps $A_+ \to B_+$ correspond to partial functions $A \rightsquigarrow B$.
Under this interpretation, active maps correspond to total functions $A\to B$. 
Inert maps are those which give a bijection when restricted to their domain of definition.
(In this interpretation, a partial map $A \rightsquigarrow B$ factors as $A \rightsquigarrow W \to B$, where $W\subseteq A$ is the domain of definition.)

It follows that $\finsetstarop$ has an active-inert factorization system; we also utilize a skeleton $\fstarop \subseteq \finsetstarop$ with objects $\lr{n} = \{1, \dots, n\}\amalg \{ \ast \}$ for $n\geq 0$.
The category $\fstarop$ is isomorphic to Segal's category $\Gamma$ \cite{Segal:CCT}.
Inert maps $\lr{1} \rat X$ in $\finsetstarop$ classify non-basepoint elements of $X$.
\end{remark}

\begin{definition}
Let $\CC$ be a category equipped with an active-inert factorization system and a functor $\gamma \colon \CC \to \finsetstarop$ which respects the factorization systems.
\begin{itemize}
	\item 
An object $u \in \CC$ is a \mydef{unit} just when
\begin{itemize}
\item $\gamma(u) \cong \lr{1}$, and
\item any active map with codomain $u$ has precisely one inert section.
\[ \begin{tikzcd}[column sep=small]
u \ar[rr, tail, dashed, "\exists !"] \ar[dr,"="'] & & c \ar[dl,-act] \\
& u
\end{tikzcd} \]
\end{itemize}
\item A \mydef{nilobject} is an object $c\in \CC$ with $\gamma(c) \cong \lr{0}$.
\end{itemize}
\end{definition}

The following is from \cite[Definition 3.1]{Berger:MCO}, and would be called a \emph{unital} hypermoment category there.

\begin{definition}
A \mydef{hypermoment category} consists of a category $\CC$, an active-inert factorization system on $\CC$, and a functor $\CC \to \fstarop$ respecting the factorization systems.
These data must satisfy the following:
\begin{enumerate}
\item For each $c \in \CC$ and each inert map $\lr{1} \rat \gamma(c)$ in $\fstarop$, there is an essentially unique inert lift $u \rat c$ where $u$ is a unit.\label{def hm inert lifts}
\item For each $c\in \CC$, there is an essentially unique active map $u \ract c$ whose domain is a unit.\label{def hm active}
\end{enumerate}
\end{definition}

\begin{remark}
One may replace the skeleton $\fstarop$ in the preceding definition with the larger category $\finsetstarop$ by weakening what is meant by `lift' in \eqref{def hm inert lifts}, i.e.\ by requiring composition with the unique isomorphism $\lr{1} \cong \gamma(u) \rat \gamma(c)$ yields the original map $\lr{1} \rat \gamma(c)$.
We will generally use this more expansive version without comment, and only provide functors $\gamma \colon \CC \to \finsetstarop$.
This is fine, as the inclusion from $\fstarop$ to $\finsetstarop$ is an equivalence.
\end{remark}

Examples of hypermoment categories include the dendroidal category $\mathbf{\Omega}$ from \cref{def dendroidal category} and the properadic graphical category $\pgcat$ from \cref{def properadic gcat}, each equipped with the functor $G \mapsto (V_G)_+$ (see \cite[3.4]{Berger:MCO}, \cite[2.2.22--23]{ChuHackney}). 

\begin{example}
The graphical category $\gcat$ (or the extended version $\egcat$) is a hypermoment category.
We define a functor $\gamma \colon \gcat \to \finsetstarop$ which on objects sends $G$ to $(V_G)_+$.
There is a composite function (see \cref{def vertex sum})
\[ \begin{tikzcd}
t_{\varphi} \colon V_G \rar[hook] & \emb(G) \rar["\hat\varphi"] & \emb(G') \rar["\varsigma"] & \wp(V_{G'}) 
\end{tikzcd} \]
and the desired map of finite pointed sets
\[
	\gamma(\varphi)^\oprm \colon (V_{G'})_+ \to (V_G)_+
\]
sends all elements of $t_{\varphi}(v) \subseteq V_{G'}$ to $v\in V_G$, and all elements in $V_{G'} \setminus \bigcup_v t_{\varphi}(v)$ to the basepoint.
This map is singly-defined by \cref{def graphical map}\eqref{old graph def vertices} (which implies the sets $t_\gamma(v)$ are pairwise disjoint), and $\gamma$ is a functor since graphical maps respect the partial orders on embedding sets.

Given any graph $G$, there is an essentially unique active map from an $n$-star to $G$, where $n$ is the size of the boundary $\eth(G)$ (see \cref{ex edge star,ex graphical star}). 
On the other hand, each vertex $v$ determines an embedding $\iota_v \colon \medstar_v \rat G$ from a star, and any other map from an $n$-star classifying $v$ is isomorphic to this one.
It thus remains to show that the units of $\gcat$ are precisely the stars.

If $H$ is a graph which has a unique vertex and at least one internal edge, then there is the active map $\medstar_n \ract G$ from above, and this map does not have an inert section since embeddings send internal edges to internal edges.
Thus $H$ is not a unit despite having $\gamma(H) \cong \lr{1}$.
On the other hand, if $G$ is an arbitrary connected graph and $\varphi \colon G \ract \medstar_n$ is an active map, then since $\gcatnought\subset \gcat$ is a sieve (Proposition 5.2 of \cite{HRY-mod1}), we deduce that $G \in \gcatnought$ is a tree.
(As in \cref{rmk nodeless maps}, it is not possible for $G$ to be the nodeless loop and have a map to $\medstar_n$, hence $\varphi$ is a map in $\gcat$.)
By considering the generalized Reedy structure (Proposition 5.5 of \cite{HRY-mod1}) on $\gcatnought$, we can conclude that $\varphi \in \gcatnought^-$ and is a composition of codegeneracy maps.
By inspection, this map has a unique inert section, hence $\medstar_n$ is a unit. 

The nilobjects of $\gcat$ are the edges, while the nilobjects of $\egcat$ are the edges and the nodeless loops.
\end{example}

\begin{remark}[Other graph categories]\label{rem other graph cat}
Each of the other graph categories considered in this paper ($\gcatnought$, $\gcatcyc$, $\pgcatsc$ $\pgcat$, $\gcato$, and $\egcato$) admits a functor to $\egcat$ sending an object to the corresponding (undirected) graph, as we see from \cref{fig graph cat functors} of the introduction. 
Each of these becomes a hypermoment category by composing with $\gamma \colon \egcat \to \finsetstarop$.
In all cases the units are again stars. 
\end{remark}

The graphical category $\gcat$ is \mydef{strongly unital} in the sense of \cite[Definition 3.12]{Berger:MCO}.
This means that the full subcategory of $\gcatemb$ on the units and nilobjects (that is, the stars and the edge) is dense, or in other words, every graph is a canonical colimit (in $\gcatemb$) of its vertex stars and its edges.
The same is true for all of the other graph categories from \cref{rem other graph cat}.
\begin{remark}\label{remark nilobjects}
Though it is true that $\egcat$ and $\egcato$ are also strongly unital, the notion of Segal object we obtain from \cite[3.11]{Berger:MCO} is different from what appeared in \cite[Definition 4.11]{HRY-mod1}, \cite[Notation 3.11]{HRY-mod2}, and \cite[\S2.1]{HRYfactorizations} in the non-reduced case.
In particular, if $X$ is a strict Segal presheaf, in the second set of references the values of $X$ on edges and nodeless loops will coincide, while in the first reference they may be distinct.
This coincidence was necessary for establishing the nerve theorems for these categories (relating to modular operads, resp.\ wheeled properads).
\end{remark}

\begin{definition}
A hypermoment category $\CC$ is \mydef{extensional} if, given a unit $u$, an active map $u\ract d$, and an inert map $u\rat c$, the pushout 
\begin{equation*}\label{eq pushout extensional} \begin{tikzcd}
u \rar[-act] \dar[tail] \ar[dr, phantom, "\ulcorner" very near end] & d \dar[tail] \\
c \rar & d' 
\end{tikzcd} \end{equation*}
exists in $\CC$ with $d \rat d'$ inert, and is sent by $\gamma$ to a pushout in $\fstarop$.
\end{definition}

\begin{example}\label{example extensional}
The graphical category $\egcat$ is extensional. 
Indeed, suppose that $v \in G$ has valence $n$, and $H$ is a graph with $n$-element boundary set.
Then any active map $\medstar_v \ract H$ gives a bijection between $\eth(H)$ and $\nbhd(v) \cong \eth(\medstar_v)$, so permits us to define a graph substitution $G\{H\}$.
This now fits into the diagram
\[\begin{tikzcd}
\medstar_v \rar[-act] \dar[tail,"\iota_v"] \ar[dr, phantom, "\ulcorner" very near end] & H \dar[tail] \\
G \rar[-act] & G\{H\} 
\end{tikzcd} \]
where the bottom map is active and the right map is an embedding.
We have $V_{G\{H\}} \cong (V_G \setminus \{v \}) \amalg V_H$, and this allows one to show the preceding square is a pushout by producing explicit maps out of $G\{H\}$ using \cref{def graphical map,def extended graphical cat}.
Our square maps via $\gamma$ to the square
\[ \begin{tikzcd}
((V_G\setminus \{v\}) \amalg V_H)_+ \rar[-act] \dar[tail] \ar[dr, phantom, "\lrcorner" very near start] & (V_G)_+ \dar[tail] \\
(V_H)_+ \rar[-act] & \{ v\}_+
\end{tikzcd} \]
in $\finset_\ast$ which is a pullback of finite sets since there is a bijection
\[
	(V_G\setminus \{v\}) \amalg V_H \cong \big\{ (v',\ast) \mid v'\in V_G \setminus \{v\} \big\} \amalg \big\{ (v,w) \mid w\in V_H \big\}.
\]
The other categories of graphs from \cref{rem other graph cat} are extensional by a similar argument.
\end{example}

The following notion was introduced in \cite[Definition 2.1]{ChuHaugseng:HCASC}.
\begin{definition}\label{def alg pattern}
An \mydef{algebraic pattern} is an $\infty$-category $\PP$ equipped with an inert-active factorization system $(\PP_\intrm, \PP_\actrm)$, along with some specified full subcategory $\PP_\elrm \subseteq \PP_\intrm$ whose objects are called \mydef{elementary}.
A \mydef{morphism of algebraic patterns} is a functor which preserves active maps, inert maps, and elementary objects.
\end{definition}
The notion of factorization system here reduces to that of orthogonal factorization system when $\PP$ is a 1-category, which is our only case of interest in this paper.

\begin{remark}
Every hypermoment category $\CC$ yields two canonical choices for algebraic pattern structure on $\CC^\oprm$, depending only on the choice of elementary objects.
If $(\CC_\actrm, \CC_\intrm)$ is the factorization system on $\CC$, then we utilize the factorization system $(\CC_\intrm^\oprm, \CC_\actrm^\oprm)$ on $\CC^\oprm$.
We use similar notation to \cite[\S3]{ChuHaugseng:HCASC}.
\begin{itemize}
\item The elementary objects of the algebraic pattern $\CC^{\oprm,\natural}$ are the units and the nilobjects of $\CC$. This is the choice made in \cite[Definition 3.12]{Berger:MCO}.
\item The elementary objects of the algebraic pattern $\CC^{\oprm,\flat}$ are the units of $\CC$. 
\end{itemize}
In what follows, we mostly consider the associated `natural' algebraic pattern associated to the hypermoment categories under consideration.
The exceptions are the categories $\egcat$ and $\egcato$, where we will exclude nodeless loops from our set of elementary objects (following the reasoning in \cref{remark nilobjects}).
It \emph{is} important that we do not work with the `flat' algebraic pattern when comparing directed and undirected graphs, as the flat analogue of \cref{thm lke segal} does not hold.
\end{remark}

\begin{definition}
Let $\CC$ be one of the graph categories under consideration, regarded as both a hypermoment category and the opposite of an algebraic pattern.
We write $\CC_\elrm \subseteq \CC_\intrm$ for the full subcategory on the units/stars and the edges, which we call `elementary' objects.
We will write $\CC^\intrm_{/c} \subseteq \CC_{/c}$ for the full subcategory of the slice on the inert maps with codomain $c$.
Likewise, we write $\CC^\actrm_{c/} \subseteq \CC_{c/}$ for the full subcategory on the coslice consisting of active maps with domain $c$.
By the factorization system, morphisms in $\CC^\intrm_{/c}$ will also be inert maps (likewise, morphisms in $\CC^\actrm_{c/}$ will be active maps).
Finally, we write $\CC^\elrm_{/c} \coloneqq \CC_\elrm \times_{\CC_\intrm} \CC^\intrm_{/c}$ for the category whose objects are inert maps with codomain $c$ and elementary domain.
\end{definition}

In this way, all of the graph categories appearing in \cref{fig graph cat functors} of the introduction are hypermoment categories, either by the above or by \cite{Berger:MCO}, and their opposites are algebraic patterns (many were listed in \cite{ChuHaugseng:HCASC}).
Each of the indicated functors yields a morphism of algebraic patterns in the sense of \cref{def alg pattern}.
Yet the point of this machinery is to study Segal objects, and morphisms of algebraic patterns are frequently poorly behaved with respect to such.

\begin{definition}\label{def segal}
Let $\CC$ be one of our graph categories.
A presheaf $Z \in \widehat{\CC}$ is \mydef{Segal} if, for each object $c\in \CC$, the canonical map
\[ 
Z_{c} \to \lim_{w \rat c} Z_w
\]
is an isomorphism, where the limit on the right is that for the composite 
\[ (\CC^{\elrm}_{/c})^\oprm \xrightarrow{\dom} \CC^\oprm \xrightarrow{Z} \set. \]
\end{definition}

We emphasize once again that this notion differs from that of \cite[3.11]{Berger:MCO} when $\CC$ is $\egcat$ or $\egcato$, since we do not consider \emph{all} nilobjects as elementary (see \cref{remark nilobjects} for an explanation).

\begin{remark}
A variety of nerve theorems say that the full subcategory of Segal presheaves is equivalent to the corresponding category of operadic structures.
A summary of these (anticipated) results is included in the following table.

\begin{center}
\begin{tabular}{lll}
\toprule
Category & Structure & Nerve Theorem
\\ \midrule 
$\mathbf{\Omega}$ & operads & \cite{MoerdijkWeiss:OIKCDS,Weber:F2FPRA,Kock:PFT,CisinskiMoerdijk:DSSIO} \\
$\gcatcyc$ & cyclic operads & \cite{Elliot:Thesis} \\
$\gcatnought$ & augmented cyclic operads & ? \\
$\gcat$ & modular operads & \cite{HRY-mod2} \\
$\pgcatsc$ & dioperads & ? \\
$\pgcat$ & properads & \cite{HRYbook} \\
$\gcato$ & wheeled properads & \cite{HRYbook,HRYfactorizations} \\
\bottomrule
\end{tabular}
\end{center}
The categories $\gcatnought$ and $\pgcatsc$ have blanks in the table, as I am unaware of anyone establishing the relevant nerve theorems in these cases. 
One expects that each of these should follow from Weber's abstract nerve theorem \cite[\S4]{Weber:F2FPRA} (though this theorem does not apply to $\gcat$, $\pgcat$, $\gcato$). 
Further, the nerve theorem for augmented cyclic operads implies that for dioperads, as explained in \cite{Hackney:SCGO}.
A closely related nerve theorem for the category $\mathbf{\Xi}$ and cyclic operads whose set of colors has a trivial involution appears in \cite{HRY-cyclic}.
\end{remark}

To simplify matters, we regard all vertical functors and the functor $\pgcat \to \gcato$ in \cref{fig graph cat functors} as replete subcategory inclusions. 
The following lemma does not hold for $\pgcat \to \gcato$ essentially because $\emb(G)$ and $\ssub(G)$ (used in the definition of $\pgcat$ in \cref{sec properadic gcat}) differ for many acyclic directed graphs $G$.

\begin{lemma}\label{lem discrete fibrations}
All displayed functors in \cref{fig graph cat functors} of the introduction, with the exception of $\pgcat \to \gcato$, are discrete fibrations.
\end{lemma}
\begin{proof}
For the four horizontal functors landing in the right-hand column, this follows from \cref{remark pre-discrete fibration} since given a map $\varphi \colon G \to G'$ of undirected graphs and an element $x' \in \opshf_{G'}$, there is a unique element $x\in \opshf_G$ so that $\varphi \colon (G,x) \to (G',x')$ is a morphism.

The vertical functors are all sieves, hence discrete fibrations.
\end{proof}

\begin{lemma}\label{lem strong segal}
If $f \colon \CC \to \DD$ is one of the functors appearing in \cref{fig graph cat functors}, then 
\[
	\CC^\elrm_{/c} \to \DD^\elrm_{/f(c)}
\]
is an equivalence for each object $c\in \CC$.
In particular, $\CC^\oprm \to \DD^\oprm$ is a \emph{strong Segal morphism} in the sense of \cite{ChuHaugseng:HCASC}.
\end{lemma}
\begin{proof}
The elementary objects in $\pgcat$ and $\gcato$ coincide, and if $K$ is an elementary object and $G$ is an arbitrary acyclic graph, then $\pgcat_\intrm(K,G) \to (\gcato)_\intrm(K,G)$ is a bijection.
It follows that $\pgcat^\elrm_{/G} \to (\gcato)^\elrm_{/G}$ is an isomorphism.

For any of the other functors $f\colon \CC \to \DD$, the bottom map of the following pullback is an isomorphism by \cref{lem discrete fibrations}, hence the top map is an isomorphism.
\[ \begin{tikzcd}
\CC^\elrm_{/c} \dar \rar \ar[dr, phantom, "\lrcorner" very near start] & \DD^\elrm_{/f(c)} \dar \\
\CC_{/c} \rar & \DD_{/f(c)}
\end{tikzcd} \]
\par \vspace{-1.5\baselineskip} \qedhere
\end{proof}

The previous lemma tells us, in particular, that restricting Segal presheaves gives us Segal presheaves (this is precisely the point of the definition of strong Segal morphism in \cite{ChuHaugseng:HCASC}).

\begin{proposition}\label{prop restriction segal}
Suppose $f \colon \CC \to \DD$ is one of the functors appearing in \cref{fig graph cat functors} of the introduction.
If $M \in \widehat{\DD}$ is Segal, then so is $f^*M \in \widehat{\CC}$.
If $f$ is one of the functors \[ \mathbf{\Omega} \to \gcatcyc \qquad \pgcatsc \to \gcatnought \qquad \gcato \to \gcat \qquad \egcato \to \egcat,\] then if $M \in \widehat{\DD}$ is a presheaf and $f^*M \in \widehat{\CC}$ is Segal, then $M$ is Segal. 
\end{proposition}
\begin{proof}
Let $c$ be an arbitrary object of $\CC$.
We have a commutative diagram
\[ \begin{tikzcd}
(f^*M)_c \rar{=} \dar & M_{f(c)} \dar \\
\lim\limits_{w \rat c} (f^*M)_w \rar{\cong} & \lim\limits_{z \rat f(c)} M_z 
\end{tikzcd} \]
whose bottom map is an isomorphism by the equivalence $\CC^\elrm_{/c} \simeq \DD^\elrm_{/f(c)}$ from \cref{lem strong segal}.
If $M$ is Segal, then the right-hand map in the square is a bijection, hence so is the left-hand map in the square.
Thus $f^*M$ is Segal.

Now suppose that $f$ is one of the four indicated functors, and that $f^*M$ is Segal.
Let $d \in \DD$ be an arbitrary object; since $f$ is surjective on objects, choose some $c\in \CC$ with $d=f(c)$.
Since the left-hand map in the above square is a bijection, so is the right-hand map, hence $M$ is Segal as well.
\end{proof}

\begin{lemma}\label{lem unique lifting inert}
If $f \colon \CC \to \DD$ is one of the functors appearing in \cref{fig graph cat functors} of the introduction, with the exception of $\pgcat \to \gcato$, then 
\[
	\CC^\intrm_{/c} \to \DD^\intrm_{/f(c)}
\]
is an equivalence for each object $c\in \CC$.
In particular, $\CC^\oprm \to \DD^\oprm$ has \emph{unique lifting of inert morphisms} in the sense of \cite[Definition 7.2]{ChuHaugseng:HCASC}.
\end{lemma}
\begin{proof}
In all of the categories, a map is inert if and only if it maps to an embedding in $\egcat$.
Thus in any of the discrete fibrations from \cref{lem discrete fibrations}, lifts of inert maps with codomain $f(c)$ are also inert.
The conclusion follows since the square
\[ \begin{tikzcd}
\CC^\intrm_{/c} \dar \rar \ar[dr, phantom, "\lrcorner" very near start] & \DD^\intrm_{/f(c)} \dar \\
\CC_{/c} \rar & \DD_{/f(c)}
\end{tikzcd} \]
is a pullback, whose bottom map is an isomorphism by \cref{lem discrete fibrations}.
\end{proof}

The previous lemma does not hold for $\pgcat \to \gcato$, as not every embedding (with codomain an acyclic directed graph) is a convex inclusion (see \cref{example structured subgraphs}). 

\begin{remark}
Despite this unique lifting of inert morphisms, several of these functors will not be \emph{extendable} morphisms of algebraic patterns in the sense of \cite[Definition 7.7]{ChuHaugseng:HCASC} as condition (2) often fails (for instance, when going from simply-connected graphs to more general graphs).
\end{remark}

\subsection{Left Kan extension and Segality}\label{subsec lke segal}
In \cref{prop restriction segal}, we showed that restriction of presheaves behaves well with respect to the Segal condition.
We now show that for functors going from directed to undirected graphs, the left Kan extension of presheaves has an appealing form.
We use this to show that left Kan extension behaves well with respect to the Segal condition.

\begin{lemma}\label{lem lke description}
Let $f\colon \CC \to \DD$ be one of the forgetful functors
\[
\pgcatsc \to \gcatnought  \qquad  
\gcato \to \gcat  \qquad 
\egcato \to \egcat.
\]
If $Z \in \widehat{\CC}$ is $\CC$-presheaf, then the left Kan extension $f_!Z \in \widehat{\DD}$ may be described by the formula
\[
	(f_!Z)_G = \coprod_{x\in \opshf_G} Z_{(G,x)}
\]
and the action on a morphism $\varphi \colon H \to G$ of $\DD$, is given on the $x\in \opshf_G$ summand by the following:
\[ 
\begin{tikzcd}
Z_{(G,x)} \rar{\varphi^*} \dar[hook] &  Z_{(H,\varphi^*x)} \rar[hook] & \coprod\limits_{y\in \opshf_H} Z_{(H,y)} \dar["="] \\
(f_!Z)_G \ar[rr,"\varphi^*", dashed] & & (f_!Z)_H.
\end{tikzcd}\]
\end{lemma}
\begin{proof}
It is a classical fact that left Kan extension along a Grothendieck opfibration may be computed by taking colimits over the fibers.
The functor $\CC^\oprm \to \DD^\oprm$ is a discrete opfibration by \cref{lem discrete fibrations}, so the fibers $\opshf_G$ are discrete, and the result follows.
\end{proof}

Notice that a similar description of the right Kan extension as a product seems unlikely, as one cannot push forward orientations.

\begin{remark}\label{rem omega cyc}
Recall the dendroidal category $\mathbf{\Omega}$ from \cref{def dendroidal category}, which is a full subcategory of ${\gcatcyc}_{/\opshf}$.
Let $f\colon \mathbf{\Omega} \to \gcatcyc$ be the forgetful functor. 
As in the preceding lemma, we have
\[
	(f_!Z)_T = \coprod_{r\in \eth(T)} Z_{(T,r)}
\]
where we are writing $r$ for the unique orientation of $T$ so that the chosen boundary element is the unique element of $\out(T) = \eth(T) \cap A_T^-$ and each $\out(v) = \nbhd(v) \cap D_T^+$ is a singleton (see \cref{rmk inp out plus minus}).
This formula is related to Lemma 4.2 and Definition 4.8 of \cite{DrummondColeHackney:CERIMS}, and is also closely related to \S2.1 of \cite{HRY-cyclic}.
\end{remark}

The following lemma will be used in the proof of \cref{thm lke segal}, which states that $f_!$ preserves Segal objects.

\begin{lemma}\label{warm up lke segal}
Let $f\colon \CC \to \DD$ be one of the forgetful functors
\[
\pgcatsc \to \gcatnought  \qquad  
\gcato \to \gcat  \qquad 
\egcato \to \egcat.
\]
If $Z \in \widehat{\CC}$ is a presheaf and $G\in \DD$ is an undirected graph, then the map 
\begin{equation}\label{exchange coprod limit}
	\coprod\limits_{x\in \opshf_G} \lim\limits_{k \colon K \rat G} Z_{(K,k^*x)} \to \lim\limits_{k \colon K \rat G} \coprod\limits_{y\in \opshf_K} Z_{(K,y)} = \lim\limits_{k \colon K \rat G} (f_!Z)_K
\end{equation}
is a bijection, where the limits are taken over the opposite of $\DD^\elrm_{/G}$ (in particular, each $K$ is either an edge or a star).
\end{lemma}
\begin{proof}
For concreteness, we prefer to take a skeleton of $\DD^\elrm_{/G}$, with objects 
\[
  \{ \iota_v \colon \medstar_v \rat G \mid v\in V_G \} \amalg \{ \kappa_e \colon {\exedge}_e \rat G \mid e\in E_G \}
\]
(see \cref{example embeddings}).

We first observe that $\opshf$ from \cref{def or presheaf} is Segal. 
This means that the function
\begin{equation}\label{eq opshf segal}
  \opshf_G \to \lim_{k\colon K \rat G} \opshf_K,
\end{equation}
sending $x\in \opshf_G$ to the element $(k^*x)$, is a bijection.
It is automatically an injection since the composite $\opshf_G \to \lim_k \opshf_K \to \prod_{\kappa_e} \opshf_{{\exedge}_e}$ 
is a bijection, as both $\opshf_G$ and $\prod_{\kappa_e} \opshf_{{\exedge}_e}$ are isomorphic to the set of involutive maps $A_G \to \{+1,-1\}$.
On the other hand, if $(y^k) \in \lim_k \opshf_K$, then by using the bijection $\opshf_G \cong \prod_{\kappa_e} \opshf_{{\exedge}_e}$ we obtain $x \in \opshf_G$ with $\kappa_e^* x = y^{\kappa_e}$ for all $e\in E_G$.
We wish to show that $k^* x = y^k$ for all $k\in \DD^\elrm_{/G}$.
If $G$ does not have any vertices, then we are done.
Otherwise, let $v$ be a vertex. 
For each $d\in \nbhd(v) = D_{\medstar_v}$ spanning an edge $e \in E_G$, there is an embedding $j_d \colon {\exedge}_e \rat \medstar_v$ sending $d$ to $d$ and making the following triangle commute.
\[ \begin{tikzcd}[column sep =small]
& G \\
\medstar_v \ar[ur,tail,"\iota_v"]  & & {\exedge}_e \ar[ll,"j_d"',tail] \ar[ul,tail,"\kappa_e"']
\end{tikzcd} \]
Then $j_d^*(y^{\iota_v}) = y^{\kappa_e} = \kappa_e^* x = j_d^*(\iota_v^*x)$. 
Since $\opshf_{\medstar_v} \cong \prod_{\nbhd(v)} \opshf_{{\exedge}_e}$, we conclude that $y^{\iota_v} = \iota_v^*x$.
We have now established that $y^k = k^* x$ for all $k\in \DD^\elrm_{/G}$, proving that \eqref{eq opshf segal} is a bijection.

Elements on the left-hand side of \eqref{exchange coprod limit} are of the form $(x,(z^k))$ where $x\in \opshf_G$ and $z^k \in Z_{(K, k^*x)}$ for each $k\colon K \rat G$ in $\DD^\elrm_{/G}$.
These are subject to the compatibility condition $j^*(z^k) = z^{kj}$ whenever $j \colon K' \rat K$ is an embedding. 
Likewise, elements on the right-hand side of \eqref{exchange coprod limit} are of the form $(y^k,s^k)$ with $y^k \in \opshf_K$, $s^k \in Z_{(K,y^k)}$ for each $k\colon K \rat G$ in $\DD^\elrm_{/G}$. 
These are subject to the compatibility conditions $j^*(y^k) = y^{kj}$ and $j^*(s^k) = s^{kj}$ whenever $j \colon K' \rat K$ is an embedding.
The map \eqref{exchange coprod limit} takes $(x, (z^k))$ to $(k^*x, z^k)$.
Bijectivity of \eqref{exchange coprod limit} follows from that of \eqref{eq opshf segal}.
\end{proof}

\begin{theorem}\label{thm lke segal}
Let $f\colon \CC \to \DD$ be one of the forgetful functors
\[
\pgcatsc \to \gcatnought  \qquad  
\gcato \to \gcat  \qquad 
\egcato \to \egcat.
\]
Then the functor $f_! \colon \widehat{\CC} \to \widehat{\DD}$ preserves Segal objects.
\end{theorem}
\begin{proof}
Let $Z\in \widehat{\CC}$ be a Segal object, that is, suppose that 
\[
	Z_{(G,x)} \to \lim_{(\CC^\elrm_{/(G,x)})^\oprm} Z_{(K,y)}
\]
is a bijection for each $(G,x)\in \CC$.
We know that $\CC^\elrm_{/(G,x)} \to \DD^\elrm_{/G}$ is an equivalence by \cref{lem strong segal}, so this becomes
\[
	Z_{(G,x)} \cong \lim_{k \colon K \rat G} Z_{(K,k^*x)}
\]
instead, where $K$ ranges over elementary objects of $\DD$.
The function
\begin{equation}\label{eq segal map for lke}
	(f_!Z)_G \to \lim_{k \colon K \rat G} (f_!Z)_K
\end{equation}
is such that the following diagram commutes for every $x_0$ in $\opshf_G$ and $k_0$ in $\DD^\elrm_{/G}$.
\[ \begin{tikzcd}
\lim\limits_{k \colon K \rat G} Z_{(K,k^*x_0)} \dar[hook,"i_{x_0}"] \ar[rr,"\pi_{k_0}"] & & 
	Z_{(K_0,k_0^*x_0)} \dar[hook,"i_{k_0^*x_0}"]
\\
\coprod\limits_{x\in \opshf_G} \lim\limits_{k \colon K \rat G} Z_{(K,k^*x)} \rar & 
	\lim\limits_{k \colon K \rat G} \coprod\limits_{y\in \opshf_K} Z_{(K,y)} \rar["\pi_{k_0}"] &
	\coprod\limits_{y\in \opshf_{K_0}} Z_{(K_0,y)}
	\\
(f_!Z)_G \uar{\cong} \rar & \lim\limits_{k \colon K \rat G} (f_!Z)_K \uar{=}
\end{tikzcd} \]
Applying \cref{warm up lke segal}, we see that \eqref{eq segal map for lke} is a bijection, hence $f_!Z$ is Segal.
\end{proof}

This proof (including the corresponding statement for \cref{warm up lke segal}) can be readily adapted by using \cref{rem omega cyc} in place of \cref{lem lke description} to give the following.

\begin{proposition}\label{prop lke segal omega}
If $f$ is the functor $\mathbf{\Omega} \to \gcatcyc$ that forgets about the directed structure, then $f_! \colon \widehat{\mathbf{\Omega}} \to \widehat{\gcatcyc}$ takes Segal objects to Segal objects. \qed
\end{proposition}

\begin{remark}
\Cref{thm lke segal} (and \cref{prop lke segal omega}) can also be recovered from a theorem of Haugseng--Kock \cite{HaugsengKock}, which concerns a left fibration over an algebraic pattern whose unstraightening is Segal. 
We briefly explain this alternate method.
We know that $\CC^\oprm \to \DD^\oprm$ is a discrete opfibration \cref{lem discrete fibrations} and $\opshf$ is Segal by \cref{warm up lke segal}. 
Proposition 3.2.5 of \cite{HaugsengKock} implies that left Kan extension along $f$ induces an equivalence of $\infty$-categories
\begin{equation}\label{eq haugseng kock equiv} \mathrm{Seg}_{\CC^\oprm}(\spaces) \xrightarrow{\sim} \mathrm{Seg}_{\DD^\oprm}(\spaces)_{/\opshf}, \end{equation}
where $\spaces$ denotes the $\infty$-category of spaces.
The full subcategory of $\widehat{\CC}$ on the Segal objects is equivalent to the subcategory of discrete objects of $\mathrm{Seg}_{\CC^\oprm}(\spaces)$ by \cite[Lemma 1.12]{BeardsleyHackney:LCC}.
The equivalence \eqref{eq haugseng kock equiv} above preserves discrete objects, and since $\opshf$ is discrete, $\mathrm{Seg}_{\DD^\oprm}(\spaces)_{/\opshf} \to \mathrm{Seg}_{\DD^\oprm}(\spaces)$ preserves discrete objects as well \cite[Lemma 5.5.6.14]{Lurie:HTT}.
We now have that left Kan extension along $f$ restricts to a functor $\mathrm{Seg}_{\CC^\oprm}(\spaces) \to \mathrm{Seg}_{\DD^\oprm}(\spaces)$ which preserves discrete objects, so $\widehat{\CC} \to \widehat{\DD}$ preserves Segal objects.
The argument for \cref{prop lke segal omega} (where $f\colon \CC \to \DD$ is $\mathbf{\Omega} \to \gcatcyc$) is analogous, using the rooting presheaf instead of $\opshf$.
\end{remark}

\appendix

\section{Categories of trees for cyclic operads}\label{sec tree cat}
In this section we specialize \cref{def new graph map} to the case when the codomain $G'$ is simply-connected, which implies that the domain is also a tree by \cite[Proposition 5.2]{HRY-mod1}.
The resulting categories of graphs are related to cyclic operads \cite{GetzlerKapranov:COCH} and higher cyclic operads \cite{HRY-cyclic,Walde:2SSIIO}.
Our goal is to give a description of new graph maps between trees that is mirrors the `complete morphisms' from Definition 1.12 of \cite{HRY-cyclic}.
This description appears as \cref{thm on trees} below.

In this section, all graphs are undirected.

\begin{definition}[Paths and trees]\label{def cycles paths}
Let $G$ be a graph.
\begin{itemize}\label{def paths and trees}
\item A \mydef{path} in $G$ is a finite alternating sequence of edges and vertices of $G$ so that if $e$ and $v$ are adjacent then some arc of $e$ appears in $\nbhd(v)$, and so that the pattern $vev$ can only appear in the path if \emph{both} arcs of $e$ are in $\nbhd(v)$.
\item A \mydef{cycle} is a path of length strictly greater than one that begins and ends at the same edge or same vertex.
\item The graph $G$ is a \mydef{tree} if and only if it is connected and does not have any paths which are cycles.
\end{itemize}
We write $\gcatnought \subset \gcat$ for the full subcategory on the trees, and $\gcatcyc \subset \gcatnought$ for the full subcategory consisting of those trees with non-empty boundary.
\end{definition}

Note that a graph $G$ is a tree just when its associated topological space is simply-connected.
A graph $G$ is connected if and only if for each pair \[ (x_1,x_2) \in (V_G \amalg E_G)^{\times 2},\] there is a path containing both $x_1$ and $x_2$.

\begin{lemma}\label{lem tree injective arcs}
If $T$ is a tree and $f \colon H \rightarrowtail T$ is an embedding, then $A_H \to A_T$ is injective.
\end{lemma}
\begin{proof}
Suppose $a_1 \neq a_2$ are distinct arcs of $H$ with $f(a_1) = f(a_2)$.
Without loss of generality, by \cref{mod1 lem 1.22} we may suppose that $a_1, a_2^\dagger \in \eth(H)$ and $a_1^\dagger, a_2 \in D_H$ (in particular, $H$ is not an edge).
Since $H$ is connected, there exists a path 
\[ P = e_0 v_1 e_1 v_2 e_2 \dots v_n e_n\]
in $H$ from $e_0 = [a_1,a_1^\dagger]$ to $e_n = [a_2, a_2^\dagger]$.
Since $a_1 \neq a_2^\dagger$ (equality would imply that $f(a_1)$ is a $\dagger$-fixed point of $A_T$), we know $e_0 \neq e_n$ so it follows that $n\geq 1$.
There is an associated path $fP$ of $T$ obtained by applying $f$ to each edge and vertex.
Since $fe_0 = fe_n$, the path $fP$ is a cycle, which is impossible since $T$ is a tree.
\end{proof}

In an arbitrary graph $G$, elements in $\emb(G)$ are not necessarily uniquely determined by their boundary (for example, consider $f$ and $fk$ from Example 1.23 of \cite{HRY-mod1}).
This complication disappears when $G$ is a tree.

\begin{lemma}\label{lem tree emb injective}
If $T$ is a tree, then $\eth \colon \emb(T) \to \wp(A_T)$ is injective.
\end{lemma}
\begin{proof}
By \cref{mod1 prop 1.25}, the only way this can fail is if there is a pair $h,k$ consisting of an embedding $h \colon H \rightarrowtail T$ where $H$ contains at least one vertex and an embedding $k \colon K \rightarrowtail T$ where $K$ is an edge, satisfying $\eth(h) = \eth(k)$.
This implies that $\eth(H) = \{ b_0, b_1 \}$ has two elements and $h(b_0)^\dagger = h(b_1)$.
Since $b_0^\dagger \in D_H$ it cannot be equal to $b_1$, hence $h$ is not injective.
This is prohibited by \cref{lem tree injective arcs}.
\end{proof}

\begin{definition}\label{def subtrees}
Suppose $T$ is a tree.
A \mydef{subtree} $S$ of $T$ is a triple of subsets $A_S \subseteq A_T$, $D_S \subseteq D_T$, and $V_S \subseteq V_T$ so that the following hold.
\begin{itemize}
\item There is a commutative diagram as displayed 
\[ \begin{tikzcd}
A_S  \dar[hook] & A_S \lar[dashed]  \dar[hook] & D_S \lar[dashed] \rar[dashed]  \dar[hook] \ar[dr, phantom, "\lrcorner" very near start] & V_S \dar[hook]  \\
A_T & A_T \lar["\dagger"'] & D_T \lar[hook'] \rar & V_T 
\end{tikzcd} \]
whose vertical maps are the inclusions and whose right square is a pullback.
\item The graph $S$ is connected.
\end{itemize}
\end{definition}

Each subtree of $T$ determines an embedding, and by \cref{lem tree injective arcs} any embedding $h\colon H \rightarrowtail T$ determines a subtree using the subsets $h(A_H), h(D_H)$, and $h(V_H)$. 
This produces a bijection between the set of subtrees of $T$ and the set $\emb(T)$.
In light of this, we use the following shorthand.

\begin{notation}
Suppose that $S$ and $T$ are trees and $\hat \varphi \colon \emb(S) \to \emb(T)$ is a function.
If $R$ is a subtree of $S$, we will write $\hat \varphi(R)$ for the subtree of $T$ associated to the element $\hat \varphi [R \to S] \in \emb(T)$.
\end{notation}

\begin{remark}[Uniqueness of unions]\label{remark overlap}
Suppose $R$ and $S$ are two subtrees of $T$.
We say that $R$ and $S$ \mydef{overlap} if $R\cap S$ (meaning the triple $(A_R \cap A_S, D_R \cap D_S, V_R \cap V_S)$) is non-empty.
As in \cite[\S1.2]{HRY-cyclic}, the graphs $R\cap S$ and $R\cup S$ are connected, and hence subtrees, if and only if $R$ and $S$ overlap.
Unions in the sense of \cref{def a union} are unique in this context, and coincide with the union of subtrees.
\end{remark}

We now turn to our desired characterization of graphical maps between trees.
By Proposition 5.2 of \cite{HRY-mod1}, $\gcatnought$ is a sieve in $\gcat$, so any graphical map with codomain a tree also has a tree as its domain.

\begin{theorem}\label{thm on trees}
Suppose that $G$ and $G'$ are trees, and $\varphi = (\varphi_0,\hat \varphi)$ is a pair consisting of an involutive function $\varphi_0 \colon A_G \to A_{G'}$ and a function $\hat \varphi \colon \emb(G) \to \emb(G')$ which satisfy condition {\rm \eqref{new graph def boundary}} of \cref{def new graph map}.
Then $\varphi$ is a new graph map if and only if 
\rm
\begin{enumerate}[label=(\roman*),ref=\roman*, start=5]
	\item If two subtrees $S,T$ of $G$ overlap, then so do $\hat \varphi (S)$ and $\hat \varphi (T)$.
	In this case, we have \label{new graph subtrees}
	\begin{enumerate}[label=(\alph*),ref=\alph*] 
		\item $\hat \varphi(S\cap T) = \hat\varphi(S) \cap \hat \varphi(T)$ and \label{v intersect}
		\item $\hat \varphi(S\cup T) = \hat\varphi(S) \cup \hat \varphi(T)$. \label{v union}
	\end{enumerate}
\end{enumerate}
\end{theorem}
\begin{proof}
In this proof, \eqref{new graph def edges}--\eqref{new graph def boundary} refer to the conditions from \cref{def new graph map}.

Suppose $\varphi$ is a new graph map and let $S$ and $T$ be overlapping subtrees of $G$.
As $\hat\varphi$ preserves order, we have the following subtrees
\[ \begin{tikzcd}
\hat\varphi(S) \ar[dr,tail] & \lar[tail] \hat\varphi(S\cap T) \rar[tail] \dar[tail] & \hat\varphi(T) \ar[dl,tail] \\
& G'
\end{tikzcd} \]
which implies that $\hat\varphi(S)$ and $\hat\varphi(T)$ overlap.

In the next two paragraphs, we will write $\medstar_v'$ for the subtree containing the single vertex $v$.
(That is, $\medstar_v'$ is the subtree associated to the embedding $\iota_v \colon \medstar_v \rightarrowtail G$.)

To show that \eqref{v intersect} holds, first note that $\hat\varphi(S\cap T)$ is a subtree of $\hat\varphi(S) \cap \hat\varphi(T)$.
To show these are equal, it suffices to show that every vertex in $\hat\varphi(S) \cap \hat\varphi(T)$ is also in $\hat\varphi(S\cap T)$.
But if $v$ is vertex in $\hat\varphi(S) \cap \hat\varphi(T)$, then we can find unique vertices $u\in V_S \subseteq V_G$ and $w\in V_T \subseteq V_G$ so that $v$ is in $\hat \varphi (\medstar_u')$ and in $\hat \varphi (\medstar_w')$.
By \eqref{new graph def intersect} the equality $u = w$ holds, so $v$ is in $\hat\varphi(S\cap T)$.

To see that \eqref{v union} holds, note that $\hat \varphi(S) \cup \hat\varphi(T)$ is a subtree of $\hat \varphi(S\cup T)$.
Given a vertex $v$ of $\hat \varphi(S\cup T)$ there is a unique vertex $u \in V_{S\cup T} \subseteq V_G$ with $v$ in $\hat\varphi (\medstar_u')$.
If $u$ is in $S$, then $v$ is in $\hat\varphi(S)$, and if $u$ is in $T$ then $v$ is in $\hat \varphi(T)$.
Hence the subtree $\hat \varphi(S) \cup \hat\varphi(T)$ contains all vertices of $\hat \varphi(S\cup T)$, so these are equal.

Now suppose $\varphi$ satisfies conditions \eqref{new graph def boundary} and \eqref{new graph subtrees}.
Condition \eqref{new graph def union} follows immediately from \eqref{v union}.

We next address \eqref{new graph def edges}, which says that $\hat \varphi$ sends edges to edges. 
If $T$ is an edge of $G$ consisting of the arcs $a,a^\dagger$, then by \eqref{new graph def boundary} the boundary of $\hat \varphi (T)$ is the set $\{ \varphi_0(a), \varphi_0(a^\dagger) \}$.
There is only one subtree having this boundary by \cref{lem tree emb injective}, namely the edge consisting of the arcs $\{ \varphi_0(a), \varphi_0(a)^\dagger \}$.
Thus \eqref{new graph def edges} holds.

Suppose 
\[ \begin{tikzcd}
S' \rar[tail,"h'"] & G & \lar[tail, "k"' ] T
\end{tikzcd} \]
are subtrees of $G$ with $V_{S'} \cap V_{T}$ empty.
Our goal is to show that $\hat \varphi (S')$ and $\hat \varphi (T)$ are vertex disjoint.
If either of $T$ or $S'$ is an edge then the same is true for $\hat \varphi (T)$ or $\hat \varphi (S')$ by \eqref{new graph def edges}, in which case $\hat \varphi(S')$ and $\hat \varphi(T)$ are vertex disjoint.
We thus suppose that both $S'$ and $T$ contain a vertex.
Consider the subset $\mathfrak{X}$ of $\emb(G)$ consisting of those subtrees $R$ with $V_R \cap V_T$ empty and $S' \subseteq R$.
Since $\mathfrak{X}$ is inhabited, we choose a maximal element $S\in \mathfrak{X}$ and then argue that the set $\eth(S) \cap A_T$ is inhabited. 
Suppose $\eth(S) \cap A_T$ is empty. 
If $a \in \eth(S) \cap D_G \subset D_G \setminus D_T$, then $t(a)$ is in $V_G \setminus (V_S \cup V_T)$, contradicting maximality of $S$. 
We conclude that $\eth(S) \cap D_G$ must also be empty, hence $\eth(S) \subseteq \eth(G)$.
By \cref{lem boundary inclusion} this implies that $S = G$, hence $T\subseteq S$, a contradiction to vertex disjointness.

Taking this $S$ and $T$, we have $S\cap T$ is an edge and thus the same is true for $\hat \varphi(S \cap T)$ by \eqref{new graph def edges}. 
Using \eqref{v intersect}, we have \[ \hat \varphi(S \cap T) = \hat \varphi(S) \cap \hat \varphi(T) \supseteq \hat \varphi(S') \cap \hat \varphi(T),\]
so $\hat \varphi(S')$ and $\hat \varphi(T)$ are vertex disjoint.
Thus \eqref{new graph def intersect} holds.
\end{proof}

This theorem makes wholly transparent the functor $\gcatcyc \to \mathbf{\Xi}$ into the category of trees from \cite[Definition 1.12]{HRY-cyclic}. 

\begin{remark}
Condition \eqref{v intersect} of \cref{thm on trees}\eqref{new graph subtrees} is automatic given the other conditions. 
We prove this in \cref{prop redundancy}.
\end{remark}

\section{Acyclic graphs and the properadic graphical category}\label{sec properadic gcat}

In this section, we turn to the connected directed graphs which are \emph{acyclic} (from now on only in the directed sense) and control \emph{properads}.
Properads were introduced by Vallette in his thesis \cite{Vallette:KDP} and, independently and in the form we use them, under the name compact symmetric polycategories in Duncan's thesis \cite{Duncan:TQC}.
Properads are the connected parts of props \cite{MacLane:CA}.

One major goal is to prove \cref{thm properadic gcat}, which characterizes the properadic graphical category $\pgcat$ from \cref{def properadic gcat} as a particular subcategory of $\gcato$, giving a more explicit description of Theorem 9.66 from \cite{HRYbook}.

\begin{example}\label{ex linear graph}
A connected directed graph $L$ is \mydef{linear} if each vertex has exactly one input and one output, and if the graph itself has exactly one input and one output.
We write ${\exedge}_i \rat L$ (resp.\ ${\exedge}_o \rat L$) for the embedding classifying the unique input (resp.\ output) edge.
\end{example}

\begin{definition}\label{def directed paths cycles}
Let $G$ be a directed graph.
\begin{itemize}
	\item A \mydef{directed path} in $G$ is a map $L \to G$ in the functor category $\finset^{\mathscr{G}}$.
	\item A \mydef{directed cycle} is $G$ is a path $L\to G$ from a non-edge graph $L$ so that the two composites ${\exedge}_i \rat L \to G$ and ${\exedge}_o \rat L \to G$ are equal.
	\item The graph $G$ is \mydef{acyclic} if and only if it is connected (as an undirected graph) and it does not have any cycles.
\end{itemize}
\end{definition}
The reader who compares this definition with \cref{def cycles paths} will notice that here we require our directed paths start and end at edges, rather than possibly at vertices.
When discussing directed graphs, acyclicity will always refer to this directed notion, rather than asking that the underlying undirected graph is a tree (which of course implies acyclicity).

\begin{warning}[Regarding \cref{sec extended}]
In this section, we do not utilize the extended directed graphs from \cref{def ext dir graph}, and all directed graphs are as in \cref{def dir graph}.
We should morally regard the directed nodeless loops from \cref{ex dir nodeless loop} as having a cycle, but it is easier to exclude them from the discussion entirely.
\end{warning}

Notice that if $H$ has a cycle and $H \to G$ is any map in $\finset^{\mathscr{G}}$, then $G$ has a cycle as well.
Hence any embedding with acyclic codomain also has acyclic domain.

Item \eqref{eq conv inclusion} in the following is the same concept as \cite[1.6.5]{Kock:GHP}, but restricted to the connected graphs.
\begin{definition}\label{def convex inclusion}
Let $G$ be an acyclic directed graph.
\begin{enumerate}
\item 
We say an embedding $h\colon H \rat G$ is a \mydef{convex inclusion} if\label{eq conv inclusion}
\begin{enumerate}
\item it is injective, and
\item it has the right lifting property in $\finset^{\mathscr{G}}$ with respect to the maps \[ \downarrow_i \amalg \downarrow_o \to L\] as $L$ ranges over the linear graphs from \cref{ex linear graph}.
\end{enumerate}
\item A \mydef{unstructured subgraph} $H$ of $G$ is a pair of subsets $E_H \subseteq E_G$ and $V_H \subseteq V_G$ so that \[ \bigcup_{v\in V_H} \big[ \inp(v) \cup \out(v) \big] \subseteq E_H.\]
Then $H$ is a graph by defining $I_H \coloneqq \bigcup_{V_H} \inp(v)$ and $O_H \coloneqq \bigcup_{V_H} \out(v)$ 
and the subset inclusions constitute an \'etale map $H \to G$. 
\item We write $\ssub(G) \subseteq \emb(G)$ for set of equivalence classes which represent convex inclusions, and call the elements of $\ssub(G)$ \mydef{structured subgraphs} of $G$.
That is, given a convex inclusion $h\colon H' \rat G$, the subsets $E_{H'} \cong h(E_{H'}) \subseteq E_G$ and $h(V_{H'})\subseteq V_G$ determine an unstructured subgraph $H$ whose inclusion is isomorphic to $h$.
We will use the notation $H\in \ssub(G)$ for such a structured subgraph.
\end{enumerate}
\end{definition}
The terminology of structured subgraphs agrees with Definition 2.2.2 (via Remark 2.2.4) of \cite{ChuHackney}. 
These were called \emph{convex open subgraphs} in \cite[1.6.5]{Kock:GHP} and \emph{subgraphs} in \cite[Definition 6.32]{HRYbook}.

\begin{remark}
The acyclicity of a directed graph $G$ typically depends on its directed structure. 
Further, the set of structured subgraphs $\ssub(G)$ also depends on this structure, while the set of embeddings $\emb(G)$ does not.
\end{remark}

\begin{example}\label{example structured subgraphs}
\leavevmode
\begin{itemize}
	\item Every edge and every vertex determines a structured subgraph (see \cref{example embeddings}).
	Hence we have the following commutative diagram.
	\[ \begin{tikzcd}
	E_G \ar[dr, hook] \rar[dashed] & \ssub(G) \dar[hook] & V_G \ar[dl, hook] \lar[dashed] \\ & \emb(G)
	\end{tikzcd} \]
	\item The graph $G$ is itself in $\ssub(G)$.
	\item The subgraph in \cref{figure not structured}, with all edges pointing down, is not a structured subgraph as there is no directed path from $e_0$ to $e_1$.
\end{itemize}
\end{example}

\begin{figure}
\labellist
\small\hair 2pt
 \pinlabel {$e_0$} at 114 130
 \pinlabel {$e_0$} at 262 130
 \pinlabel {$e_1$} at 114 57
 \pinlabel {$e_1$} at 262 57
\endlabellist
\centering
\includegraphics[scale=0.5]{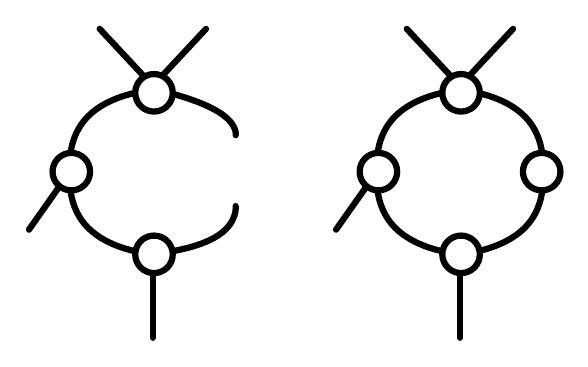}
\caption{A subgraph which is not structured}
\label{figure not structured}
\end{figure}
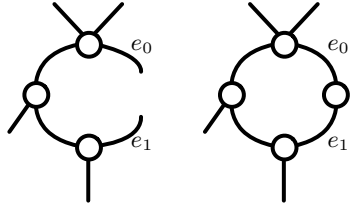

\begin{remark}\label{rem unions}
If $H$ and $K$ are unstructured subgraphs of $G$, then their union $H\cup K$ with
\begin{align*}
	E_{H \cup K} &= E_H \cup E_K \subseteq E_G \\
	V_{H \cup K} &= V_H \cup V_K \subseteq V_G
\end{align*}
is also an unstructured subgraph of $G$.
If $H$ and $K$ are two \emph{structured} subgraphs of $G$, then $H\cup K$ may or may not be a structured subgraph.
One obvious possiblity for failure is that $H\cup K$ may be disconnected, but even if it is connected its inclusion may not be convex.
(\Cref{figure not structured} provides one example, by taking $H$ to be spanned by the valence three vertices, and $K$ to be spanned by the valence four vertex.)
We will write $H\cup K \in \ssub(G)$ when this union happens to be a structured subgraph.
\end{remark}

If $H, K$ are structured subgraphs of $G$, then they may have several unions when considered as elements of $\emb(G)$. 
But there is at most one union that is again a structured subgraph.

\begin{lemma}\label{lem a union vs union of subgraphs}
Suppose $H,K\in \ssub(G)$, write $h$ and $k$ for the two inclusions, and 
let $U \subseteq \emb(G)$ denote the set of all unions of $h$ and $k$ (in the sense of \cref{def a union}).
If $U$ is inhabited, then $U \cap \ssub(G) = \{ [H \cup K \rat G] \}$.
\end{lemma}
\begin{proof}
If $H\cup K \in \ssub(G)$ then the inclusion $H \cup K \rat G$ is an upper bound for $h$ and $k$ and by \cref{rem unions} we know $V_{H\cup K} = V_H \cup V_K$, so the inclusion is a union.

Now suppose $J \in \ssub(G)$ and the inclusion $J \rat G$ is a union of the inclusions $h$ and $k$.
By \cref{def a union}\eqref{def a union cup}, we know that $V_J = V_H \cup V_K,$ hence $V_J = V_{H\cup K}$.
Further, we have $H \rat J \rat G$ and $K \rat J \rat G$, so $E_H \cup E_K \subseteq E_J$. 
If $J$ is an edge, then this must be an equality.
If $J$ is not an edge, then every $e\in E_J$ is incident to some vertex $v \in V_J = V_{H\cup K}$, hence must already be in $E_H \cup E_K$. Thus $E_{H\cup K} = E_J$, so we conclude that $H \cup K = J$.
\end{proof}

Now that we have structured subgraphs and their unions, we can present the version of the properadic graphical category that appeared in  Definitions 2.2.11 and 2.2.14 of \cite{ChuHackney}.
\begin{definition}\label{def properadic gcat}
The \mydef{properadic graphical category}, denoted $\pgcat$, has objects the acyclic directed graphs.
A morphism $\varphi \colon G \to G'$ consists of two functions $\bar\varphi_0 \colon E_G \to E_{G'}$ and $\check \varphi \colon \ssub(G) \to \ssub(G')$ so that 
\begin{enumerate}[label=(\arabic*),ref=\arabic*]
\item The diagram
\[ \begin{tikzcd}
\mathbb{N}E_G  \dar{\mathbb{N} \bar \varphi_0 } & \ssub(G) \rar{\out}\lar[swap]{\inp} \dar{\check \varphi}& \mathbb{N}E_G \dar{\mathbb{N} \bar \varphi_0 }\\
 \mathbb{N}E_{G'} & \ssub(G') \rar{\out}\lar[swap]{\inp} & \mathbb{N}E_{G'}
\end{tikzcd} \]
commutes.\label{def properadic gcat bdry}
\item Suppose that $H_1, H_2 \in \ssub(G)$. If $H_1 \cup H_2 \in \ssub(G)$, then $\check \varphi(H_1 \cup H_2) = \check \varphi(H_1) \cup \check \varphi (H_2)$.\label{def properadic gcat union}
\end{enumerate}
Composition in $\pgcat$ is given by composition of pairs.
\end{definition}
The equivalence with the original definition of the  properadic graphical category from \cite{HRYbook} is proved in Theorem A.1 of \cite{ChuHackney}.

Our goal is to prove \cref{thm properadic gcat}, which characterizes $\pgcat$ as a subcategory of $\gcato$.
We give a proof that does not directly rely on the notions of properad or wheeled properad, instead relying on properties of graphs and the notions developed here.
It is also possible to prove this using (wheeled) properads, by comparing the two notions of subgraph from \S6.3.2 and \S9.4.1 of \cite{HRYbook}.

We begin by exhibiting a function $\pgcat(G,G') \to \gcato(G,G')$ which we later prove (in \cref{thm properadic gcat}) is part of an injective-on-objects functor.

\begin{proposition}\label{prop nonfunctor pgcat to gcat}
Suppose $\varphi = (\bar \varphi_0, \check \varphi) \colon G \to G'$ is a properadic graphical map between two acyclic directed graphs.
Then $\varphi$ determines graphical map in $\gcato$ using $\varphi_0 = \varphi_0^- \amalg \varphi_0^+ \colon A_G \to A_{G'}$ defined as follows
\[ \begin{tikzcd}
A_G^- \dar[dashed,"\varphi_0^-"] \rar["\cong"] & E_G \dar["\bar \varphi_0"] &  A_G^+ \dar[dashed,"\varphi_0^+"] \lar["\cong"'] \\
A_{G'}^- \rar["\cong"] & E_{G'}  &  A_{G'}^+ \lar["\cong"']
\end{tikzcd} \]
and using the composite function
\[ \begin{tikzcd}
V_G \rar[dashed, "\varphi_1"] \dar[hook] & \emb(G') \\
\ssub(G) \rar{\check \varphi} & \ssub(G'). \uar[hook]
\end{tikzcd} \]
\end{proposition}
\begin{proof}
One checks the conditions of \cref{def graphical map}.
Condition \eqref{old graph def vertices} is known for properadic graphical maps using the definition in \cite{HRYbook}. 
Condition \eqref{old graph def boundary} follows from \cref{def properadic gcat}\eqref{def properadic gcat bdry}.

If $\eth(G) = \varnothing$, then since $G$ is acyclic there must be some vertex $v$ so that at least one of $\inp(v)$ or $\out(v)$ is not a one-element set. Hence condition \eqref{old graph def collapse} holds.
\end{proof}

The following useful proposition appears as Corollary 6.62 of \cite{HRYbook}.

\begin{proposition}\label{prop forgetful}
The forgetful functor $\pgcat \to \set$ that sends $G$ to $E_G$ and $\varphi$ to $\bar \varphi_0$ is faithful. \qed
\end{proposition}

\begin{remark}\label{remark additional axiom}
In the definition of properadic graphical maps $\varphi \colon G \to G'$, we could have forgotten about the function $\bar \varphi_0$ had we been willing to add an additional axiom on $\check \varphi$:
\begin{enumerate}[label=(\arabic*),ref=\arabic*,start=0]
\item The function $\check \varphi \colon \ssub(G) \to \ssub(G')$ sends edges to edges. \label{rmk properadic gcat edge}
\end{enumerate}
Indeed, doing so would let us define $\bar \varphi_0$ to fit into the following diagram
\[ \begin{tikzcd}
E_G \rar[dashed,"\bar \varphi_0"] \dar[hook] & E_{G'} \dar[hook] \\
\ssub(G) \rar{\check \varphi} & \ssub(G').
\end{tikzcd} \]
In the original definition, \eqref{rmk properadic gcat edge} is a consequence of \eqref{def properadic gcat bdry}.
\end{remark}

\begin{lemma}\label{lem dir embed inje}
Suppose $k \colon K \rat G$ is an embedding in $\gcato$.
If $k(e) = k(e')$ for distinct edges $e,e'\in E_K$, then $k(e)$ is an internal edge of $G$ and one of $e$ or $e'$ is in $\inp(K)$ while the other is in $\out(K)$.
As a consequence, if $k(\inp(K)) \subseteq \inp(G)$ or $k(\out(K)) \subseteq \out(G)$, then $k$ is injective.
\end{lemma}
\begin{proof}
This follows from \cref{mod1 lem 1.22}, using \cref{cnst gcato to kock,rmk inp out plus minus}.
\end{proof}

\begin{lemma}\label{lem partially grafted corollas}
Suppose $J$ and $G$ are acyclic directed graphs, $J$ has two vertices $V_J = \{ u , v\}$, and $\psi \colon J \ract G$ is an active map in $\gcato$.
Then \[ \psi_1(u), \psi_1(v)\in \ssub(G) \subseteq \emb(G).\]
\end{lemma}
\begin{proof}
Notice that exactly one of the sets $\inp(u) \cap \out(v)$ and $\inp(v) \cap \out(u)$ is inhabited, for otherwise $J$ would have a directed cycle or be disconnected. 
Without loss of generality, we suppose that the first of these is empty, that is, we suppose $\inp(u) \subseteq \inp(J)$, $\out(v) \subseteq \out(J)$, and $\inp(v) \cap \out(u) \neq \varnothing$.

We first show that $\psi_1(u)$ and $\psi_1(v)$ are injective, hence represented by inclusions of unstructured subgraphs $K_u$ and $K_v$.
Write $\psi_u \colon K_u' \rat G$ for a representative of $\psi_1(u)$.
Since $\psi_u (\inp (K_u')) = \bar \psi_0 (\inp(u)) \subseteq \bar \psi_0 (\inp(J)) = \inp(G)$, by \cref{lem dir embed inje} we have $\psi_u$ is injective.
We write $K_u \rat G$ for the associated inclusion of (unstructured) subgraphs, and likewise $\psi_1(v) = [K_v \rat G]$.

If $K_u$ is an edge (necessarily in the set $\inp(G)$), then since the unique element of $\out(u)$ is a member of the set $\inp(v)$, the function $\bar\psi_0$ gives bijections $\inp(v) \cong \inp(G)$ and $\out(v) = \out(J) \cong \out(G)$, hence $K_v = G$. Likewise, if $K_v$ is an edge then $K_u = G$; in either case $K_u$ and $K_v$ are both in $\ssub(G)$.

We now suppose each of $K_u$ and $K_v$ has a vertex.
If $e\in E_{K_v}$, we now show that there does not exist a vertex $w$ of $K_u$ with $e\in \inp(w)$.
Suppose, to the contrary, that $e\in \inp(w)$ for some $w\in V_{K_u}$.
Then since $e\in E_{K_u}$, it is also an element of
\[
	E_{K_u} \cap E_{K_v} = (\inp (K_u) \cup \out(K_u)) \cap (\inp(K_v) \cup \out(K_v))
\]
since $K_u$ and $K_v$ are vertex disjoint.
But $e\notin \out(K_u)$, hence $e\in \inp(K_u) \subseteq \inp(G)$.
Since $G$ contains a vertex, $e\notin \out(G) \supseteq \out(K_v)$, hence $e\in \inp(K_v)$.
As $K_v$ was assumed to contain a vertex, this tells us $w \in V_{K_v}$, which is impossible since $K_u$ and $K_v$ are vertex disjoint.

As a consequence of the previous paragraph, if $P$ is any path in $G$ from one edge of $K_v$ to another, then every vertex and every edge on this path must also be in $K_v$.
We conclude by \cref{def convex inclusion} that $K_v$ is a structured subgraph of $G$, and a flipped argument shows that $K_u$ is a structured subgraph of $G$ as well.
\end{proof}

\begin{lemma}\label{almost isolated lemma}
Let $H$ be an acyclic directed graph with at least two vertices, let $u\in V_H$ be an almost isolated vertex \cite[Definition 2.60]{HRYbook} and let $J \in \ssub(H)$ be the structured subgraph induced by all of the remaining vertices.
Suppose $G$ is another acyclic directed graph.
If $\varphi \colon H \to G$ is a graphical map and $\hat\varphi(H) \in \ssub(G)$, then $\hat\varphi(\medstar_u)$ and $\hat\varphi(J)$ are structured subgraphs as well.
\end{lemma}
An almost isolated vertex in a graph with two or more vertices is one whose inputs (or outputs) are contained in the inputs (resp.\ outputs) of the graph so that deleting the vertex and all of its inputs (resp.\ outputs) leaves behind a connected graph.
\begin{proof}
Let $\alpha \colon P \ract H$ be the active map in $\gcato$ from an acyclic directed graph with two vertices $u,v$ which acts on vertices by $\alpha_1(v) = J \in \ssub(H)$ and $\alpha_1(u) = \medstar_u \in \ssub(H)$ (that is, $\alpha$ is the graph complement of $J$ from \cref{def graph compl}).
Using the active-inert factorization of $\varphi$, we have the following diagram
\[ \begin{tikzcd}
P \ar[rr,-act, "\alpha"] \ar[drrr, -act,bend right=12, "\psi" swap] & &  H \ar[rr,"\varphi"] \ar[dr,-act,"\beta"] & & G \\
& & & \hat\varphi(H) \ar[ur,tail]
\end{tikzcd} \]
and by \cref{lem partially grafted corollas} we have $\psi_1(u) = K_u$ and $\psi_1(v) = K_v$ in $\ssub(\hat\varphi(H))$.
But $\psi_1(u) = \hat \beta([\iota_u])$ and $\psi_1(v) = \hat \beta([J \rat H])$, so since the composition of convex inclusions are again convex inclusions, we have $K_u, K_v \in \ssub(G)$.
\end{proof}

\begin{lemma}\label{baby restriction lemma}
Suppose $\varphi = (\bar \varphi_0, \hat \varphi) \colon G \to G'$ is a map in $\gcato$, with $G$ and $G'$ acyclic.
The restriction
\[ \begin{tikzcd}
\ssub(G) \rar[hook] \dar[dashed, "\check \varphi = \hat \varphi|_{\ssub(G)}"] & \emb(G) \dar{\hat \varphi} \\
\ssub(G') \rar[hook] & \emb(G')
\end{tikzcd} \]
exists if and only if $\hat \varphi([\id_G]) \in \ssub(G')$.
\end{lemma}
\begin{proof}
Since $G \in \ssub(G)$, we see that the condition is necessary.

Now suppose $\hat \varphi (G)$ is in $\ssub(G')$.
Let $K\in \ssub(G)$.
If $K$ is an edge, then $\hat \varphi(K)$ is an edge as well, hence a structured subgraph. 
We thus assume that $K$ is has a vertex. 
We also assume that $K$ has strictly fewer vertices than $G$, otherwise $K = G$ gets sent to a structured subgraph.

Using the characterization of structured subgraph from \cite[Definition 6.32]{HRYbook}, the inclusion $K \rat G$ admits a presentation as a composition of outer coface maps 
\[
	K = K_0 \rat K_1 \rat \cdots \rat K_m = G
\]
where $m\geq 0$ is the number of elements of $V_G \setminus V_K$.
This means (see Definition 2.60 and \S6.1.2 of \cite{HRYbook}) that $K_{i+1}$ has an almost isolated vertex $v_i$ and $K_i\in \ssub(K_{i+1})$ is induced by all of the remaining vertices.
If $\hat \varphi(K_{i+1})$ is in $\ssub(G')$, then $\hat \varphi(K_i)$ is in $\ssub(G')$ by \cref{almost isolated lemma}. 
But we knew  $\hat \varphi(K_m) = \hat \varphi(G) \in \ssub(G')$, and we conclude that $\hat \varphi(K) = \hat\varphi(K_0)$ is in $\ssub(G')$.
\end{proof}

\begin{proposition}\label{prop restriction}
Suppose $\varphi = (\bar \varphi_0, \hat \varphi) \colon G \to G'$ is a map in $\gcato$, with $G$ and $G'$ acyclic.
If $\hat \varphi (G) \in \ssub(G')$, then the restriction $(\bar\varphi_0, \check \varphi)$  from \cref{baby restriction lemma} is a properadic graphical map.
\end{proposition}
\begin{proof}
By \cref{baby restriction lemma}, we know the restriction exists.
We verify the two conditions from \cref{def properadic gcat}.
Condition \eqref{def properadic gcat bdry} follows from \cref{prop oriented graph cat}\eqref{oriented graph map boundary}.

Suppose $H$ and $K$ are structured subgraphs of $G$ so that $H \cup K \in \ssub(G)$.
Then $\check \varphi(H\cup K) \in \ssub(G')$ is a union of $\check \varphi(H)$ and $\check \varphi(K)$, so 
\[
	\check \varphi(H\cup K) = \check\varphi(H) \cup \check\varphi(K)
\]
by \cref{lem a union vs union of subgraphs}.
Thus $(\bar\varphi_0, \check \varphi)$ satisfies condition \eqref{def properadic gcat union}, hence is a properadic graphical map.
\end{proof}

\begin{theorem}\label{thm properadic gcat}
The properadic graphical category $\pgcat$ may be identified with the subcategory of $\gcato$ with
\begin{enumerate}
\item objects the acyclic directed graphs, and \label{properadic gcat objs}
\item morphisms those $\varphi \colon G \to G'$  so that $\hat \varphi (G) \in \ssub(G')$. \label{properadic gcat mors}
\end{enumerate}
\end{theorem}
\begin{proof}
First observe that \eqref{properadic gcat mors} is closed under composition: $\widehat{\psi\varphi}(G) = \hat \psi(\hat \varphi (G))$ is a structured subgraph by applying \cref{baby restriction lemma} to $\psi$.
We write $\mathbf{M} \subseteq \gcato$ for the indicated subcategory.

By \cref{prop restriction} there is a function $f_{G,G'} \colon \mathbf{M}(G,G') \to \pgcat(G,G')$ given by restriction.
Since composition in each of $\mathbf{M}$ and $\pgcat$ is given by composition of pairs of functions, this constitutes a functor $\mathbf{M} \to \pgcat$.

We also have a function 
\begin{equation}\label{eq function pgcat to gcato}
\begin{aligned} 
\pgcat(G,G') &\to \gcato(G,G') \\
(\bar \varphi_0, \check \varphi) & \mapsto (\varphi_0, \varphi_1) \sim (\bar \varphi_0, \hat \varphi)
\end{aligned}
\end{equation}
from \cref{prop nonfunctor pgcat to gcat} (using \cref{prop oriented graph cat}).
We have isomorphisms
\[ \begin{tikzcd}[row sep=tiny]
\inp (\hat \varphi (G)) & \lar["\cong"',"\bar \varphi_0"] \inp (G) \rar["\cong","\bar \varphi_0"'] & \inp (\check \varphi (G)) \\
\out (\hat \varphi (G)) & \lar["\cong"',"\bar \varphi_0"] \out (G) \rar["\cong","\bar \varphi_0"'] & \out (\check \varphi (G))
\end{tikzcd} \]
so that $\inp (\hat \varphi (G)) = \inp (\check \varphi (G))$ and $\out (\hat \varphi (G)) = \out (\check \varphi (G))$ as subsets of $E_{G'}$.
To see that $\hat \varphi(G) = \check \varphi(G)$ in $\emb(G')$, it suffices by \cref{mod1 prop 1.25} to check that $\hat \varphi(G)$ is an edge if and only if $\check \varphi(G)$ is an edge; but these both occur if and only if $\varphi_1(v)$ is an edge for every $v$.
Thus $\hat \varphi(G) = \check \varphi(G) \in \ssub(G')$, and we conclude that the function \eqref{eq function pgcat to gcato} factors through $d_{G,G'} \colon \pgcat(G,G') \to \mathbf{M}(G,G')$.

Since maps in $\pgcat$ are uniquely determined by their actions on edge sets (\cref{prop forgetful}) and these are preserved by $f$ and $d$, we immediately have $fd$ is the identity on $\pgcat(G,G')$. 
On the other hand, if $df\varphi = \psi$ then $\psi_1$ is given by the left-bottom composite and $\varphi_1$ is given by the right-top composite in the commutative diagram
\[ \begin{tikzcd}
V_G \dar[hook] \ar[dr,hook] \\
\ssub(G) \rar[hook] \dar{\check \varphi = \hat \varphi|_{\ssub(G)}} & \emb(G) \dar{\hat \varphi} \\
\ssub(G') \rar[hook] & \emb(G'),
\end{tikzcd} \]
hence $\varphi_1 = \psi_1$.
Thus $df$ is the identity on $\mathbf{M}(G,G')$, and we conclude that $\pgcat$ is isomorphic to $\mathbf{M}$.
\end{proof}

\begin{corollary}\label{cor pgcatact}
The functor $\pgcat_\actrm \to \gcato^\actrm$ is fully faithful.
\end{corollary}
\begin{proof}
If $G$ is an acyclic directed graph, then the inclusion $\ssub(G) \to \emb(G)$ takes the unique maximal element $G$ to the unique maximal element $[\id_G]$.
Since active maps in both categories are those maps preserving the top element (\cite[Definition 2.2.17]{ChuHackney} and \cref{rmk factorization}), the functor $\pgcat \to \gcato$ restricts.
We already know that this functor is faithful.

If $\varphi \colon G \ract G'$ is an active map in $\gcato$, then $\hat \varphi([\id_G]) = [\id_{G'}] \in \ssub(G')$, hence by \cref{thm properadic gcat} we know $\varphi$ is a map in $\pgcat$.
It is a map in $\pgcat_\actrm$ because it takes $G\in \ssub(G)$ to $G'\in \ssub(G')$.
\end{proof}

\begin{remark}
Given a graphical category, we expect the subcategory of active maps to be closely related to a corresponding operadic category \cite{BataninMarkl:OCDDC} (see Proposition 3.2 of \cite{Berger:MCO}).
If an operadic category for properads exists, the previous corollary suggests that it should simply be a full subcategory of the operadic category for wheeled properads from \cite{BataninMarkl:OCNEKD,BataninMarklObradovic:MMGROA}.
\end{remark}

Recall that there is a full subcategory $\pgcatsc \subset \pgcat$ (equivalent to the category called $\mathbf{\Theta}$ in \cite[\S6.3.5]{HRYbook}) whose objects are trees when considered as undirected graphs, that is, those graphs mapping to objects in the subcategory $\gcatnought \subset \gcat$ from \cref{def paths and trees}.
For such a directed tree $T$, every embedding with codomain $T$ is a convex inclusion.
In particular, we have $\ssub(T) = \emb(T)$ is the set of subtrees of $T$ in the sense of \cref{def subtrees} (forgetting direction).

\begin{corollary}\label{cor sc equiv}
The induced functor $\pgcatsc \to (\gcatnought)_{/\opshf}$ is an equivalence.
\end{corollary}
\begin{proof}
If $S$ and $T$ are directed trees, then since $\ssub(T) = \emb(T)$, every map in $\gcato(S,T)$ satisfies condition \eqref{properadic gcat mors} of \cref{thm properadic gcat}. 
Hence
\[
\pgcatsc(S,T) = \pgcat(S,T) \to \gcato(S,T) = (\gcatnought)_{/\opshf}(S,T)
\]
is a bijection.
\end{proof}

We can use this corollary to establish the following fact about undirected trees.

\begin{proposition}\label{prop redundancy}
\Cref{thm on trees}(\ref{new graph subtrees}.\ref{v intersect}) is redundant.
That is, suppose $G$ and $G'$ are undirected trees and $\varphi = (\varphi_0,\hat \varphi)$ is a pair consisting of an involutive function $\varphi_0 \colon A_G \to A_{G'}$ and a function $\hat \varphi \colon \emb(G) \to \emb(G')$.
Then $\varphi$ is a new graph map (that is, a map in $\gcatnought$) if and only if it satisfies the following conditions:
\rm
\begin{enumerate}[label=(\roman*),ref=\roman*, start=4]
\item (\cref{def new graph map}) The diagram
\[ \begin{tikzcd}
\emb(G) \rar{\eth} \dar{\hat \varphi}& \mathbb{N}A_G \dar{\mathbb{N} \varphi_0 }\\
\emb(G') \rar{\eth} & \mathbb{N}A_{G'}
\end{tikzcd} \]
commutes. \label{new graph def boundary again}
	\item (\cref {thm on trees}) If two subtrees $S,T$ of $G$ overlap, then so do $\hat \varphi (S)$ and $\hat \varphi (T)$.
	In this case, we have \label{new graph subtrees again}
	\begin{enumerate}[label=(\alph*),ref=\alph*, start=2] 
		\item $\hat \varphi(S\cup T) = \hat\varphi(S) \cup \hat \varphi(T)$. \label{v union again}
	\end{enumerate}
\end{enumerate}
\end{proposition}
\begin{proof}
We only need to prove the reverse implication. 
Suppose that $(\varphi_0, \hat \varphi) \colon G \to G'$ satisfies \eqref{new graph def boundary again} and \eqref{new graph subtrees again} just above. 
Choose an arbitrary element $y \in \opshf_{G'}$ (where $\opshf$ is the orientation presheaf from \cref{def or presheaf}), and define the element $x\in \opshf_G$ by $x_a \coloneqq y_{\varphi_0(a)}$.
Define, as usual, $\bar \varphi_0 \colon E_G \to E_{G'}$ to be the map obtained from $\varphi_0$ by passing to $\dagger$-orbits.
We claim that 
\[ (\bar \varphi_0, \hat \varphi) \colon (G,x) \to (G',y)\]
is a map in $\pgcat$, hence in $\pgcatsc$.
Once this is established, we will be done, since $(\bar \varphi_0, \hat \varphi)$ maps to  $(\varphi_0, \hat \varphi) \colon (G,x) \to (G',y)$ under the equivalence $\pgcatsc \to (\gcatnought)_{/\opshf}$ from \cref{cor sc equiv}, hence $(\varphi_0, \hat \varphi)$ must be a map in $\gcatnought$.

Using the chosen orientations, the commutative diagram from \eqref{new graph def boundary again} splits into the middle two squares of
\[ \begin{tikzcd}
\mathbb{N}E_G \dar{\mathbb{N} \bar \varphi_0 } & 
	\mathbb{N}A_G^+  \lar[swap]{\cong} \dar{\mathbb{N} \varphi_0}& 
	\emb(G) \rar{\out}\lar[swap]{\inp} \dar{\hat \varphi} & 
	\mathbb{N}A_G^- \rar{\cong} \dar{\mathbb{N} \varphi_0}& 
	\mathbb{N}E_G \dar{\mathbb{N} \bar \varphi_0 }\\
\mathbb{N}E_{G'} & 
	\mathbb{N}A_{G'}^+ \lar[swap]{\cong} & 
	\emb(G') \rar{\out}\lar[swap]{\inp} & 
	\mathbb{N}A_{G'}^- \rar{\cong} & 
	\mathbb{N}E_{G'}
\end{tikzcd} \]
using the conventions of \cref{def inout emb}.
Hence \cref{def properadic gcat}\eqref{def properadic gcat bdry} holds.

On the other hand, we have that subtrees of $G$, structured subgraphs of $G$, and embeddings with codomain $G$ all coincide.
Further, the union of two subtrees $S$ and $T$ in the sense of \cref{remark overlap} is the same as the union of structured subgraphs in the sense of \cref{rem unions}, and this union is a structured subgraph if and only if $S$ and $T$ overlap.
Thus \eqref{new graph def boundary again} implies \cref{def properadic gcat}\eqref{def properadic gcat union}, and we conclude that $\varphi$ is a morphism in $\pgcatsc$, as desired.
\end{proof}

\begin{remark}
It would be interesting to know whether or not the corresponding condition (2a) about intersections from \cite[Definition 1.12]{HRY-cyclic} is superfluous.
\end{remark}

\section{Nodeless loops and the extended graphical category}\label{sec extended}
Whenever we talk about the graphical category $\gcat$ or the oriented graphical category $\gcato$, we would like to include the nodeless loop in our category. 
This is a graph without any vertices whose geometric realization is a circle.
We hope for this graph to show up, as it is used to represent the self-gluing of an identity element in a modular operad or a wheeled properad. 
(Note, however, that Raynor showed in \cite[\S7]{Raynor:DLCSM} that it is possible to give a monadic definition of these objects which does not require the nodeless loop to be regarded as a graph.)
The reason we have mostly avoided the issue up until now is that \cref{def jk graphs,def dir graph} do not include nodeless loops, and in alternative combinatorial models for graphs (such as those found in \cite{BataninBerger:HTAPM} and \cite{YauJohnson:FPAM}) the notion of `\'etale map,' and hence `embedding,' is not a first-class concept.
One can modify the earlier definitions of graph to handle nodeless loops, but they become a good deal less elegant and not as easy to work with, which is why we did not do so from the start.
Thus we are in the unfortunate state of affairs where there is no completely suitable combinatorial definition of graph with loose ends for our purposes.

Nevertheless, in this section we treat the extended graphical category $\egcat$ and show that \cref{def new graph map} also describes maps in this category, whereas in \cite{HRY-mod1} we did not have a uniform definition for $\gcat$ and $\egcat$ (compare \cref{def graphical map} and \cref{def extended graphical cat}).
We emphasize that we are not out to describe a different operadic structure here: the nerve theorem of \cite{HRY-mod2} holds equally well whether we have the nodeless loop or not, so $\gcat$ and $\egcat$ both describe modular operads (though perhaps they do not describe the same notion of $\infty$-modular operads).

In Definition 4.1 of \cite{HRY-mod1}, the following refinement of the notion of graph from \cref{def jk graphs} was presented that included explicit boundary data; that is, instead of having $\eth(G) = A_G \setminus D_G$, it should merely be a subset satisfying certain properties.
\begin{definition}[Extension of undirected graphs]\label{refined graph}
A \mydef{graph} $G$ consists of the data from \cref{def jk graphs} together with a subset $\eth(G) \subseteq A_G$ satisfying
\begin{itemize}
	\item $D_G^\dagger \setminus D_G \subseteq \eth(G) \subseteq A_G \setminus D_G$, and 
	\item $\eth(G) \setminus D_G^\dagger$ is a $\dagger$-closed subset of $A_G$.
\end{itemize}
\end{definition}
This allows one to encode graphs that have nodeless loops as connected components, but only adds (up to isomorphism) a single new \emph{connected} graph:
\begin{example}[{\cite[Definition 4.2]{HRY-mod1}}]\label{ex undirected nodeless loop}
A graph $K$ is a \mydef{nodeless loop} if it has two arcs, an empty vertex set, and an empty boundary $\eth(K) = \varnothing$.
Such a graph $K$ has exactly one internal edge, and $K$ is not isomorphic to $\exedge$.
\end{example}

Rather than repeat the general extended definition of embedding from Definition~4.6 of \cite{HRY-mod1}, we will simply import the properties of nodeless loops that we need:

\begin{proposition}\label{facts about nodeless loop}
Suppose $K$ is a nodeless loop.
\begin{enumerate}
\item 
If $f \colon K \rat G$ is an embedding, then $f$ is an isomorphism (in particular, $G$ is a nodeless loop as well).
\item 
If $h\colon H \rat K$ is an embedding, then $H$ is either isomorphic to the edge $\exedge$ or $H$ is a nodeless loop.
\item
The set $\emb(K)$ consists of exactly two elements, $[\exedge] < [\id_K]$, and the boundaries of these are $\eth([\exedge]) = A_K$ and $\eth([\id_K]) = \varnothing$. \qed
\end{enumerate}
\end{proposition}

We now recall the extension, from Definition 4.7 of \cite{HRY-mod1}, to \cref{def graphical map} to include the nodeless loops.

\begin{definition}[Extended graphical category]\label{def extended graphical cat}
Let $G$ and $G'$ be connected (undirected) graphs, including the possibility that one or both is a nodeless loop.
A \mydef{graphical map} $\varphi \colon G \to G'$ is a pair $(\varphi_0, \varphi_1)$ consisting of an involutive function $\varphi_0 \colon A_G \to A_{G'}$ and a function $\varphi_1 \colon V_G \to \emb(G')$, so that \eqref{old graph def vertices} and \eqref{old graph def boundary} of \cref{def graphical map} hold, as well as 
\begin{enumerate}[label=(\roman*'),ref=\roman*',start=3]
\item If the boundary of $G$ is empty and $\varphi_1(v)$ is an edge for every $v$, then $G'$ is a nodeless loop.\label{extended graph def collapse}
\end{enumerate}
These maps assemble into the \mydef{extended graphical category}, denoted $\egcat$.
\end{definition}

\begin{remark}[Maps involving nodeless loops]\label{rmk nodeless maps}
There are few maps in $\egcat$ involving a nodeless loop $K$.
Indeed, there is a non-trivial automorphism $K \to K$ which interchanges the two arcs, and any map $K \to G$ with domain a nodeless loop is an isomorphism between nodeless loops. 
There are more maps with codomain $K$.
If $G$ has a single vertex and no arcs (that is, $G$ is isomorphic to $\medstar_0$), then there is a unique map $G \to K$, and it sends the vertex to $[\id_K] \in \emb(K)$. 
If all vertices of $G$ have valence two, then there are exactly two maps $\varphi \colon G \to K$ --
\cref{def graphical map}\eqref{old graph def boundary} forces $\varphi_1(v) = [\exedge]$ for each vertex $v$, and also shows that $\nbhd(v) \hookrightarrow A_G \to A_K$ is a bijection of sets.
As $\varphi_0$ is an involutive function, $A_G\to A_K$ is determined by the choice of where to send a single arc.
There are no other maps with codomain $K$, see \cite[Remark 4.8]{HRY-mod1}.
\end{remark}

When $K$ is a nodeless loop, the notion of a union from \cref{def a union} and the notion of vertex disjoint from \cref{def vertex disjoint} extend immediately to $\emb(K)$.
Thus the new graph maps from \cref{def new graph map} still make sense when one of $G$ or $G'$ is a nodeless loop.

\begin{theorem}\label{thm extended graph cat}
If $G$ and $G'$ are connected graphs (including the possibility of a nodeless loop), then new graph maps $\varphi \colon G \to G'$ from \cref{def new graph map} are in bijection with the maps $G \to G'$ in the extended graphical category.
This identifies the category of graphs (including the nodeless loops) and new graph maps with the extended graphical category.
\end{theorem}
\begin{proof}
In light of \cref{old new equivalence}, we need only consider maps involving a nodeless loop $K$, and to establish the bijections it is enough to enumerate such maps.
We then must show that any compositions $\varphi \circ \psi$ with a nodeless loop appearing as a domain or codomain of $\varphi$ or $\psi$ behave the same on both sides.

Suppose $\varphi \colon K \to G'$ is a new graphical map with $K$ a nodeless loop.
By \cref{def new graph map}\eqref{new graph def boundary}, $\hat \varphi ([\id_K])$ has empty boundary, so by \cref{lem boundary inclusion} we have that the boundary of $G'$ is empty and $\hat \varphi([\id_K]) = [\id_{G'}]$.
On the other hand, $[\id_K]$ is a union of $[\exedge]$ and $[\exedge]$, so by conditions \eqref{new graph def edges} and  \eqref{new graph def union} of \cref{def new graph map} we have $[\id_{G'}]$ is the union of edges, hence $G'$ does not have any vertices.
The only connected graphs with no vertices and an empty boundary are the nodeless loops.

Now that we've established $G'$ is a nodeless loop and $\hat \varphi$ is an isomorphism, we conclude that the only data in the new graph map is the involutive function $\varphi_0 \colon A_K \to A_{G'}$, and there are exactly two such maps.
By \cref{rmk nodeless maps} we have established a bijection between new graph maps and extended graphical maps from $K$ to $G'$.

Now suppose $G \to K$ is a new graph map.
As the boundaries of elements of $\emb(K)$ have cardinality zero and two, \cref{def new graph map}\eqref{new graph def boundary} implies that any vertices of $G$ must have valence zero or two.
In the first case, there is a unique map $G \to K$. Indeed, if $G$ has a vertex of valence zero, then it is isomorphic to $\medstar_0$. The set $\emb(\medstar_0) \cong \emb(G)$ is a singleton and preservation of boundary means that $\hat \varphi[\id_G]$ must be $[\id_K]$.

If every vertex of $G$ has valence two, then preservation of boundary implies that $\nbhd(v) \to A_K$ is a bijection for every $v \in V_G$, so the map $\varphi_0 \colon A_G \to A_K$ is determined by where it sends any individual arc (as in \cref{rmk nodeless maps}).
There is a unique function $\hat \varphi \colon \emb(G) \to \emb(K)$ satisfying condition \eqref{new graph def boundary} of \cref{def new graph map}, namely 
\[
	\hat \varphi ([h]) = \begin{cases}
		[\exedge] & \text{if } \eth([h]) \neq \varnothing \\
		[\id_K] & \text{if } \eth([h]) = \varnothing.
	\end{cases}
\]
Thus there are precisely two maps $G \to K$ in this second case.

By \cref{rmk nodeless maps} we have established the desired bijections for maps having the nodeless loop as its domain or codomain.
It remains to show that compositions involving such a map behave the same whether considered as new graph maps or maps in the extended graphical category. 
But all maps involving the nodeless loop are completely determined by $\varphi_0$, and the correspondence does not change the arc map.
\end{proof}

\subsection{The extended oriented graphical category}\label{subsec extended oriented}
We now return to the oriented graphical category from \cref{sec directed}.
The original version of the wheeled properadic graphical category from \cite{HRYbook} included a (directed) nodeless loop.
To cut down on enumerating multiple special cases, and in order to use the efficient \cref{def dir graph}, we previously avoided the nodeless loop. We now add it back in.

The following is an extension of \cref{def dir graph}, and is the directed version (and a simple translation using \cref{cnst gcato to kock,rmk inp out plus minus}) of \cref{refined graph}.
\begin{definition}[Extension of directed graphs]\label{def ext dir graph}
A \mydef{directed graph} consists of a diagram of finite sets
\[
\begin{tikzcd}[column sep=small]
	E_G & I_G \rar\lar[hook'] & V_G & O_G \rar[hook] \lar & E_G	
\end{tikzcd}
\]
along with a pair of subsets $\inp(G)$ and $\out(G)$ of $E_G$ so that 
\begin{itemize}
\item $I_G \setminus O_G \subseteq \inp(G) \subseteq E_G \setminus O_G$,
\item $O_G \setminus I_G \subseteq \out(G) \subseteq E_G \setminus I_G$, and 
\item $\inp(G) \setminus I_G = \out(G) \setminus O_G$.
\end{itemize}
\end{definition}

Notice that the graphs from \cref{def dir graph} are graphs in this sense.
Indeed, if $\inp(G) = E_G \setminus O_G$ and $\out(G) = E_G \setminus I_G$,  then \[ \inp(G) \setminus I_G = E_G \setminus (O_G \cup I_G) = \out(G) \setminus O_G.\] 
Below, we will write $L_G\subseteq E_G$ for the complement of $\inp(G) \cup \out(G) \cup I_G \cup O_G$.

\begin{remark}\label{remark no loops}
For a graph $G$, the following are equivalent: 
$L_G =\varnothing$, 
$\inp(G) = E_G\setminus O_G$, and
$\out(G) = E_G \setminus I_G$.
In particular, the graphs of \cref{def dir graph} are precisely those graphs with $L_G = \varnothing$.
We will indicate the equivalence of the first two, as the equivalence of the first and third is identical.
If 
$\inp(G) = E_G \setminus O_G$ then
\[
	E_G = \inp(G) \cup O_G \subseteq \inp(G) \cup \out(G) \cup I_G \cup O_G = E_G \setminus L_G
\]
and we conclude that $L_G$ is empty.
On the other hand, suppose that $L_G = \varnothing$.
Then we have \[ E_G \setminus O_G = (\inp(G) \cup \out(G) \cup I_G) \setminus O_G, \]
but $\out(G) \setminus O_G = \inp(G) \setminus I_G \subseteq \inp(G)$ and $I_G \setminus O_G \subseteq \inp(G)$ by assumption, hence $E_G \setminus O_G \subseteq \inp(G) \subseteq E_G \setminus O_G$.
\end{remark}

This more expansive definition of graph adds only a single new connected graph.
\begin{example}\label{ex dir nodeless loop}
The \mydef{directed nodeless loop} is the directed graph with $E_G = \ast$ and \[ I_G = V_G = O_G = \inp(G) = \out(G) = \varnothing.\]
\end{example}

The following is an extension of \cref{def dir embedding}, and is the directed analogue of Definition 4.6 of \cite{HRY-mod1}.

\begin{definition}[Embeddings]\label{def extended etale embedding}
Suppose that $G$ and $H$ are directed graphs in the sense of \cref{def ext dir graph}.
A \mydef{p-map} from $G$ to $H$ is a natural transformation of underlying functors $\mathscr{G} \to \finset$ so that 
\begin{enumerate}
	\item  
	the two squares
\[
\begin{tikzcd}[column sep=small]
I_G \rar  \dar \ar[dr, phantom, "\lrcorner" very near start] & V_G  \dar& O_G  \lar  \dar \ar[dl, phantom, "\llcorner" very near start] \\
I_H \rar & V_H & O_H \lar 	
\end{tikzcd}
\]
	are pullbacks, 
	and \label{item p-map}
	\item
	if $L_G \subseteq E_G$ is the complement of $\inp(G) \cup \out(G) \cup I_G \cup O_G$, then $L_G$ maps into $L_H$.
\end{enumerate} 
A p-map between connected graphs is an \mydef{embedding} just when $V_G \to V_H$ is injective.	
\end{definition}

\begin{remark}[Terminology]
Notice that if $L_G = \varnothing = L_H$ (see \cref{remark no loops}), then a p-map is the same thing as an \'etale map from \cref{def dir embedding}.
Further, the undirected analogue of p-map was called \emph{\'etale} in \cite[Definition 4.6]{HRY-mod1}, but we are deliberately avoiding that terminology here. 
This is because a p-map (or an \'etale map from \cite{HRY-mod1}) does not have much to do with either the geometric or algebraic situations.
From the geometric perspective, we would expect \'etale maps to correspond to deformation classes of oriented local homeomorphisms of the associated topological graphs (see \cref{rem etale local homeo}), which implies that the nodeless loop should have a self-\'etale map of degree $n$ for each $n\in \mathbb{N}$, corresponding to the degree $n$ map of the circle $S^1$.
Algebraically, elements of the free wheeled properad generated by $H$ will be represented by natural transformations (with connected domain) satisfying only \eqref{item p-map}.
We are unconvinced that the notion of p-map is widely useful on its own, except in the cases covered by \cref{def dir embedding}.
On the other hand, the notion of embedding is still geometrically meaningful, as there is essentially only one injective local homeomorphism into the circle from either the open interval or the circle. 
\end{remark}

The orientation $\gcat$-presheaf $\opshf$ from \cref{def or presheaf} can be extended in a natural way to a $\egcat$-presheaf (also called $\opshf$).
That is, if $K$ is a nodeless loop with arc set $\{ a, a^\dagger \}$, then 
\[ \opshf_K = \big\{ (x_a, x_{a^\dagger}) \mid x_a \in \{+1,-1\}, x_{a^\dagger} = -x_a \big\} \subseteq \prod_{A_K} \{+1,-1\} \] 
contains two elements.

\begin{proposition}\label{prop directed graph iso classes}
Isomorphism classes of directed graphs in the sense of \cref{def ext dir graph} are in bijective correspondence with isomorphism classes of
\begin{itemize}
\item maps of $\egcat$-presheaves of the form $\coprod_{i=1}^n G_i \to \opshf$ (with each $G_i$ a representable presheaf associated to an undirected connected graph),
\item directed Yau--Johnson graphs \cite[Definition 1.21]{YauJohnson:FPAM}, and
\item 
directed Batanin--Berger graphs \cite[13.4]{BataninBerger:HTAPM}.
\end{itemize} 
\end{proposition}
\begin{proof}
This combines versions of \cite[1.1.13]{Kock:GHP} and \cref{cnst gcato to kock} for the more general class of graphs, and the equivalences Proposition 15.6 and (a variation of) Proposition 15.2 from \cite{BataninBerger:HTAPM}.
\end{proof}

\begin{definition}\label{def extended oriented}
The \mydef{extended oriented graphical category}, denoted $\egcato$, is the category whose objects are morphisms $G \to \opshf$ from a representable presheaf to the orientation presheaf.
\end{definition}

\begin{theorem}\label{thm oriented wheeled equiv extended}
The extended oriented graphical category $\egcato$ is equivalent to the wheeled properadic graphical category which includes the nodeless loop.
Morphisms may be described exactly as in \cref{prop oriented graph cat}.
\end{theorem}
This version of the wheeled properadic graphical category was called $\mathcal{A}$ in \cite[\S2]{HRYfactorizations}.
Our strategy below is similar to that from \cref{thm extended graph cat}, namely to enumerate all maps involving nodeless loops.
\begin{proof}
First, the characterization of morphisms in the extended oriented graphical category follows just like in \cref{prop oriented graph cat}.

Let $\widetilde{\mathbf{C}}$ denote the wheeled properadic category which includes nodeless loops from \cite[\S2]{HRYfactorizations}, and let $\mathbf{C}$ be the full subcategory which excludes them (using the conventions of \cref{remark no listings}; note that nodeless loops have a unique listing).
By sending a nodeless loop $K$ to the $E_K$-colored wheeled properad having just an identity morphism and its contraction, we obtain the extension to the functor from \cref{prop functor to wproperad} below left.
\begin{equation}\label{eq extending isomorphism}
\begin{tikzcd}[column sep=small]
\gcato \ar[rr,"\simeq"] \dar[hook] & & \mathbf{C} \dar[hook] \\
\egcato \ar[dr] \ar[rr,dashed] & & \widetilde{\mathbf{C}} \ar[dl] \\
& \wproperad
\end{tikzcd} \end{equation}
To prove existence of the dashed equivalence, it suffices to compare hom-sets.
We already have everything we need in \cref{rmk nodeless maps}. 
In $\egcato$ there is a unique map between two nodeless loops, and no other maps with a nodeless loop as the domain. 
Likewise, if $K$ is a nodeless loop then there is a map $G \to K$ if and only if every vertex of $G$ has one input and one output or the unique vertex of $G$ has valence zero.
If there is a map $G\to K$, it is unique.

On the other hand, in the second paragraph of page 220 of \cite{HRYfactorizations} it was observed that only isomorphisms in $\widetilde{\mathbf{C}}$ can have a nodeless loop as their domain, and a nodeless loop possesses a single automorphism (the identity).
If $f \colon G \to K$ is a map whose codomain is a nodeless loop, we can form the (essentially unique) Reedy factorization guaranteed by \cite[Theorem 1.2]{HRYfactorizations}.
\[ \begin{tikzcd}[column sep=small]
G \ar[rr,"f"] \ar[dr,"f^-"'] & & K \\
& H \ar[ur,"f^+"'] &
\end{tikzcd} \]
By the characterization of plus maps in Proposition 3.3 of \cite{HRYfactorizations}, every vertex of $H$ maps to a subgraph which is not an edge.
Thus either $H$ has a single vertex of valence zero (that is, $H\cong \medstar_0$) mapping the vertex to $K$ or $H$ does not have any vertices.
In the latter case, $H$ is either an edge or a nodeless loop.
In any of these situations, the map $f^+ \colon H \to K$ is unique.

We now show that $f^-$ is unique as well.
For each of these three graphs $H$, if $f^- \colon G \to H$ is an isomorphism then by \cite[Lemma 3.9]{HRYfactorizations} it is the unique such.
This finishes the case when $H\cong \medstar_0$, as the only maps with this graph as their codomain are isomorphisms.
Now when $H$ is an edge or a nodeless loop, there is at most one minus map $G \to H$ (that is, isomorphic to a composition of codegeneracy maps \cite[Definition 9.39]{HRYbook}), which sends every vertex to the edge subgraph and preserves boundaries. 
Thus the map $f^-$ is unique, and hence if there is a map $G \to K$ it is unique as well.
Further, by our analysis of $f^-$, every vertex of $G$ has either precisely one input and one output, or $G$ has a unique vertex of valence zero.

We have now established (unique) bijections between the sets of maps $G\to K$ or $K\to G$ in both categories whenever $K$ is a nodeless loop. 
Since the vertical inclusions in \eqref{eq extending isomorphism} are both sieves, we conclude that we can extend the identity-on-objects equivalence $\gcato \simeq \mathbf{C}$ to the dashed functor, which is again an identity-on-objects equivalence.
\end{proof}

If $F$ is a functor whose domain category has an orthogonal factorization system, and $d$ is any object of the codomain, then there is a canonical orthogonal factorization system on $F\downarrow d$ (or $d\downarrow F$), generalizing the usual fact that a factorization system on a category induces one on any (co)slice. 
We thus obtain the following from Theorem 4.9 of \cite{HRY-mod1}, which lifts the factorization system from \cref{prop U ofs} to the extended graphical category.
See also the related result \cite[Proposition 9.75]{HRYbook}.
\begin{corollary}
The wheeled properadic graphical category has an active-inert orthogonal factorization system. \qed
\end{corollary}

\renewcommand{\MR}[1]{}


\bibliographystyle{amsalpha}
\bibliography{modular}
\end{document}